\documentclass[12pt]{article}

\title{Hochschild Cohomology, Modular Tensor Categories, and Mapping Class Groups}

\author{S.~Lentner, S.~Mierach, C.~Schweigert, Y.~Sommerh\"auser}

\date{}

\usepackage{geometry}

\usepackage{amsmath}
\usepackage{amsfonts}
\usepackage{amsthm}
\usepackage{amssymb}
\usepackage{graphicx}
\usepackage{extarrows}
\usepackage{enumitem}
\usepackage{bbm}

\usepackage{mathtools}
\usepackage[all]{xy}
\usepackage{ifthen}
\usepackage{tikz}
\usepackage{tikz-cd}
\usepackage{tikzscale}
\usetikzlibrary{shapes.geometric}
\usetikzlibrary{decorations, decorations.markings} 
\usetikzlibrary{arrows, arrows.meta}
\usetikzlibrary{matrix, arrows, decorations.pathmorphing}
\usetikzlibrary{calc}

\tikzset{ ->-/.style args={#1 #2 #3}{
decoration={markings,mark= at position 0.5 with {\arrow{stealth}}, }, 
postaction={decorate}, }, 
->-/.default= {0.5 6pt black }}

\tikzset{ -<-/.style args={#1 #2 #3}{
decoration={markings,mark= at position 0.5 with {\arrow[>=stealth]{<}}, }, 
postaction={decorate}, }, 
-<-/.default= {0.5 6pt black }}

\tikzset{pics/.cd,
handle/.style={code={
\draw[fill=gray!10]  (-2,0) coordinate (-left) 
to [out=260, in=60] (-3,-2) 
to [out=240, in=110] (-3,-4) coordinate (-k)
to [out=290,in=180] (0,-6) coordinate (-num)
to [out=0,in=250] (3,-4) 
to [out=70,in=300] (3,-2) 
to [out=120,in=280] (2,0)  coordinate (-right);
\pgfgettransformentries{\tmpa}{\tmpb}{\tmp}{\tmp}{\tmp}{\tmp}
\pgfmathsetmacro{\myrot}{-atan2(\tmpb,\tmpa)}
\draw[rotate around={\myrot:(0,-2.5)}] (-1.2,-2.4) to[bend right]  (1.2,-2.4);
\draw[fill=white,rotate around={\myrot:(0,-2.5)}] (-1,-2.5) to[bend right]coordinate[pos=0.5] (-B) coordinate[pos=1] (-D) (1,-2.5) 
to[bend right] coordinate [pos=1] (-E) coordinate[pos=0.7] (-A) coordinate[pos=0.4] (-C) (-1,-2.5);
}}}

\tikzset{pics/.cd,
rightpart/.style={code={
\fill[gray!10]  
(-0.4,-4.5) coordinate(-a)
to [out=150,in=20] (-1.1,-4.5)
to [out=200,in=140](-2.2,-5.5) to [out=320,in=185] (0,-6.2)
to [out=5,in=250] (3,-4) 
to [out=70,in=300] (3,-2) 
to [out=120,in=280] (2,0)  coordinate (-right);
\draw[fill=gray!10]  
(-2,0) coordinate (-left) 
to [out=260, in=60] (-3,-2) 
to [out=240, in=110] (-3,-4) coordinate(-k);
\draw[violet]  
(-0.4,-4.5) coordinate(-a)
to [out=150,in=20] (-1.1,-4.5)
to [out=200,in=140] (-2.2,-5.5) coordinate(-b);
\draw[]
(-2.2,-5.5) to [out=320,in=185] (0,-6.2)
to [out=5,in=250] (3,-4) 
to [out=70,in=300] (3,-2) 
to [out=120,in=280] (2,0)  coordinate (-right);

}}}

\makeatletter
\def\enumfix{%
\if@inlabel
 \noindent \par\nobreak\vskip-\topsep\hrule\@height\z@
\fi}
\let\oldenumerate\enumerate
\def\enumerate{\enumfix\oldenumerate}
\makeatother

\newlist{parlist}{enumerate}{1}
\setlist[parlist]{label=(\arabic*),wide=0pt,topsep=0pt}

\theoremstyle{definition}
\newtheorem{Satz}{Satz}[section]
\newtheorem{Lemma}[Satz]{Lemma}
\newtheorem{Proposition}[Satz]{Proposition}
\newtheorem{Cor}[Satz]{Corollary}
\newtheorem{Theorem}[Satz]{Theorem}
\newtheorem{Definition}[Satz]{Definition}

\newtheoremstyle{break}{}{}{}{}{\bfseries}{.}{\newline}{}

\theoremstyle{break}

\makeatletter
\renewenvironment{proof}[1][\proofname]{\par
\vspace{-5pt}
\pushQED{\qed}%
\normalfont
\topsep0pt \partopsep0pt 
\trivlist
\item[\hskip\labelsep
\bfseries
#1\@addpunct{.}]\ignorespaces
}{%
\popQED\endtrivlist\@endpefalse
\addvspace{6pt plus 6pt} 
}
\makeatother

\newcommand{\Z}{\mathbb{Z}}  
 
\newcommand{\R}{\mathbb{R}}  
\newcommand{\C}{\mathbb{C}} 
\newcommand{\K}{K}

\newcommand{\CC}{\mathcal{C}}   
   
\newcommand{\V}{\mathcal{V}}   
\newcommand{\DD}{\mathcal{D}}

\DeclareMathOperator\Ker{Ker}
\DeclareMathOperator\Hom{Hom}
\DeclareMathOperator\End{End}
\DeclareMathOperator\Img{Im}
\DeclareMathOperator\Aut{Aut}
\DeclareMathOperator\Out{Out}
\DeclareMathOperator\ad{ad}

\DeclareMathOperator\id{id}

\DeclareMathOperator\op{op}

\DeclareMathOperator\ev{ev}
\DeclareMathOperator\coev{coev}

\DeclareMathOperator\Diffeo{Diffeo}

\DeclareMathOperator\Rex{Rex}

\newcommand{\GL}{\mathrm{GL}}
\newcommand{\SL}{\mathrm{SL}}
\newcommand{\Sp}{\mathrm{Sp}}
\newcommand{\PGL}{\mathrm{PGL}}

\newcommand{\fS}{\mathfrak{S}}
\newcommand{\fT}{\mathfrak{T}}

\newcommand{\fN}{\mathfrak{N}}

\newcommand{\fr}{\mathfrak{r}}
\newcommand{\fs}{\mathfrak{s}}
\newcommand{\ft}{\mathfrak{t}}
\newcommand{\fn}{\mathfrak{n}}
\newcommand{\fd}{\mathfrak{d}}
\newcommand{\fq}{\mathfrak{q}}
\newcommand{\fb}{\mathfrak{b}}
\newcommand{\fz}{\mathfrak{z}}
\newcommand{\hS}{\hat{\mathfrak{S}}}
\newcommand{\hT}{\hat{\mathfrak{T}}}
\newcommand{\Pu}{\mathrm{P}}

\newcommand{\Ext}{\mathrm{Ext}}

\renewcommand{\1}{\mathbbm{1}}
\newcommand{\TFT}{\mathcal{Z}}
 
\newcommand{\ot}{\mathbin{\otimes}}
\newcommand\deq{\vcentcolon =}

\begin{document}

\begin{flushright}
{\sf ZMP-HH/20-5}\\
{\sf Hamburger$\;$Beitr\"age$\;$zur$\;$Mathematik$\;$Nr.$\;$825}
\end{flushright}
\vskip 11mm

\begin{center}{\bf \Large 
Hochschild Cohomology, Modular Tensor Categories, and Mapping Class Groups}

\large{\bf I.~General Theory}

\vskip 8mm

{\large Simon Lentner, Svea Nora Mierach,\\  \vskip 1mm Christoph Schweigert, Yorck Sommerh\"auser}
\vskip 9mm
\end{center}

\vspace{15mm}

\begin{abstract}
\noindent 
Given a finite modular tensor category, we associate with each compact surface with boundary a cochain complex in such a way that the mapping class group of the surface acts projectively on its cohomology groups. In degree zero, this action coincides with the known projective action of the mapping class group on the space of chiral conformal blocks. In the case that the surface is a torus and the category is the representation category of a factorizable ribbon Hopf algebra, we recover our previous result on the projective action of the modular group on the Hochschild cohomology groups of the Hopf algebra.
\end{abstract}

\setcounter{tocdepth}{2}
\newpage
\tableofcontents

\newpage
\section*{Introduction}
\addcontentsline{toc}{section}{Introduction}
It is by now well-understood that a semisimple modular tensor category gives rise to a topological field theory. In particular, such a topological field theory assigns finite-dimensional vector spaces to surfaces, possibly with boundary. These vector spaces are constructed from the homomorphism spaces of the category and are called the spaces of chiral conformal blocks, or briefly the block spaces. They depend functorially on the objects of the category assigned to the boundary components, so that it is more appropriate to say that a topological field theory assigns to a surface not only a space, but rather a functor. To diffeomorphisms of surfaces, the topological field theory assigns natural transformations between these functors, which gives rise to projective representations of mapping class groups. These data obey factorization constraints related to the cutting and gluing of surfaces, properties formalized in the notion of a modular functor.

Already early in the development of the theory, it was noted that such projective representations can be obtained even if the modular category is not semisimple. In this case, the arising functors are no longer exact, but they are still left exact, so that it is natural to study their derived functors. The main result of the present work is that the cohomology groups arising from these derived functors still carry a projective action of the mapping class group in such a way that the original action is recovered in degree zero. It is therefore appropriate to call these spaces derived block spaces.

This result can be applied to a special case: For the category, one can use the representation category of a factorizable ribbon Hopf algebra, and for the surface, one can use the torus. In this case, the mapping class group is isomorphic to the modular group, and the cohomology groups become the Hochschild cohomology groups of the Hopf algebra. In this way, we recover the main result of our previous article~\cite{LMSS1}.

The approach used to reach our present result is inspired by the principle of `propagation of vacua' (cf.~\cite[Par.~2.2, p.~476]{TUY}). This technique introduces an additional boundary component on the surface that, if labeled with the monoidal unit of the category, leads to block spaces that are canonically isomorphic to the block spaces without the additional boundary component. In our construction, the new boundary component serves as the position where we insert a projective resolution of the unit object. By the functoriality of the block spaces, we then obtain a cochain complex on which the mapping class group of the surface with one additional boundary component acts. The main point of the argument will be that this additional boundary component can be closed again when passing to cohomology. To establish this, we proceed in two steps, using standard techniques from the theory of mapping class groups, namely the capping sequence in the first step and the Birman sequence in the second step. 

The present work sets out the general theory needed to establish the result just described, but also prepares the ground for the computation of explicit examples of these mapping class group representations, which will be addressed in our forthcoming article~\cite{LMSS2}. A new aspect of our approach is that we consider an action of the mapping class group on the fundamental group by requiring that the base point of the fundamental group be kept fixed. In this way, it is possible to avoid the necessity to identify homomorphisms that are related by simultaneous conjugation that can be found in many of the other articles on topological or conformal field theory. 

The material is organized as follows: Section~\ref{Sec:MapClGr} reviews surfaces, fundamental groups, and mapping class groups as well as the capping sequence and the Birman sequence. It also introduces the action of the mapping class group on the fundamental group just mentioned.
Section~\ref{Sec:TensCat} explains how representations of mapping class groups are assigned to certain tensor categories in topological field theory. For this, we use the framework created by V.~Lyubashenko in his articles~\cite{L1} and~\cite{L2}; in particular, we use the approach to surfaces via nets and ribbon graphs described in his articles. In Section~\ref{Sec:DerFunct}, we first state and prove our main result and then explain why this result generalizes our previous one from~\cite{LMSS1}. In fact, our present result was already mentioned in~\cite{LMSS1}, and it was also described in~\cite{FS2}. Here, we are now supplying proofs for our claims. 

Throughout the text, the word `projective' is used frequently. It has two very different meanings: When speaking about projective modules and projective resolutions, the term is used in the sense of ring theory and homological algebra. When speaking about projective space and projective representations or actions, the term is used in the sense of projective geometry. In particular, for a vector space~$V$, we denote the associated projective space, i.e., the set of its one-dimensional subspaces, by~$P(V)$, and the projectivity or homography induced by a bijective linear map~$f$ by~$P(f)$. The set of all projectivities from~$P(V)$ to itself forms the projective linear group~$\PGL(V)$, which is isomorphic to the general linear group~$\GL(V)$
modulo the scalar multiples of the identity transformation. By a projective representation or projective action, we mean a group homomorphism from a given group to the projective linear group.

\enlargethispage{-2.5mm}

We will assume throughout that our base field~$\K$ is algebraically closed. In general, we compose mappings and morphisms in a category from right to left, i.e., we have
$(g \circ f)(x) = g(f(x))$, so that~$f$ is applied first. In contrast, we concatenate paths in the fundamental group from left to right, i.e., in the concatenation~$\kappa \gamma$, the path~$\kappa$ is traced out first, while the path~$\gamma$ is traced out afterwards.

We use the symbol~$\V$ for the category of finite-dimensional vector spaces and the symbol~$S_n$ for the symmetric group in $n$~letters. Additional notation will be explained in the text. 

While carrying out this research, the first and the third author were partially supported by the RTG 1670 `Mathematics inspired by String Theory and Quantum Field Theory' and by the `Deutsche Forschungsgemeinschaft' under Germany's Excellence Strategy EXC 2121 `Quantum Universe' -- 390833306, while the second and the fourth author were partially supported by NSERC grant RGPIN-2017-06543.

\section{Mapping class groups} \label{Sec:MapClGr}
\subsection{The classification of surfaces} \label{Surf}
The principal result of the classification of surfaces asserts that every compact, connected, orientable, smooth 2-dimensional manifold is diffeomorphic to the connected sum~$\Sigma_g$ of a sphere with $g \ge 0$ tori, where it is understood that a diffeomorphism is arbitrarily often differentiable, and smoothness is understood in the same way. The number~$g$ of the attached tori, which are often called handles, is uniquely determined by the surface and is called its genus. If we include manifolds with boundary, then the classification theorem asserts that a compact, connected, orientable, smooth \mbox{2-dimensional} manifold with boundary is diffeomorphic to~$\Sigma_g$ with~$n \ge 0$ open disks removed, where the disks have to satisfy the requirement that their closures do not intersect. We will denote this surface by $\Sigma_{g,n}$, so that 
$\Sigma_g = \Sigma_{g,0}$. If the boundary is not empty, it is the finite disjoint union of~$n$ connected components, called the boundary components. We assume that on each of the boundary components, a point has been distinguished, or marked, as we also say. The Euler characteristic\nopagebreak~of~$\Sigma_{g,n}$ is
\[\chi(\Sigma_{g,n}) = 2 - 2g - n.\]

\pagebreak

By assumption, the surfaces above are orientable, but when we refer to~$\Sigma_{g,n}$, we will always assume that an orientation has been chosen. In the pictures of~$\Sigma_{g,n}$ that we will encounter below, $\Sigma_{g,n}$ will be drawn embedded into 3-dimensional space, in which case we will assume that the orientation is represented by the outward-pointing normal vector field. 

In the standard proof of the classification theorem, the surface is realized not as a connected sum, but rather as a quotient space of a polygon whose edges are labeled in a certain normal form. In the normal form for the closed surface~$\Sigma_{g}$, the edges are labeled consecutively as
\[\alpha_1, \beta_1, \alpha_{1}^{-1}, \beta_1^{-1}, \alpha_2, \beta_2, \alpha_2^{-1}, \beta_2^{-1}, \ldots, \alpha_g, \beta_g, \alpha_{g}^{-1}, \beta_g^{-1},\]
while for the surface~$\Sigma_{g,n}$ with boundary, the edges are labeled consecutively as
\[\alpha_1, \beta_1, \alpha_{1}^{-1}, \beta_1^{-1}, \ldots, \alpha_g, \beta_g, \alpha_{g}^{-1}, \beta_g^{-1}, 
\xi_1, \rho_1, \xi_1^{-1}, \ldots, \xi_n, \rho_n, \xi_n^{-1}.\]
In the case $g=n=3$, this polygon has the form
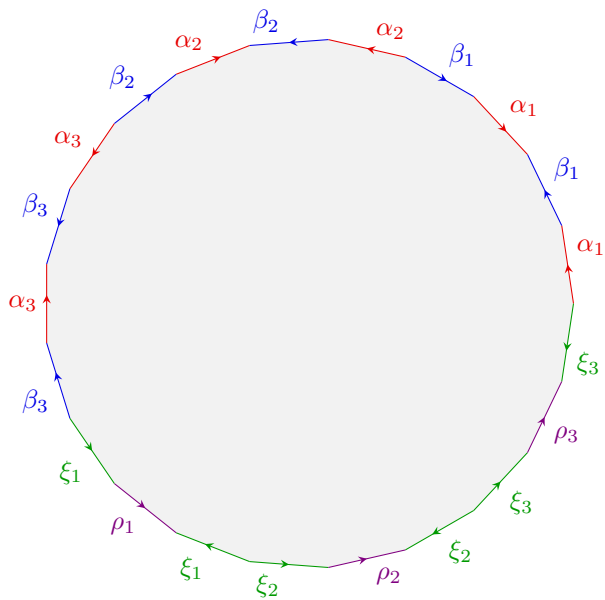
\begin{center}
\fontsize{10}{12}\selectfont
\begin{tikzpicture}
\node[fill=gray!10] (pol) [
  minimum size=200pt,
  regular polygon, regular polygon sides=21,
  rotate=270,
  ]{};
\foreach \x/\y/\i in {1/2/1,5/6/2,9/10/3} 
  \draw[black!10!red,auto=right, ->-]
    (pol.corner \x)--(pol.corner \y)
      node[black!10!red,midway]{$\alpha_{\i}$};
\foreach \x/\y/\i in {3/4/1,7/8/2,11/12/3} 
   \draw[black!10!red,auto=right, -<-]
     (pol.corner \x)--(pol.corner \y)
     node[black!10!red,midway]{$\alpha_{\i}$};
\foreach \x/\y/\i in {2/3/1,6/7/2,10/11/3} 
  \draw[black!10!blue,auto=right, ->-]
    (pol.corner \x)--(pol.corner \y)
      node[black!10!blue,midway]{$\beta_{\i}$};
\foreach \x/\y/\i in {4/5/1,8/9/2,12/13/3}
   \draw[black!10!blue,auto=right, -<-]
     (pol.corner \x)--(pol.corner \y)
     node[black!10!blue,midway]{$\beta_{\i}$};
\foreach \x/\y/\i in {13/14/1, 16/17/2,19/20/3} 
  \draw[black!40!green,auto=right,->-]
    (pol.corner \x)--(pol.corner \y)
      node[black!40!green,midway]{$\xi_ {\i}$};
\foreach \x/\y/\i in {15/16/1,18/19/2,21/1/3} 
   \draw[black!40!green,auto=right,-<-]
     (pol.corner \x)--(pol.corner \y)
     node[black!40!green,midway]{$\xi_{\i}$};
\foreach \x/\y/\i in {14/15/1,17/18/2,20/21/3} 
  \draw[violet,auto=right,->-]
    (pol.corner \x)--(pol.corner \y)
      node[violet,midway]{$\rho_{\i}$};
\end{tikzpicture}
\end{center}
where the inverse signs have been depicted by reversing the orientation of the corresponding arrow.

In the case of a closed surface, the quotient space is then realized in such a way that all the vertices of the polygon become a single point~$x$, while the edges labeled $\alpha_i$ and $\beta_i$ are, respectively, identified with their counterparts labeled~$\alpha_i^{-1}$ and~$\beta_i^{-1}$, which are traced out in the opposite direction, as we use the general convention that $\gamma^{-1}(t) \deq \gamma(1-t)$ for a path defined on the unit interval. In the case of the surface~$\Sigma_{g,n}$ with boundary, in addition an edge labeled $\xi_j$ is identified in the same way with its counterpart labeled~$\xi_j^{-1}$. The edges labeled~$\rho_j$ then become closed curves in the quotient that represent the~$n$~boundary components (cf.~\cite[Kap.~6, \S~40, p.~142ff]{ST}, \cite[Chap.~I, Sec.~10, p.~37ff]{Mas}). In this case, not all vertices of the polygon map to the same point~$x$, but only those that are not the start point or the end point of one of the edges~$\rho_1,\ldots,\rho_n$. The start point of~$\rho_j$ is instead identified with the end point of~$\rho_j$ and in this way yields the marked point on the $j$th boundary component.

If we carry out only the second identification in the case depicted above, i.e., the identification of~$\xi_j$ with~$\xi_j^{-1}$, we obtain the following picture of the intermediate stage:
\vspace{0.3cm}
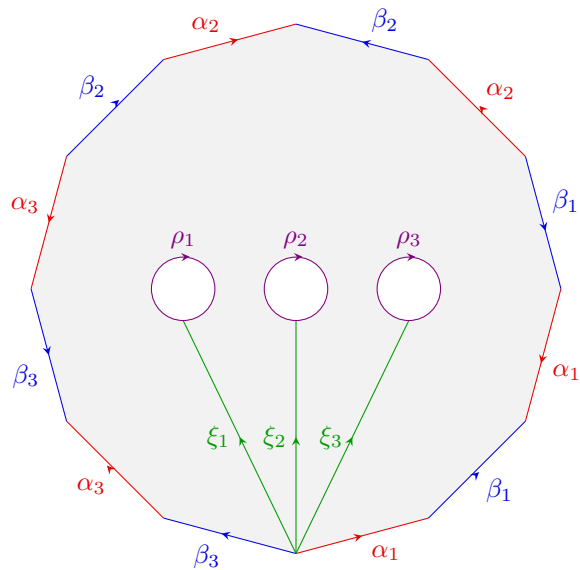
\begin{center}
\fontsize{10}{12}\selectfont
\begin{tikzpicture}[scale=0.6]
\node[fill=gray!10] (pol) [ 
minimum size=200pt, 
regular polygon, regular polygon sides=12, rotate=195 ]{};
\foreach \x/\y/\i in {1/2/1,5/6/2,9/10/3} 
\draw[black!10!red,auto=right,->-]
(pol.corner \x)--(pol.corner \y)
node[black!10!red,midway]{$\alpha_{\i}$};
\foreach \x/\y/\i in {3/4/1,7/8/2,11/12/3} 
\draw[black!10!red,auto=right,-<-]
(pol.corner \x)--(pol.corner \y)
node[black!10!red,midway]{$\alpha_{\i}$};
\foreach \x/\y/\i in {2/3/1,6/7/2,10/11/3} 
\draw[black!10!blue,auto=right,->-]
(pol.corner \x)--(pol.corner \y)
node[black!10!blue,midway]{$\beta_ {\i}$};
\foreach \x/\y/\i in {4/5/1,8/9/2,12/1/3} 
\draw[black!10!blue,auto=right,-<-]
(pol.corner \x)--(pol.corner \y)
node[black!10!blue,midway]{$\beta_{\i}$};

\draw[black!40!green,decoration={ markings, mark=at position 0.5 with {\arrow{stealth}}}, postaction={decorate}]  
(pol.corner 1) --(-2.5,-20pt) node[midway, left] {$\xi_1$} ;
\draw[black!40!green,decoration={ markings, mark=at position 0.5 with {\arrow{stealth}}}, postaction={decorate}]  
(pol.corner 1) -- (0,-20pt) node[midway, left] {$\xi_2$};
\draw[black!40!green,decoration={ markings, mark=at position 0.5 with {\arrow{stealth}}}, postaction={decorate}] 
(pol.corner 1) -- (2.5,-20pt) node[midway, left] {$\xi_3$};

\draw[violet,decoration={markings, mark=at position 0.26 with {\arrow[>=stealth]{<}}},
postaction={decorate},fill=white] (-2.5,0) circle (20pt)  ;
\node[violet,above] at (-2.5,20pt) {$\rho_1$};
\draw[violet,decoration={markings, mark=at position 0.26 with {\arrow[>=stealth]{<}}},
postaction={decorate},fill=white] (0,0) circle (20pt);
\node[violet,above] at (0,20pt) {$\rho_2$};
\draw[violet,decoration={markings, mark=at position 0.26 with {\arrow[>=stealth]{<}}},
postaction={decorate},fill=white] (2.5,0) circle (20pt);
\node[violet,above] at (2.5,20pt) {$\rho_3$};
\end{tikzpicture}
\end{center}
\vspace{0.4cm}
When we carry out all the identifications, the edges labeled $\alpha_i$ and $\beta_i$ become in the quotient the closed curves that appear in the following picture:
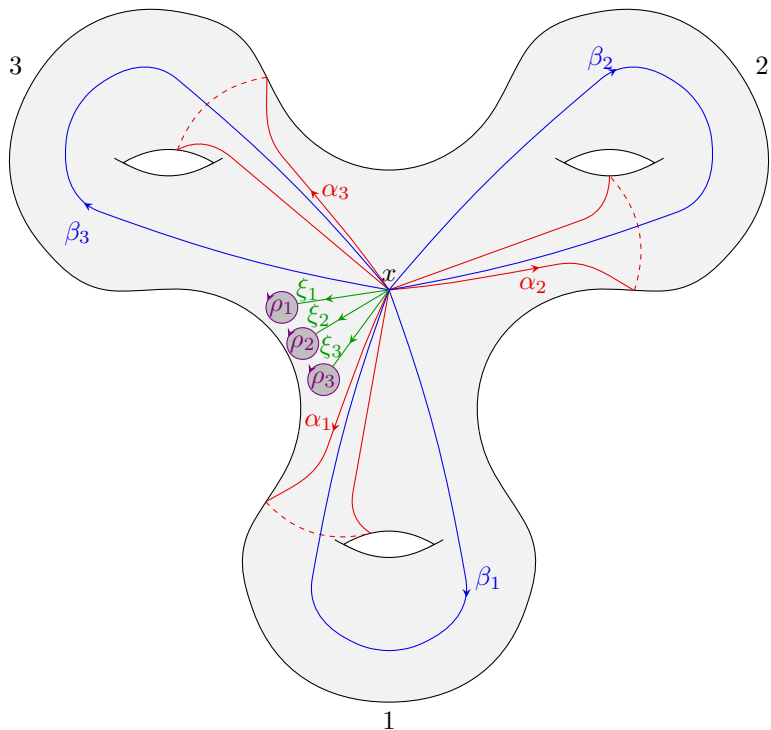
\begin{center}
\fontsize{10}{12}\selectfont
\begin{tikzpicture}[scale=0.6]
\pic[scale=0.6] (lower) at (0,-pi) {handle}; 
\pic[rotate=120,scale=0.6] (tr) at (30:pi) {handle};
\pic[rotate=-120,scale=0.6] (tl) at (150:pi) {handle};
\fill[gray!10]  (lower-right) to[out=100,in=200] (tr-left)-- 
(tr-right) to[out=-140,in=-40] (tl-left)
-- (tl-right) to[out=-20,in=80] (lower-left) -- cycle;
\draw (lower-right) to[out=100,in=200] (tr-left);
\draw (tr-right) to[out=-140,in=-40] (tl-left);
\draw (tl-right) to[out=-20,in=80] (lower-left);

\node[above] (v1) at (0,0) {$x$}; 
\node[below] at (lower-num) {1}; 
\node[above right] at (tr-num) {2}; 
\node[above left] at (tl-num) {3}; 

\draw[black!10!red] 
(0,0) to[out=260,in=80] (-0.8,-4.5) to[out=260,in=150] (lower-A);
\draw[black!10!red, dash pattern= on 2pt off 2pt] (-2.73,-4.7) to[bend right](lower-A);
\draw[black!10!red,decoration={ markings, mark=at position 0.6 with {\arrow{stealth}}}, postaction={decorate}]
(0,0) to[out=245,in=70] node[near end,below left]{$\alpha_1$} (-1.37,-3.5)  to[out=250,in=30] (-2.73,-4.7);

\draw[rotate=120,black!10!red] 
(0,0) to[out=260,in=80] (-0.8,-4.5) to[out=260,in=150](tr-B);
\draw[rotate=120,black!10!red, dash pattern= on 2pt off 2pt] (-2.73,-4.7) to[bend right](tr-B);
\draw[rotate=120,black!10!red,decoration={ markings, mark=at position 0.6 with {\arrow{stealth}}}, postaction={decorate}]
(0,0) to[out=245,in=70] node[very near end, below]{$\alpha_2$}(-1.37,-3.5)  to[out=250,in=30] (-2.73,-4.7);

\draw[rotate=240,black!10!red] 
(0,0) to[out=260,in=80] (-0.8,-4.5) to[out=260,in=150] (tl-C);
\draw[rotate=240,black!10!red, dash pattern= on 2pt off 2pt] (-2.73,-4.7) to[bend right] (tl-C);
\draw[rotate=240,black!10!red,decoration={ markings, mark=at position 0.5 with {\arrow{stealth}}}, postaction={decorate}]
(0,0) to[out=245,in=70] node[near end, right]{$\alpha_3$} (-1.37,-3.5)  to[out=250,in=30] (-2.73,-4.7);

\draw[black!10!blue,decoration={ markings, mark=at position 0.63 with {\arrow[>=stealth]{<}}}, postaction={decorate}] 
(0,0) to[out=250,in=80] (-1.7,-6.4) to[out=260,in=150] (-1,-7.7) to[out=330,in=178] (0,-8) 
to[out=2,in=210] (1,-7.7) to[out=30,in=280] 
(1.7,-6.4) node[right] {$\beta_1$}to[out=100,in=290] (0,0) ;

\draw[rotate=120,black!10!blue,decoration={ markings, mark=at position 0.63 with {\arrow[>=stealth]{<}}}, postaction={decorate}] 
(0,0) to[out=250,in=80] (-1.7,-6.4) to[out=260,in=150] (-1,-7.7) to[out=330,in=178] (0,-8) 
to[out=2,in=210] (1,-7.7) to[out=30,in=280] (1.7,-6.4)node[above] {$\beta_2$} to[out=100,in=290] (0,0) ;

\draw[rotate=240,black!10!blue,decoration={ markings, mark=at position 0.63 with {\arrow[>=stealth]{<}}}, postaction={decorate}] 
(0,0) to[out=250,in=80] (-1.7,-6.4) to[out=260,in=150] (-1,-7.7) to[out=330,in=178] (0,-8) 
to[out=2,in=210] (1,-7.7) to[out=30,in=280] (1.7,-6.4) node[below left] {$\beta_3$}to[out=100,in=290] (0,0) ;

\draw[violet,rotate=-60,decoration={markings, mark=at position 0.6 with {\arrow[>=stealth]{>}}}, 
postaction={decorate},fill=lightgray] (-0.85,-2.25)  circle (10pt) node {$\rho_1$} ;
\draw[violet,rotate=-60,decoration={markings, mark=at position 0.6 with {\arrow[>=stealth]{>}}},
postaction={decorate},fill=lightgray] (0.075,-2.25) circle (10pt) node {$\rho_2$} ;
\draw[violet,rotate=-60,decoration={markings, mark=at position 0.6 with {\arrow[>=stealth]{>}}},
postaction={decorate},fill=lightgray] (1,-2.25) circle (10pt) node {$\rho_3$} ;

\draw[black!40!green,rotate=-60,decoration={ markings, mark=at position 0.7 with {\arrow{stealth}}}, postaction={decorate}]  
(0,0) -- (-0.75,-1.91) ;
\draw[black!40!green,rotate=-60,decoration={ markings, mark=at position 0.7 with {\arrow{stealth}}}, postaction={decorate}]  
(0,0)  -- (0.03,-1.89) ;
\draw[black!40!green,rotate=-60,decoration={ markings, mark=at position 0.7 with {\arrow{stealth}}}, postaction={decorate}] 
(0,0) -- (0.85,-1.93) ;
\node[black!40!green,above] at (-1.8,-0.46) {$\xi_1$}; 
\node[black!40!green,above] at (-1.55,-1.04) {$\xi_2$};
\node[black!40!green,above] at (-1.27,-1.73) {$\xi_3$};
\end{tikzpicture}
\end{center}
\vspace{0.4cm}
This realization of the surface is called the polygon model of the surface, as opposed to the connected sum model that we briefly described at the beginning of the paragraph. Note that, in order to obtain the outward-pointing orientation of the surface mentioned above, the surface of the polygon needs to be oriented by a normal vector field that is orthogonal to the page and is pointing away from the reader, which is not the standard orientation of the plane used elsewhere in the text, for example in Paragraph~\ref{FundGroup}. Furthermore, in this model the boundary curves~$\rho_j$ do not carry the orientation that is induced by the surface on its boundary, but rather the opposite orientation.

\subsection{The fundamental group} \label{FundGroupGen}
In the polygon model of a closed surface, we use the point~$x$ in the quotient that all vertices of the polygon map to as the base point of the fundamental group $\pi_1(\Sigma_{g},x)$. The elements of the fundamental group are the homotopy classes~$[\gamma]$ of closed paths~$\gamma$ defined on the unit interval, where homotopy is relative to this base point~$x$. It turns out that the relative homotopy classes of the~$2g$ elements $\alpha_1,\beta_1,\ldots,\alpha_g,\beta_g$ just defined generate the fundamental group.  In the connected sum model, each pair $\alpha_i,\beta_i$ lies on one of the~$g$ tori that we attached to a sphere at the beginning of Paragraph~\ref{Surf}. Two neighboring tori are connected by the path 
$\mu_{i} \deq \alpha_{i+1}^{-1} \beta_i \alpha_i \beta_i^{-1}$, where it is understood that
$\mu_{g} \deq \alpha_{1}^{-1} \beta_g \alpha_g \beta_g^{-1}$ if~$i=g$. To see this, we first observe that 
the path $\beta_i \alpha_i \beta_i^{-1}$ is homotopic to a kind of mirror image of~$\alpha_i$:
\vspace{0.5cm}
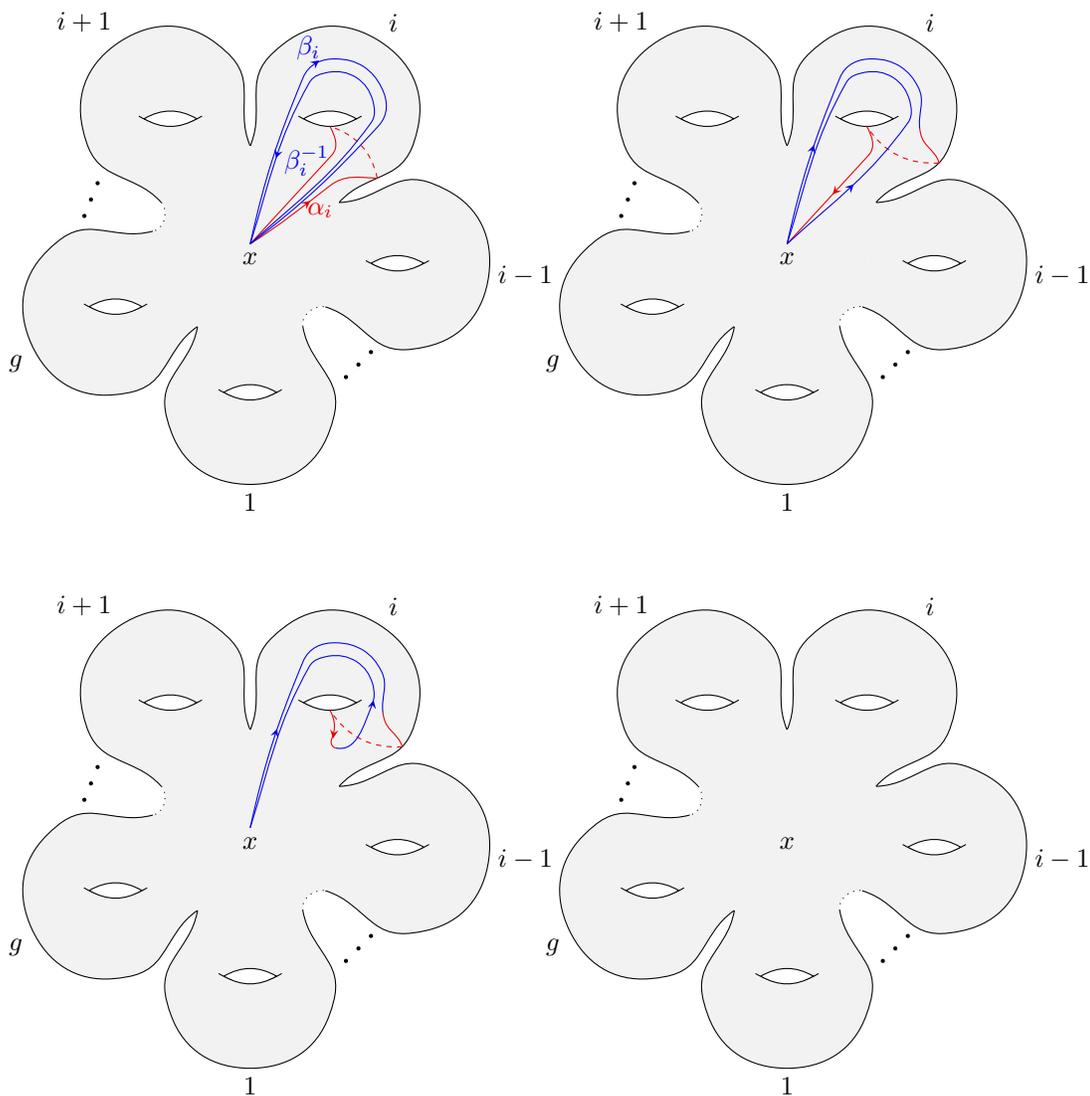
\begin{center}
\begin{minipage}[t]{7cm}
\fontsize{10}{12}\selectfont
\begin{tikzpicture}[scale=0.35]
\pic[scale=0.35] (lower) at (0,-pi) {handle}; 
\pic[rotate=82.5,scale=0.35] (lr) at (-7.5:pi) {handle};
\pic[rotate=147.5,scale=0.35] (tr) at (57.5:pi) {handle};
\pic[rotate=212.5,scale=0.35] (tl) at (122.5:pi) {handle};
\pic[rotate=295,scale=0.35] (ll) at (205:pi) {handle};
\fill[gray!10]  (lower-right) to[out=100,in=170] (lr-left)-- 
(lr-right) to (tr-left)
-- (tr-right) to (tl-left)
-- (tl-right) to[out=310,in=10] (ll-left) 
--(ll-right) to (lower-left) -- cycle;
\draw[dotted] (lower-right) to[out=100,in=170] (lr-left);
\draw[dotted] (tl-right) to[out=310,in=10] (ll-left);

\node[rotate=68,scale=1.5] at (-6,1.8) {\dots};
\node[rotate=45,scale=1.5] at (4.2,-4.5) {\dots};

\node[below] (v1) at (0,0) {$x$}; 
\node[below] at (lower-num) {1}; 
\node[right] at (lr-num) {$i-1$}; 
\node[above right] at (tr-num) {$i$}; 
\node[above left] at (tl-num) {$i+1$}; 
\node[below left] at (ll-num) {$g$}; 

\draw[rotate=147.5,black!10!red] 
(0,0) to[out=260,in=80] (-0.8,-4.5) to[out=260,in=150] (tr-B);
\draw[rotate=147.5,black!10!red, dash pattern= on 2pt off 2pt] (-2.73,-4.7) to[bend right] (tr-B);
\draw[rotate=147.5,black!10!red,decoration={ markings, mark=at position 0.5 with {\arrow{stealth}}}, postaction={decorate}]
(0,0) to[out=245,in=70] node[pos=0.6, right]{$\alpha_i$} (-1.37,-3.5)  to[out=250,in=30] (-2.73,-4.7);

\draw[rotate=147.5,black!10!blue,decoration={ markings, mark=at position 0.61 with {\arrow[>=stealth]{<}}}, postaction={decorate}] 
(0,0) to[out=250,in=80] (-1.7,-6.4) to[out=260,in=150] (-1,-7.7) to[out=330,in=178] (0,-8) 
to[out=2,in=210] (1,-7.7) to[out=30,in=280] node[near end, above] {$\beta_i$} (1.7,-6.4) to[out=100,in=290] (0,0) ;

\draw[rotate=147.5,black!10!blue,decoration={ markings, mark=at position 0.8 with {\arrow[>=stealth]{>}}}, postaction={decorate}] 
(0,0) to[out=255,in=87] 
(-1.4,-6.4) to[out=273,in=135] (-1,-7.1) to[out=320,in=178] (0,-7.5) 
to[out=2,in=220] (1,-7.1) to[out=45,in=270] 
(1.4,-6.4) to[out=93,in=285] node[midway,right] {$\beta_i^{-1}$}(0,0) ;
\end{tikzpicture}
\end{minipage}
\begin{minipage}[t]{7cm}
\fontsize{10}{12}\selectfont
\begin{tikzpicture}[scale=0.35]
\pic[scale=0.35] (lower) at (0,-pi) {handle}; 
\pic[rotate=82.5,scale=0.35] (lr) at (-7.5:pi) {handle};
\pic[rotate=147.5,scale=0.35] (tr) at (57.5:pi) {handle};
\pic[rotate=212.5,scale=0.35] (tl) at (122.5:pi) {handle};
\pic[rotate=295,scale=0.35] (ll) at (205:pi) {handle};
\fill[gray!10]  (lower-right) to[out=100,in=170] (lr-left)-- 
(lr-right) to (tr-left)
-- (tr-right) to (tl-left)
 -- (tl-right) to[out=310,in=10] (ll-left) 
  --(ll-right) to (lower-left) -- cycle;
\draw[dotted] (lower-right) to[out=100,in=170] (lr-left);
\draw[dotted] (tl-right) to[out=310,in=10] (ll-left);

\node[rotate=68,scale=1.5] at (-6,1.8) {\dots};
\node[rotate=45,scale=1.5] at (4.2,-4.5) {\dots};

\node[below] (v1) at (0,0) {$x$}; 
\node[below] at (lower-num) {1}; 
\node[right] at (lr-num) {$i-1$}; 
\node[above right] at (tr-num) {$i$}; 
\node[above left] at (tl-num) {$i+1$}; 
\node[below left] at (ll-num) {$g$}; 

\draw[rotate=147.5,black!10!red,decoration={ markings, mark=at position 0.5 with {\arrow[>=stealth]{<}}}, postaction={decorate}] 
(0,0) to[out=260,in=80] 
(-0.8,-4.5) to[out=260,in=150] (tr-B);
\draw[rotate=147.5,black!10!red, dash pattern= on 2pt off 2pt] (-3.22,-5.7) to[bend left] (tr-B);
\draw[rotate=147.5,black!10!red] (-3.22,-5.7) to[out=320,in=165] (-2.5,-6) to[out=345,in=130] (-1.9,-6.4);

\draw[rotate=147.5,black!10!blue,decoration={ markings, mark=at position 0.7 with {\arrow[>=stealth]{<}}}, postaction={decorate}] 
(-1.9,-6.4) to[out=310,in=150] 
(-1,-7.7) to[out=330,in=178] (0,-8) 
to[out=2,in=210] (1,-7.7) to[out=30,in=280] 
(1.7,-6.4) to[out=100,in=290] (0,0) ;

\draw[rotate=147.5,black!10!blue,decoration={ markings, mark=at position 0.2 with {\arrow[>=stealth]{>}}}, postaction={decorate}] 
(0,0) to[out=255,in=87] 
(-1.4,-6.4) to[out=273,in=135] (-1,-7.1) to[out=320,in=178] (0,-7.5) 
to[out=2,in=220] (1,-7.1) to[out=45,in=270] 
(1.4,-6.4) to[out=93,in=285] (0,0) ;
\end{tikzpicture}
\end{minipage}

\vspace{0.5cm}

\begin{minipage}[t]{7cm}
\fontsize{10}{12}\selectfont
\begin{tikzpicture}[scale=0.35]
\pic[scale=0.35] (lower) at (0,-pi) {handle}; 
\pic[rotate=82.5,scale=0.35] (lr) at (-7.5:pi) {handle};
\pic[rotate=147.5,scale=0.35] (tr) at (57.5:pi) {handle};
\pic[rotate=212.5,scale=0.35] (tl) at (122.5:pi) {handle};
\pic[rotate=295,scale=0.35] (ll) at (205:pi) {handle};
\fill[gray!10]  (lower-right) to[out=100,in=170] (lr-left)-- 
(lr-right) to (tr-left)
-- (tr-right) to (tl-left)
 -- (tl-right) to[out=310,in=10] (ll-left) 
  --(ll-right) to (lower-left) -- cycle;
\draw[dotted] (lower-right) to[out=100,in=170] (lr-left);
\draw[dotted] (tl-right) to[out=310,in=10] (ll-left);

\node[rotate=68,scale=1.5] at (-6,1.8) {\dots};
\node[rotate=45,scale=1.5] at (4.2,-4.5) {\dots};

\node[below] (v1) at (0,0) {$x$}; 
\node[below] at (lower-num) {1}; 
\node[right] at (lr-num) {$i-1$}; 
\node[above right] at (tr-num) {$i$}; 
\node[above left] at (tl-num) {$i+1$}; 
\node[below left] at (ll-num) {$g$};

\draw[rotate=147.5,black!10!red,decoration={ markings, mark=at position 0.5 with {\arrow[>=stealth]{<}}}, postaction={decorate}] 
(-1.1,-4.3) to[out=20,in=100] (-0.8,-4.5) to[out=280,in=150] (tr-B);
\draw[rotate=147.5,black!10!red, dash pattern= on 2pt off 2pt] (-3.22,-5.7) to[bend left](tr-B);
\draw[rotate=147.5,black!10!red] (-3.22,-5.7) to[out=320,in=165] (-2.5,-6) to[out=345,in=130] (-1.9,-6.4);

\draw[rotate=147.5,black!10!blue,decoration={ markings, mark=at position 0.7 with {\arrow[>=stealth]{<}}}, postaction={decorate}] 
(-1.9,-6.4) to[out=310,in=150] 
(-1,-7.7) to[out=330,in=178] (0,-8) 
to[out=3,in=210] (1,-7.7) to[out=30,in=280] 
(1.7,-6.4) 
to[out=100,in=290] (0,0) ;

\draw[rotate=147.5,black!10!blue,decoration={ markings, mark=at position 0.2 with {\arrow[>=stealth]{>}}}, postaction={decorate}] 
(-1.1,-4.3) to[out=200,in=100] (-1.4,-6.4) to[out=280,in=135] (-1,-7.1) to[out=320,in=178] (0,-7.5) 
to[out=3,in=220] (1,-7.1) to[out=45,in=270] 
(1.4,-6.4)
to[out=93,in=285] (0,0) ;
\end{tikzpicture}
\end{minipage}
\begin{minipage}[t]{7cm}
\fontsize{10}{12}\selectfont
\begin{tikzpicture}[scale=0.35]
\pic[scale=0.35] (lower) at (0,-pi) {handle}; 
\pic[rotate=82.5,scale=0.35] (lr) at (-7.5:pi) {handle};
\pic[rotate=147.5,scale=0.35] (tr) at (57.5:pi) {handle};
\pic[rotate=212.5,scale=0.35] (tl) at (122.5:pi) {handle};
\pic[rotate=295,scale=0.35] (ll) at (205:pi) {handle};
\fill[gray!10]  (lower-right) to[out=100,in=170] (lr-left)-- 
(lr-right) to (tr-left)
-- (tr-right) to (tl-left)
-- (tl-right) to[out=310,in=10] (ll-left) 
--(ll-right) to (lower-left) -- cycle;
\draw[dotted] (lower-right) to[out=100,in=170] (lr-left);
\draw[dotted] (tl-right) to[out=310,in=10] (ll-left);

\node[rotate=68,scale=1.5] at (-6,1.8) {\dots};
\node[rotate=45,scale=1.5] at (4.2,-4.5) {\dots};

\node[below] (v1) at (0,0) {$x$}; 
\node[below] at (lower-num) {1}; 
\node[right] at (lr-num) {$i-1$}; 
\node[above right] at (tr-num) {$i$}; 
\node[above left] at (tl-num) {$i+1$}; 
\node[below left] at (ll-num) {$g$};

\scalebox{-1}[1]{
\draw[rotate=212.5,black!10!blue] 
(0,0) to[out=260,in=80] (-1.1,-4.7) ;
\draw[rotate=212.5,black!10!red] (-1.1,-4.7) to[out=260,in=190] (tl-C);
\draw[rotate=212.5,black!10!red, dash pattern = on 2pt off 2pt] (-2.73,-4.7) to[bend right] (tl-C);
\draw[rotate=212.5,black!10!blue,,decoration={ markings, mark=at position 0.65 with {\arrow[>=stealth]{>}}}, postaction={decorate}] (0,0) to[out=245,in=70] (-1.37,-3.5); 
\draw[rotate=212.5,black!10!red] (-1.37,-3.5) to[out=250,in=30] (-2.73,-4.7);}
\end{tikzpicture}
\end{minipage}
\end{center}

\pagebreak
If we now also concatenate with the curve $\alpha_{i+1}^{-1}$, we see that the curve~$\mu_{i}$ is homotopic to a curve that connects the $i$th and $(i+1)$st handle:
\vspace{0.4cm}
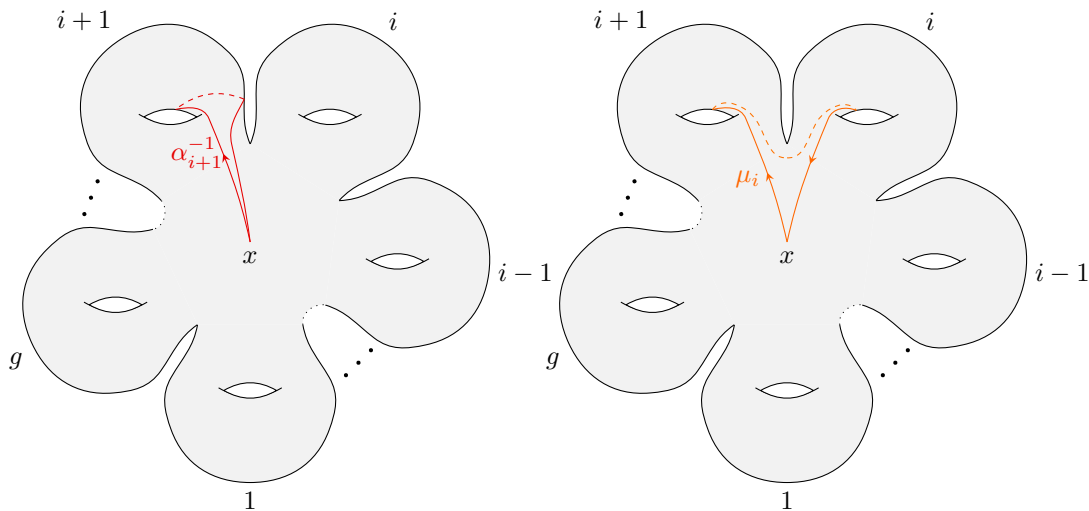
\begin{center}
\begin{minipage}[t]{7cm}
\fontsize{10}{12}\selectfont
\begin{tikzpicture}[scale=0.35]
\pic[scale=0.35] (lower) at (0,-pi) {handle}; 
\pic[rotate=82.5,scale=0.35] (lr) at (-7.5:pi) {handle};
\pic[rotate=147.5,scale=0.35] (tr) at (57.5:pi) {handle};
\pic[rotate=212.5,scale=0.35] (tl) at (122.5:pi) {handle};
\pic[rotate=295,scale=0.35] (ll) at (205:pi) {handle};
\fill[gray!10]  (lower-right) to[out=100,in=170] (lr-left)-- 
(lr-right) to (tr-left)
-- (tr-right) to (tl-left)
-- (tl-right) to[out=310,in=10] (ll-left) 
--(ll-right) to (lower-left) -- cycle;
\draw[dotted] (lower-right) to[out=100,in=170] (lr-left);
\draw[dotted] (tl-right) to[out=310,in=10] (ll-left);

\node[rotate=68,scale=1.5] at (-6,1.8) {\dots};
\node[rotate=45,scale=1.5] at (4.2,-4.5) {\dots};

\node[below] (v1) at (0,0) {$x$}; 
\node[below] at (lower-num) {1}; 
\node[right] at (lr-num) {$i-1$}; 
\node[above right] at (tr-num) {$i$};
\node[above left] at (tl-num) {$i+1$}; 
\node[below left] at (ll-num) {$g$};

\scalebox{-1}[1]{
\draw[rotate=212.5,black!10!blue] 
(0,0) to[out=260,in=80] (-1.1,-4.7) ;
\draw[rotate=212.5,black!10!red] (-1.1,-4.7) to[out=260,in=190] (tl-C);
\draw[rotate=212.5,black!10!red, dash pattern = on 2pt off 2pt] (-2.73,-4.7) to[bend right] (tl-C);
\draw[rotate=212.5,black!10!blue,,decoration={ markings, mark=at position 0.65 with {\arrow[>=stealth]{>}}}, postaction={decorate}] (0,0) to[out=245,in=70] (-1.37,-3.5); 
\draw[rotate=212.5,black!10!red] (-1.37,-3.5) to[out=250,in=30] (-2.73,-4.7);}

\draw[rotate=212.5,black!10!red,decoration={ markings, mark=at position 0.55 with {\arrow[>=stealth]{>}}}, postaction={decorate}] 
(0,0) to[out=250,in=80] node[pos=0.5,above left]{$\alpha_{i+1}^{-1}$} (-1.2,-4.8) to[out=260,in=160] (tl-C);
\draw[rotate=212.5,black!10!red, dash pattern= on 2pt off 2pt] (-2.73,-4.7) to[bend right](tl-C);
\draw[rotate=212.5,black!10!red]
(0,0) to[out=245,in=70] (-1.4,-3.5)  to[out=250,in=30] (-2.73,-4.7);
\end{tikzpicture}
\end{minipage}
\begin{minipage}[t]{7cm}
\fontsize{10}{12}\selectfont
\begin{tikzpicture}[scale=0.35]
\pic[scale=0.35] (lower) at (0,-pi) {handle}; 
\pic[rotate=82.5,scale=0.35] (lr) at (-7.5:pi) {handle};
\pic[rotate=147.5,scale=0.35] (tr) at (57.5:pi) {handle};
\pic[rotate=212.5,scale=0.35] (tl) at (122.5:pi) {handle};
\pic[rotate=295,scale=0.35] (ll) at (205:pi) {handle};
\fill[gray!10]  (lower-right) to[out=100,in=170] (lr-left)-- 
(lr-right) to (tr-left)
-- (tr-right) to (tl-left)
-- (tl-right) to[out=310,in=10] (ll-left) 
--(ll-right) to (lower-left) -- cycle;
\draw[dotted] (lower-right) to[out=100,in=170] (lr-left);
\draw[dotted] (tl-right) to[out=310,in=10] (ll-left);

\node[rotate=68,scale=1.5] at (-6,1.8) {\dots};
\node[rotate=45,scale=1.5] at (4.2,-4.5) {\dots};

\node[below] (v1) at (0,0) {$x$}; 
\node[below] at (lower-num) {1}; 
\node[right] at (lr-num) {$i-1$}; 
\node[above right] at (tr-num) {$i$}; 
\node[above left] at (tl-num) {$i+1$}; 
\node[below left] at (ll-num) {$g$}; 

\draw[rotate=147.5,red!20!orange,decoration={ markings, mark=at position 0.5 with {\arrow[>=stealth]{>}}}, postaction={decorate}] 
(tr-A) to[out=10,in=280] (1.2,-4.8) to[out=100,in=290] (0,0) ;

\draw[rotate=212.5,red!20!orange,decoration={ markings, mark=at position 0.43 with {\arrow[>=stealth]{>}}}, postaction={decorate}] 
(0,0) to[out=250,in=80]node[left] {$\mu_i$} (-1.2,-4.8) to[out=260,in=160]  (tl-C);

\draw[red!10!orange,rotate=212.5,dash pattern= on 2pt off 2pt] (tl-C) to[out=200,in=270] (-1.5,-4.5) to[out=90,in=330] 
 (-1.72,-2.7);
\draw[red!10!orange,rotate=147.5,dash pattern= on 2pt off 2pt] (1.72,-2.7) to[out=210,in=90] (1.5,-4.5) to[out=270,in=340] (tr-A);
\end{tikzpicture}
\end{minipage}
\end{center}
\vspace{0.4cm}
These pictures in fact illustrate that the concatenated path $\mu_{g} \mu_{g-1} \cdots \mu_{1}$ is homotopic to the constant path based at~$x$, which means that
\begin{align*}
(\alpha_{1}^{-1} \beta_g \alpha_g \beta_g^{-1})
(\alpha_{g}^{-1} \beta_{g-1} \alpha_{g-1} \beta_{g-1}^{-1}) \cdots 
(\alpha_{2}^{-1} \beta_1 \alpha_1 \beta_1^{-1})
\end{align*}
represents the unit element in the fundamental group. Alternatively, by conjugating with $\alpha_{1}$ and inverting, we see that
\begin{align*}
\alpha_1 \beta_1 \alpha_{1}^{-1} \beta_1^{-1} \alpha_2 \beta_2 \alpha_2^{-1} \beta_2^{-1} \cdots 
\alpha_g \beta_g \alpha_{g}^{-1} \beta_g^{-1}
\end{align*}
represents the unit element in the fundamental group. This fact is even easier to see in the polygon model, where this path exactly traces out the boundary of the polygon and can therefore be contracted to a point in the interior of the polygon. Using the Seifert-van Kampen theorem, one can show that this relation is a defining relation for the fundamental group (cf.~\cite[Prop.~17.6, p.~242]{F}).

For the surfaces~$\Sigma_{g,n}$ with boundary, additional generators are needed for the fundamental group, which we denote by $\delta_1,\dots,\delta_n$. Starting at the base point, $\delta_j$ circles around the $j$th boundary component. Together $\delta_1,\dots,\delta_n$ form a so-called bouquet of circles. In the polygon model, $\delta_j$ arises by mapping the concatenation of $\xi_j$, $\rho_j$, and~$\xi_j^{-1}$ to the quotient. We therefore see that in this case, 
\begin{align*}
\alpha_1 \beta_1 \alpha_{1}^{-1} \beta_1^{-1} \alpha_2 \beta_2 \alpha_2^{-1} \beta_2^{-1} \cdots 
\alpha_g \beta_g \alpha_{g}^{-1} \beta_g^{-1} \delta_1 \cdots \delta_n 
\end{align*}
is homotopic to the constant path based at~$x$, and again this yields a defining relation for the fundamental group (cf.~\cite[Chap.~I, No.~43B, p.~100]{AS}). Because it is possible to solve in this relation for the last of the new generators, this implies that, in the presence of boundary components, the fundamental group of~$\Sigma_{g,n}$ is a free group on $2g+n-1$ generators. We note that, by a slight movement, it is always possible to find a representative in a relative homotopy class that does not intersect the boundary. In the sequel, we will assume in particular that this has been done for the generators~$\delta_j$.

\subsection{Mapping class groups} \label{MapClGr}
An important object associated with the surface~$\Sigma_{g,n}$ is its mapping class group 
$\Gamma_{g,n}$. There are several variants for the definition of this group in the literature; we will now explain which one is used in this article and how it compares to other variants.

By the boundary invariance theorem (cf.~\cite[Chap.~IV, Prop.~3.9, p.~61]{D}), a diffeomorphism of~$\Sigma_{g,n}$ will necessarily map a boundary component to a boundary component, but we require that the diffeomorphisms we consider also map the marked point on each boundary component to the corresponding marked point on the other boundary component. The group~$\Diffeo^+(\Sigma_{g,n})$ of orientation-preserving diffeomorphisms of~$\Sigma_{g,n}$ that permute the marked points becomes a topological group when endowed with the compact-open topology (cf.~\cite[Thm.~4, p.~598]{A}). By the fundamental property of the compact-open topology (cf.~\cite[Satz~14.17, p.~167]{Q}), the path-component~$\Diffeo^+_0(\Sigma_{g,n})$ of the identity mapping consists precisely of those orien\-tation-preserving diffeomorphisms that are isotopic to the identity. It is a normal subgroup. Continuity implies that the elements of~$\Diffeo^+_0(\Sigma_{g,n})$ cannot permute the marked points, but rather need to preserve them individually. We define the mapping class group as the corresponding quotient group:
\begin{Definition}
The mapping class group $\Gamma_{g,n}$ of $\Sigma_{g,n}$ is the quotient group
\[\Gamma_{g,n} \deq \Diffeo^+(\Sigma_{g,n})/\Diffeo^+_0(\Sigma_{g,n}).\]
In accordance with our notation for surfaces in Paragraph~\ref{Surf}, we will also write
$\Gamma_{g}$ if~$n=0$. The mapping class of a diffeomorphism~$\psi \in \Diffeo^+(\Sigma_{g,n})$ will be denoted by~$[\psi]$.
\end{Definition}

In view of this definition, two diffeomorphisms represent the same element in the mapping class group if and only if they are contained in the same path-component of~$\Diffeo^+(\Sigma_{g,n})$; i.e., if and only if they are isotopic. Because the path connecting them is contained in~$\Diffeo^+(\Sigma_{g,n})$, the corresponding isotopy must permute the marked points at each time of the deformation process. By continuity, this is only possible if all diffeomorphisms occurring in this process, and in particular those at the beginning and the end, permute the marked points in the same way. By enumerating the marked points, we can therefore obtain a group homomorphism
\[p\colon \Gamma_{g,n} \to S_n\]
to the symmetric group~$S_n$. We call the kernel of~$p$ the pure mapping class group and denote it by~$\Pu\Gamma_{g,n}$. An element in the pure mapping class group can be represented by a diffeomorphism that restricts to the identity on the entire boundary, not only on the marked points. As explained in \cite[Sec.~1.4, p.~42]{FM}, the mapping class group can alternatively be defined by using homeomorphisms instead of diffeomorphisms. 

For a subset~$U$ of~$\Sigma_{g,n}$, we also consider the subgroup of~$\Diffeo^+(\Sigma_{g,n})$ consisting of those diffeomorphisms that restrict to the identity on~$U$. If we divide it by its path component of the identity mapping, the arising quotient group~$\Gamma_{g,n}(U)$ is called the relative mapping class group modulo~$U$. It comes with a canonical group homomorphism 
\[F_U\colon \Gamma_{g,n}(U) \to \Gamma_{g,n}\]
called the forgetful map, because it arises by forgetting the information about the restriction to~$U$.
This homomorphism is not necessarily injective, because it might happen that a diffeomorphism is isotopic to the identity, although the diffeomorphisms that occur during this deformation process cannot be chosen so that they restrict to the identity on~$U$. If~$U = \{u\}$ consists of a single point~$u$, we will also write~$\Gamma_{g,n}(u)$ and~$F_u$ instead of~$\Gamma_{g,n}(\{u\})$ and~$F_{\{u\}}$.

The definition of mapping class groups used here is the one from~\cite[Sec.~4, p.~485]{L1}. Often, other definitions are used, for example in~\cite[Sec.~2.1, p.~44]{FM} or~\cite[Sec.~1, p.~101]{Ko}. In these definitions, there are additional marked points in the interior, not on the boundary like in our definition. These points are frequently called punctures. In contrast to our definition, both the definition in~\cite{FM} and the definition in~\cite{Ko} require that the diffeomorphisms restrict to the identity on the boundary. In~\cite{FM}, the punctures in the interior may be permuted by a diffeomorphism, while they are required to be preserved individually in~\cite{Ko}.

\subsection{Dehn twists} \label{Dehn}
An annulus can be defined as the Cartesian product~$A \deq S^1 \times [0,2 \pi]$ of the unit circle $S^1 \deq \{z \in \C \mid |z|=1\}$ and the interval~$[0, 2 \pi] \subset \R$. $A$~is an orientable {\mbox{2-manifold}} with boundary, which is diffeomorphic to~$\Sigma_{0,2}$ and therefore sometimes called a binion. We single out one of the two possible orientations by requiring that in the tangent space of~$A$ at the point~$(1,\pi)$ the tangent vector~$v_1$ to the curve $t \mapsto (1, \pi + t)$ and the tangent vector~$v_2$  to the curve $t \mapsto (e^{it}, \pi)$ form a positively oriented basis~$v_1,v_2$. Note that both curves start at the specified point for the parameter value $t=0$. On~$A$, we define the twist map
\[\fd\colon A \to A,~(z,t) \mapsto (z e^{it},t). \]

Now suppose that~$\gamma\colon S^1 \to \Sigma_{g,n}$ is a simple closed curve that does not intersect the boundary. It follows from the existence of tubular neighborhoods (cf.~\cite[Chap.~4, Thm.~5.2, p.~110]{H}) that there is an orientation-preserving embedding $\phi\colon A \to \Sigma_{g,n}$
with the property that $\phi(z,0) = \gamma(z)$. 
\begin{Definition}
We define $\fd_\gamma\colon \Sigma_{g,n} \to \Sigma_{g,n}$ as the map
\[p \mapsto 
\begin{cases}
(\phi \circ \fd \circ \phi^{-1})(p) &:  p \in \phi(A) \\
p &:  p \notin \phi(A)
\end{cases}\] 
Note that this map is a homeomorphism, not a diffeomorphism. But as we mentioned above, the mapping class group can also be defined using homeomorphisms, so that~$\fd_\gamma$ determines an element~$[\fd_\gamma]$ in the mapping class group, called the Dehn twist along~$\gamma$. Since $\fd(z,0) = (z,0)$, we have 
$\fd_\gamma \circ \gamma = \gamma$.
\end{Definition}
It follows from the isotopy of tubular neighborhoods (cf.~\cite[Chap.~4, Thm.~5.3, p.~112]{H}) that the mapping class of~$\fd_\gamma$ does not depend on~$\phi$. Furthermore, any curve isotopic to~$\gamma$ yields the same Dehn twist in the mapping class group. Further details on Dehn twists and their geometric meaning can be found in~\cite[Sec.~3.1, p.~64ff]{FM}. It is important to note that in some references the other orientation for the annulus~$A$ is used in the definition of a Dehn twist, for example in~\cite[Par.~2.2, p.~316f]{L2} and~\cite[Sec.~2, p.~102]{Ko}. Dehn twists defined using the other orientation are the inverses of the Dehn twists as they are defined here.

It is almost immediate from this definition that non-intersecting curves $\gamma_1$ and $\gamma_2$ give rise to commuting Dehn-twists $\fd_{\gamma_1}$ and $\fd_{\gamma_2}$. This and other basic properties of Dehn-twists are discussed in~\cite[Sec.~3.2, p.~73ff]{FM}.  

We will introduce some special notation for a number of Dehn twists along certain curves. The first of these are the following: In the construction of~$\Sigma_{g,n}$ from~$\Sigma_{g}$, the removal of the $j$th open disk from~$\Sigma_g$ leaves a connected component of the boundary that is diffeomorphic to the unit circle~$S^1$. In the polygon model discussed in Paragraph~\ref{Surf}, this curve is parametrized by~$\rho_j$. By moving this curve slightly to the interior, for example by using a collar (cf.~\cite[Chap.~4, Thm.~6.1, p.~113]{H}), we obtain a simple closed curve~$\partial_j$ that is freely homotopic to~$\rho_j$, and therefore also to~$\delta_j$, but does not intersect the boundary. We will use the notation~$\fd_j$ for the Dehn twist along~$\partial_j$. As explained in~\cite[Par.~4.2.5, p.~102]{FM}, $\fd_j$~is contained in the center of the pure mapping class group.

\subsection{Braidings} \label{Braid}
Besides Dehn twists, the second type of elements of the mapping class group that we will need are the braidings of the boundary components. To define them, we consider the surface~$B$ in~$\R^2$ that consists of the closed unit disk from which two open disks with radius~$1/4$ and centers~$(1/2,0)$ and~$(-1/2,0)$ on the $x$-axis, respectively, have been removed. $B$~is an orientable {\mbox{2-manifold}} with boundary, which is sometimes called a trinion, as it is diffeomorphic to~$\Sigma_{0,3}$. We orient~$B$ by requiring that the canonical basis at the origin is positively oriented. On the boundary components, i.e., the unit circle and the two smaller circles, we choose the base points~$(0,1)$, $(1/2,1/4)$, and~$(-1/2,1/4)$, respectively. As discussed in~\cite[Sec.~V.2.5, p.~251]{T}, there is an orientation-preserving diffeomorphism $\fb\colon B \to B$, which we call the braiding map, that is the identity on the unit circle, interchanges the two smaller circles as well as their base points while preserving their orientations, and transforms the line segments connecting the base point of the unit circle with the base points of the smaller circles as indicated in the picture

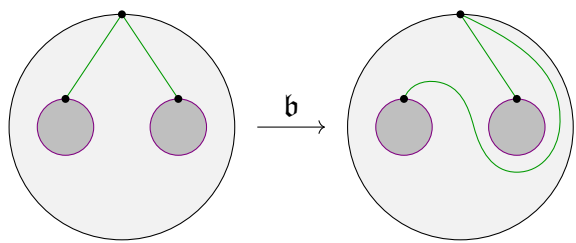
\begin{center}
\begin{tikzpicture}[scale=1.5]
\draw[fill=gray!10] (0,0) circle (1);
\draw[violet,fill=lightgray] (-0.5,0) circle (0.25);
\draw[violet,fill=lightgray] (0.5,0) circle (0.25);
\draw[black!40!green] (0,1) to (-0.5,0.25);
\draw[black!40!green]  (0,1) to (0.5,0.25);
\node[circle, fill=black,inner sep=0pt,minimum size=3pt] at (0,1) {};
\node[circle, fill=black,inner sep=0pt,minimum size=3pt] at (-0.5,0.25) {};
\node[circle, fill=black,inner sep=0pt,minimum size=3pt] at (0.5,0.25) {};

\draw[->] (1.2,0) to node[above] {$\mathfrak{b}$} (1.8,0);

\draw[fill=gray!10] (3,0) circle (1);
\draw[violet,fill=lightgray] (2.5,0) circle (0.25);
\draw[violet,fill=lightgray] (3.5,0) circle (0.25);

\draw[black!40!green] (3,1) to[out=-25,in=85] (3.88,0) to[out=265,in=0] (3.5,-0.4) to[out=180, in=285] (3.12,0) to[out=105,in=350] (2.8,0.4) to[out=170,in=70] (2.5,0.25);
\draw[black!40!green]  (3,1) to (3.5,0.25);

\node[circle, fill=black,inner sep=0pt,minimum size=3pt] at (3,1) {};
\node[circle, fill=black,inner sep=0pt,minimum size=3pt] at (2.5,0.25) {};
\node[circle, fill=black,inner sep=0pt,minimum size=3pt] at (3.5,0.25) {};
\end{tikzpicture}
\end{center}
(cf.~also \cite[Fig.~2.9, p.~254]{T}). As also stated in \cite[loc.~cit.]{T}, these conditions determine~$\fb$ up to isotopy. It should be noted that the braiding map we define here is the inverse of the braiding map in~\cite[Sec.~V.2.5, p.~251]{T}, but coincides with the map used in~\cite[Par.~5.1.1, p.~119]{FM}.

We will now use the braiding map to define the braiding of two boundary components of~$\Sigma_{g,n}$. It should be noted that it is not possible to define such a braiding if the boundary components are distributed over the surface in a somewhat arbitrary fashion, as in our description of the surface in terms of the connected sum model, as this definition depends not only on these two boundary components alone, but also on their position relative to the other boundary components. Rather, it is necessary that the boundary components are not only numbered, but in fact arranged in a certain order, like in the polygon model described in Paragraph~\ref{Surf}. We use the standard arrangement of the boundary components in this model as follows: For two indices satisfying~\mbox{$1\le i<j \le n$}, we choose an orientation-preserving embedding 
\mbox{$\varphi_{i,j}\colon B \to \Sigma_{g,n}$} of surfaces that maps the circle with center~$(-1/2,0)$ to the $i$th boundary component of~$\Sigma_{g,n}$ and the circle with center~$(1/2,0)$ to the $j$th boundary component of~$\Sigma_{g,n}$, mapping base points to base points. The unit circle is mapped to a smooth curve~$\gamma_{i,j}$ on~$\Sigma_{g,n}$ that is defined as follows: Starting at~$x$, we follow a path slightly right from~$\xi_i$ until we reach the curve $\partial_i$, which we follow until we need to return along~$\xi_i$, which we now do on its left side, which is the right side for the reversed orientation of~$\xi_i$. Shortly before reaching~$x$, however, we turn right and follow a path slightly right from~$\xi_j$ until we reach $\partial_j$, which we follow in the same way as before until we need to return along~$\xi_j$, which we now do on its left side until we reach the base point~$x$. The curve $\gamma_{i,j}$ is illustrated in the following picture:
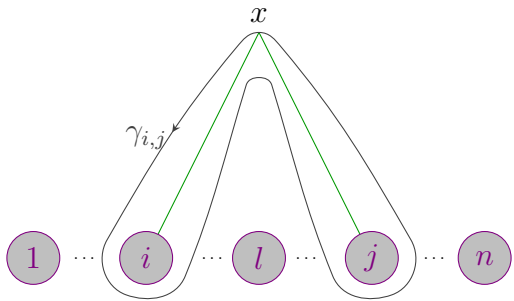
\begin{center}
\begin{tikzpicture}
\node[above] (v1) at (0,0) {$x$};
\draw[black!40!green] (0,0) to (-1.5,-3);
\draw[black!40!green] (0,0) to (1.5,-3);
\draw[violet,fill=lightgray] (-3,-3)  circle (10pt) node {1} ;
\draw[violet,fill=lightgray] (-1.5,-3)  circle (10pt) node {$i$} ;
\draw[violet,fill=lightgray] (0,-3) circle (10pt) node {$l$} ;
\draw[violet,fill=lightgray] (1.5,-3) circle (10pt) node {$j$} ;
\draw[violet,fill=lightgray] (3,-3)  circle (10pt) node {$n$} ;
\node[scale=0.6] at (-2.3,-3) {\dots};
\node[scale=0.6] at (-0.6,-3) {\dots};
\node[scale=0.6] at (0.65,-3) {\dots};
\node[scale=0.6] at (2.35,-3) {\dots};
\draw[darkgray,decoration={ markings, mark=at position 0.11 with {\arrow[>=stealth]{>}}}, postaction={decorate}] (0,0) to[out=180,in=50] (-0.3,-0.2) to[out=230,in=60] (-2.02,-2.76) to[out=240,in=160] (-1.77,-3.49) to[out=340,in=248] (-0.98,-3.24) to[out=68,in=253] (-0.17,-0.7) to[out=73,in=180] (0,-0.6) to[out=0,in=107] (0.17,-0.7) to[out=287,in=112] (0.98,-3.24) to[out=292,in=200] (1.77,-3.49) to[out=20,in=300] (2.02,-2.76) to[out=120,in=310] (0.3,-0.2) to[out=130,in=0] (0,0);
\node[darkgray] at (-1.5,-1.4) {$\gamma_{i,j}$};
\end{tikzpicture}
\end{center}

We now define the braiding of the $i$th and the $j$th boundary component in the same way as we defined Dehn twists in Paragraph~\ref{Dehn}:
\begin{Definition}
We define $\fb_{i,j}\colon \Sigma_{g,n} \to \Sigma_{g,n}$ as the map
\[p \mapsto 
\begin{cases}
(\varphi_{i,j} \circ \fb \circ \varphi_{i,j}^{-1})(p) &:  p \in \varphi_{i,j}(B) \\
p &:  p \notin \varphi_{i,j}(B)
\end{cases}\] 
\end{Definition}

Again, this map is a homeomorphism, not a diffeomorphism, but as for Dehn twists, it nonetheless determines an element~$[\fb_{i,j}]$ in the mapping class group, called the braiding of the $i$th and the $j$th boundary component. It is obvious that under the 
projection~$p \colon \Gamma_{g,n} \to S_n$ to the symmetric group introduced in Paragraph~\ref{MapClGr}, $[\fb_{i,j}]$ maps to the transposition of~$i$ and~$j$. 

We note that a mapping class very similar to~$[\fb_{i,j}]$, using punctures instead of boundary components, is called a half-twist in \cite[Sec.~9.4, p.~255]{FM}. There are several relations between half-twists and Dehn twists; one is explained right there, another one will be discussed in Paragraph~\ref{MoreSpecDehn}.

\subsection{The action on the fundamental group} \label{FundGroup}
During certain steps of our argument, it will be necessary to cut out small disks from our surface or at least to avoid certain points. For a Dehn twist, it makes a difference how we avoid these points. Suppose that a homotopy class is represented by a smooth simple closed curve~$\gamma$ that starts and ends at a point~$x$ on our surface. We choose a neighborhood~$U$ of~$x$ that is diffeomorphic to the open unit disk in~$\R^2$ via an orientation-preserving diffeomorphism~$\phi$ that sends~$x$ to the origin. As before, we assume that the unit disk has the orientation in which the canonical basis is a positively oriented basis for the tangent space at the origin. By perhaps passing to a smaller neighborhood, we can assume that~$\gamma$ passes through the unit disk as indicated on the left of the picture below. Because~$\phi$ is orientation-preserving, it is meaningful to talk about pushing off~$\gamma$ to the left or to the right, as indicated in the middle and right picture below:
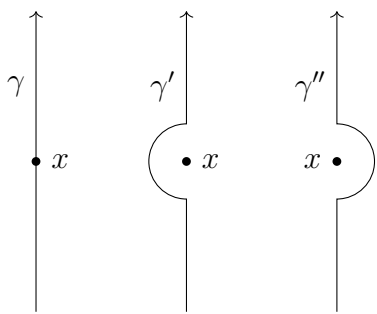
\begin{center}
\begin{tikzpicture}
\node[circle,fill,scale=0.3,label=right:$x$] (x1) at (0,0) {};
\node[circle,fill,scale=0.3,label=right:$x$] (y1) at (2,0) {};
\node[circle,fill,scale=0.3,label=left:$x$] (x2) at (4,0) {};
\draw[->] (0,-2) -- (0,2) node[left,yshift=-1cm]{$\gamma$};
\draw[->] (2,-2) -- (2,-0.5) arc(270:90:0.5) -- (2,2)
node[left,yshift=-1cm]{$\gamma'$};
\draw[->] (4,-2) -- (4,-0.5) arc(-90:90:0.5) -- (4,2)
node[left,yshift=-1cm]{$\gamma''$};
\end{tikzpicture}
\end{center}
In this way, we obtain two curves~$\gamma'$ and~$\gamma''$ whose free homotopy classes  in~$\Sigma_{g,n}\setminus\{x\}$ do not depend on our choices. Both curves avoid a small neighborhood of the point~$x$ and can be chosen so that they are still simple closed curves. Although we have in general~$[\fd_{\gamma'}] \neq [\fd_{\gamma''}] \in \Gamma_{g,n}(x)$, we clearly have
\vspace{1mm}
\[F_x([\fd_{\gamma'}]) =  F_x([\fd_{\gamma''}]) = [\fd_{\gamma}]\]
\vspace{1mm}
for their images under the forgetful map~$F_x\colon \Gamma_{g,n}(x) \to \Gamma_{g,n}$.

\vspace{1mm}

As explained in~\cite[Par.~8.2.7, p.~235]{FM}, we can choose for~$x$ the base point of the fundamental group to get an action of~$\Gamma_{g,n}(x)$ on~$\pi_1(\Sigma_{g,n},x)$. In particular, the Dehn twists~$\fd_{\gamma'}$ and~$\fd_{\gamma''}$ act on~$\pi_1(\Sigma_{g,n},x)$. We now give explicit formulas for this action in the case where the homotopy class acted upon can be represented by a curve~$\kappa$ that does not intersect~$\gamma$, except at the base point~$x$. It should be noted that, in the present situation, the curves~$\kappa$ and~$\gamma$, and therefore also~$\gamma'$ and~$\gamma''$, are oriented, although, according to our construction, Dehn twists like~$\fd_{\gamma'}$ do not depend on the orientation of~$\gamma'$. In each homotopy class, we can choose smooth representatives (cf.~\cite[Par.~1.2.2, p.~26]{FM}), which by definition of smoothness have a nowhere vanishing tangent vector. If we in addition can choose these representatives so that their tangent vectors~$\dot{\kappa}_x$ and~$\dot{\gamma}_x$ at~$x$ are linearly independent, we say that~$\kappa$ and~$\gamma$ intersect transversely at~$x$. If~$\dot{\kappa}_x$ and~$\dot{\gamma}_x$ form a positively oriented basis of the tangent space at~$x$, we define the algebraic intersection number~$i_A(\kappa, \gamma)$ as~$+1$, and if they form a negatively oriented basis of the tangent space at~$x$, we define the algebraic intersection number~$i_A(\kappa, \gamma)$ as~$-1$. This definition can be extended to curves with finitely many intersection points by defining 
$i_A(\kappa, \gamma)$ as the sum of the so-determined signs at all intersection points. This definition is illustrated in the picture

\vspace{1mm}

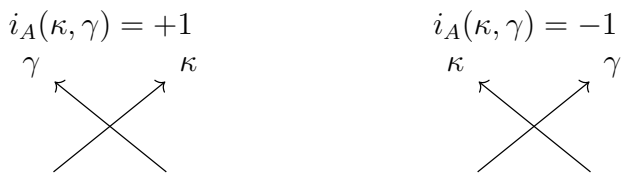
\begin{center}
\begin{tabular}{llllllll}
$i_A(\kappa, \gamma)= +1$ & & & & & & & $i_A(\kappa, \gamma)= -1$\\
\begin{tikzcd}[column sep = large, row sep = large]
\ar[from=2-2]\gamma & \ar[from=2-1]\kappa\\
\textcolor{white}{\gamma} & \textcolor{white}{\kappa}
\end{tikzcd} & & & & & & & \begin{tikzcd}[column sep = large, row sep = large]
\ar[from=2-2]\kappa & \ar[from=2-1]\gamma\\
\textcolor{white}{\gamma} & \textcolor{white}{\kappa}
\end{tikzcd}
\end{tabular}
\end{center}

in which the standard orientation of the plane has been used, i.e., the orientation in which the canonical basis is positive and the normal vector points towards the reader.

\pagebreak

The action of the two Dehn twists on the homotopy class~$[\kappa]$ of~$\kappa$ is summarized in the following table:
\vspace{-3mm}
\renewcommand{\arraystretch}{1.5}
\begin{center}
\begin{tabular}{r|cc}
& $i_A(\kappa,\gamma)=1$ & $i_A(\kappa,\gamma)=-1$ \\ 
\hline 
$\fd_{\gamma'}([\kappa])$ & $[\kappa \gamma]$ & $[\gamma^{-1} \kappa]$  \\ 
$\fd_{\gamma''}([\kappa])$ & $[\gamma \kappa]$ & $[\kappa \gamma^{-1}]$  
\end{tabular}
\end{center}

Here, expressions like $[\kappa \gamma] = [\kappa] [\gamma]$ mean the product in the fundamental group: The concatenation of first~$\kappa$, followed by~$\gamma$. To understand the table, let us consider how the entry for $\fd_{\gamma'}([\kappa])$ in the case $i_A(\kappa,\gamma)=1$ comes about: The curve starts at~$x$ with~$\kappa$. In view of how~$\gamma'$ has been pushed off~$\gamma$ and how the curves intersect, $\kappa$~is moving away from~$\gamma'$ and traces out~$\kappa$ almost completely before approaching the base point~$x$ upon returning. But before reaching~$x$, it encounters~$\gamma'$ and at this point is forced to turn left in view of our definition of a Dehn twist. But because $i_A(\kappa,\gamma)=1$, turning left is in this case the same as following~$\gamma$ in the direction of its orientation. The other entries arise from analogous considerations.

It is of course possible that we cannot choose representatives whose tangent vectors~$\dot{\kappa}_x$ and~$\dot{\gamma}_x$ at~$x$ are linearly independent. In this case, the two curves do not intersect transversely, and by pushing off~$\gamma$ either to the left or to the right, we can avoid intersection completely. Depending on the orientation of~$\gamma$, the situation is described by one of the following two pictures:
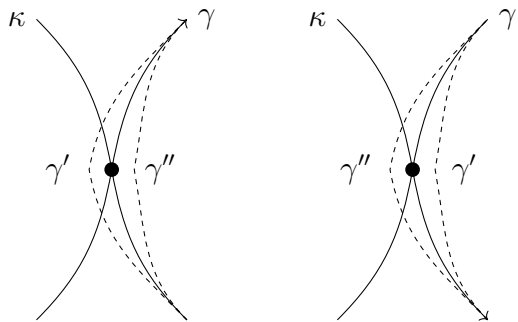
\vspace{-1mm}
\begin{center}
\begin{tikzpicture}
\node[circle,fill,scale=0.5] (x1) at (0,0) {};
\node[circle,fill,scale=0.5] (y1) at (4,0) {};
\draw[-] (-1,-2) to [out=45, in=-100] (0,0) to [out=100, in=315](-1,2) node[left]{$\kappa$};
\draw[->] (1,-2)  to [out=135, in=-80] (0,0) to [out=80, in=-135] (1,2) node[right]{$\gamma$};
\draw[dash pattern= on 2pt off 2pt] (1,-2)  to [out=135, in=-80] (-0.3,0) to [out=80, in=-135] (1,2) node[left,yshift=-2cm,xshift=-1.4cm]{$\gamma'$};
\draw[dash pattern= on 2pt off 2pt] (1,-2)  to [out=135, in=-80] (0.3,0) to [out=80, in=-135] (1,2) node[left,yshift=-2cm]{$\gamma''$};
\draw[-] (3,-2) to [out=45, in=-100] (4,0) to [out=100, in=315] (3,2) node[left]{$\kappa$};
\draw[<-] (5,-2)  to [out=135, in=-80] (4,0) to [out=80, in=-135] (5,2) node[right]{$\gamma$};
\draw[dash pattern= on 2pt off 2pt] (5,-2)  to [out=135, in=-80] (3.7,0) to [out=80, in=-135] (5,2) node[left,yshift=-2cm,xshift=-1.4cm]{$\gamma''$};
\draw[dash pattern= on 2pt off 2pt] (5,-2)  to [out=135, in=-80] (4.3,0) to [out=80, in=-135] (5,2) node[left,yshift=-2cm]{$\gamma'$};
\end{tikzpicture}
\end{center}
\vspace{-1mm}
In the first case, the actions of the Dehn twists are given by
\[\fd_{\gamma'}([\kappa]) = [\gamma^{-1} \kappa \gamma] \qquad \text{and} \qquad  
\fd_{\gamma''}([\kappa]) = [\kappa],\]  
while in the second case they are given by 
\[\fd_{\gamma'}([\kappa]) = [\kappa] \qquad \text{and} \qquad  
\fd_{\gamma''}([\kappa]) = [\gamma \kappa \gamma^{-1}].\]  
In both cases, these formulas are valid regardless of the orientation of~$\kappa$.

\subsection{Dehn twists for special curves} \label{SpecDehn}
In Paragraph~\ref{Dehn}, we have already introduced the notation $\fd_j$ for the Dehn twist determined by the curve~$\partial_j$. We now also assign names to the Dehn twists derived from the other curves stemming from the polygon model discussed in Paragraph~\ref{Surf} and define
\[\ft_i \deq \fd_{\alpha_i}, \qquad \qquad \fr_i \deq \fd_{\beta_i}, \qquad \qquad  \fn_i \deq \fd_{\mu_i}. \]
Because the curves~$\alpha_i$ and~$\beta_i$ intersect exactly once, the corresponding Dehn twists satisfy the braid relation $[\ft_i \fr_i \ft_i] = [\fr_i \ft_i \fr_i]$ 
(cf.~\cite[Par.~3.5.1, Prop.~3.11, p.~77]{FM}). For the elements
$\fs_i \deq \ft_i^{-1} \fr_i^{-1} \ft_i^{-1}$,
the braid relation implies that $[\fs_i \ft_i \fs_i^{-1}] = [\fr_i]$. 

These special Dehn twists are important in view of the following result, which is known as the Dehn-Lickorish theorem:
\begin{Theorem} \label{DehnLick} 
If $g \ge 1$, the Dehn twists $[\ft_i]$, $[\fr_i]$, and $[\fn_l]$, for $i=1,\dots,g$ and $l=1,\dots, g-1$, which are called the Lickorish generators, generate the mapping class groups~$\Gamma_{g} = \Gamma_{g,0}$ and~$\Gamma_{g,1}$.  
\end{Theorem}\vspace{-0.2cm}
A proof of this result can be found in \cite[Par.~4.4.4, p.~113ff]{FM}. We note that in the case $g=0$, the groups~$\Gamma_{g}$ and~$\Gamma_{g,1}$ are trivial anyway as a consequence of the Alexander lemma (cf.~\cite[Par.~2.2.1, p.~47ff]{FM}). The case $g=1$ will be discussed in more detail in Paragraph~\ref{Torus}. 

In the case where there is more than one boundary component, we need additional Dehn twists around the curves $\zeta_l$, which we denote by $\fz_l$ for $l=1,\dots,n-1$. The curve $\zeta_l$ separates the $l$th and $(l+1)$st boundary component and connects to the $g$th handle as shown in the picture for the case $g=3$ and $n=2$:
\vspace{-0.5cm}
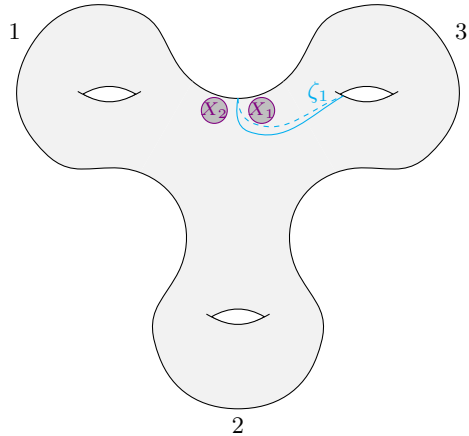
\begin{center}
\fontsize{9}{9}\selectfont
\begin{tikzpicture}[scale=0.35]
\pic[scale=0.35] (lower) at (0,-pi) {handle}; 
\pic[rotate=120,scale=0.35] (tr) at (30:pi) {handle};
\pic[rotate=-120,scale=0.35] (tl) at (150:pi) {handle};
\fill[gray!10]  (lower-right) to[out=100,in=200] (tr-left)-- 
(tr-right) to[out=-140,in=-40] (tl-left)
-- (tl-right) to[out=-20,in=80] (lower-left) -- cycle;
\draw (lower-right) to[out=100,in=200] (tr-left);
\draw (tr-right) to[out=-140,in=-40] (tl-left);
\draw (tl-right) to[out=-20,in=80] (lower-left);

\node[below] at (lower-num) {2}; 
\node[above right] at (tr-num) {3}; 
\node[above left] at (tl-num) {1}; 
\fontsize{7}{7}\selectfont
\draw[violet,rotate=-180, 
postaction={decorate},fill=lightgray] (-0.9,-2.2)  circle (14pt) node {$X_1$} ;
\draw[violet,rotate=-180,
postaction={decorate},fill=lightgray] (0.9,-2.2) circle (14pt) node {$X_2$} ;
\fontsize{9}{9}\selectfont
\draw[cyan] (0,2.67) to[out=260, in=160] (0.4,1.4) to[out=340, in=210]  (4,2.75);
\node[left,cyan] at (3.8,2.9) {$\zeta_1$};
\draw[cyan,dash pattern=on 2pt off 2pt] (0,2.67) to[out=270, in=160] (0.6,1.7) to[out=340, in=200] (4,2.75);
\end{tikzpicture}
\end{center}

For this case, we have the following variant of the Dehn-Lickorish theorem (cf.~\cite[Par.~4.4.4, p.~113f]{FM}): 
\begin{Theorem}
If $g \geq 1$ and $n \geq 2$, the mapping classes $[\ft_i]$, $[\fr_i]$, $[\fn_j]$, $[\fd_k]$, $[\fz_{l}]$, and~$[\fb_{l,l+1}]$, for $i=1,\dots,g$, $j=1,\dots, g-1$, $k=1,\dots,n$, and $l=1,\dots,n-1$, generate the mapping class group $\Gamma_{g,n}$.
\end{Theorem}
We can assume without loss of generality that the curve~$\partial_j$ does not only avoid the~$n$~boundary components, but also the base point~$x$. Then the corresponding Dehn twist~$\fd_j$ is contained in~$\Gamma_{g,n}(x)$ and therefore can act on~$\pi_1(\Sigma_{g,n},x)$, as explained in Paragraph~\ref{FundGroup}. Since any homotopy class in~$\pi_1(\Sigma_{g,n},x)$ can be represented by a curve that does not intersect~$\partial_j$, this action is trivial. However, the curves $\alpha_i$, $\beta_i$, and~$\mu_i$ start and end at the base point~$x$, and as explained in Paragraph~\ref{FundGroup}, it makes a difference for the action on~$\pi_1(\Sigma_{g,n},x)$ whether the curves meet before or after the base point~$x$. We will now make specific choices for these curves. For $\ft_i = \fd_{\alpha_i}$, we choose $\ft'_i \deq \fd_{\alpha'_i}$, for $\fr_i = \fd_{\beta_i}$, we choose $\fr''_i \deq \fd_{\beta''_i}$, and for $\fn_i = \fd_{\mu_i}$, we choose $\fn''_i \deq \fd_{\mu''_i}$. For~$\fs_i$, we mix the choices and define $\fs'_i \deq \ft_i'^{-1} \fr_i''^{-1} \ft_i'^{-1}$. These maps act on the generators of the fundamental group as follows:
\begin{Proposition} \label{MapClassCyc1}
Suppose that $i \le g$. \vspace{-10pt}
\begin{enumerate}
\item 
For $j \le n$, we have $\ft'_i([\delta_j]) = \fr''_i([\delta_j]) = [\delta_j]$. If $i \neq g$, we also have $\fn''_i([\delta_j]) = [\delta_j]$.

\item 
We have $\ft'_i([\beta_i]) = [\beta_i \alpha_i]$ and $\ft'_i([\alpha_i]) = [\alpha_i]$. For $j \neq i$, we have $\ft'_i([\beta_j]) = [\beta_j]$ and $\ft'_i([\alpha_j]) = [\alpha_j]$.

\item 
We have $\fr''_i([\alpha_i]) = [\alpha_i \beta_i^{-1}]$ and $\fr''_i([\beta_i]) = [\beta_i]$. For $j \neq i$, we have $\fr''_i([\alpha_j]) = [\alpha_j]$ and $\fr''_i([\beta_j]) = [\beta_j]$.

\item 
For $i=1,\ldots, g-1$, we have $\fn''_i([\alpha_i]) = [\alpha_i]$ as well as
\[\fn''_i([\beta_i]) = [\mu_i \beta_i], \quad 
\fn''_i([\alpha_{i+1}]) = [\mu_i \alpha_{i+1} \mu_i^{-1}], \quad 
\fn''_i([\beta_{i+1}]) = [\beta_{i+1} \mu_i^{-1}].\]
For $j \neq i$ and $j \neq i+1$, we have $\fn''_i([\alpha_j]) = [\alpha_j]$ and $\fn''_i([\beta_j]) = [\beta_j]$. 

\item 
We have
$\fs'_i([\alpha_i]) = [\alpha_i \beta_i \alpha_i^{-1}]$ and $\fs'_i([\beta_i]) = [\alpha_i^{-1}]$. 
For $j \neq i$, we have that \mbox{$\fs'_i([\alpha_j]) = [\alpha_j]$} and $\fs'_i([\beta_j]) = [\beta_j]$.
\end{enumerate}
\end{Proposition}
\begin{proof}
\begin{parlist}
\item 
We first note that~$\alpha_i$ and~$\beta_i$ are oriented so that~$\alpha'_i$ and~$\beta''_i$ arise by pushing~$\alpha_i$ and~$\beta_i$ off the base point~$x$ in the direction of the $i$th attached torus, so that we could cut the $i$th attached torus off again in such a way that the base point~$x$ on the one hand and $\alpha'_i$ and $\beta''_i$ on the other hand would lie on different connected components. If $j \neq i$, this implies that the curves $\alpha'_i$ and $\beta''_i$ do not intersect the curves~$\alpha_j$ and~$\beta_j$. This implies that $\ft'_i([\beta_j]) = [\beta_j]$ and $\ft'_i([\alpha_j]) = [\alpha_j]$ as well as $\fr''_i([\alpha_j]) = [\alpha_j]$ and $\fr''_i([\beta_j]) = [\beta_j]$. For the same reason, $\ft'_i$ and $\fr''_i$ preserve the homotopy class of~$\delta_j$.

\item
We have $i_A(\beta_i,\alpha_i) = -i_A(\alpha_i,\beta_i) = 1$, so that it follows from the table in Paragraph~\ref{FundGroup} that $\ft'_i([\beta_i]) = [\beta_i \alpha_i]$ and $\fr''_i([\alpha_i]) = [\alpha_i \beta_i^{-1}]$.  Clearly, we have $\ft'_i([\alpha_i]) = [\alpha_i]$ and~$\fr''_i([\beta_i]) = [\beta_i]$. Therefore, the second and the third assertions are now completely proved. 

\item
For the fourth assertion, we see from the discussion in Paragraph~\ref{FundGroupGen} that~$\mu_i$ is oriented in such a way that~$\mu''_i$ arises by pushing~$\mu_i$ off the base point~$x$ in the direction of the $i$th and the $(i+1)$st attached tori, so that we could cut these two tori off again in such a way that the base point~$x$ and $\mu''_i$ would lie on different connected components. For $j  \neq i$ and $j \neq i+1$, this implies that $\mu''_i$ does not intersect~$\alpha_j$ or~$\beta_j$, so that $\fn''_i([\alpha_j]) = [\alpha_j]$ and $\fn''_i([\beta_j]) = [\beta_j]$. This discussion also shows that~$\mu''_i$ does not intersect~$\alpha_i$, so that $\fn''_i([\alpha_i]) = [\alpha_i]$. On the other hand, $\mu_i$ intersects~$\beta_i$ and~$\beta_{i+1}$ exactly once, the intersection is transversal, and the intersection numbers are $i_A(\beta_i, \mu_i) = 1$ and $i_A(\beta_{i+1}, \mu_i) = -1$, respectively. It therefore follows from the table in Paragraph~\ref{FundGroup} that
\[\fn''_i([\beta_i]) = [\mu_i \beta_i] \quad \text{and} \quad 
\fn''_i([\beta_{i+1}]) = [\beta_{i+1} \mu_i^{-1}].\]
In contrast, the curves $\mu_i$ and $\alpha_{i+1}$ do not intersect transversally, but rather as in the second picture in Paragraph~\ref{FundGroup}, so that 
$\fn''_i([\alpha_{i+1}]) = [\mu_i \alpha_{i+1} \mu_i^{-1}]$. Since $i \neq g$, the curve $\mu''_i$ does not intersect~$\delta_j$, so that~$\fn''_i([\delta_j]) = [\delta_j]$. 

\item
It remains to show the fifth assertion. For $j \neq i$, the claim follows easily from the second and the third assertion. For the case $j=i$, we argue as follows: By inverting part of the third assertion, we have
$\fr_i''^{-1}([\alpha_i \beta_i^{-1}]) = [\alpha_i]$, which implies 
that~$\fr_i''^{-1}([\alpha_i]) [\beta_i]^{-1} = \fr_i''^{-1}([\alpha_i]) \fr_i''^{-1}([\beta_i]^{-1}) 
= [\alpha_i]$. This in turn yields
\[\fs'_i([\alpha_i]) = (\ft_i'^{-1} \fr_i''^{-1} \ft_i'^{-1})([\alpha_i])
= (\ft_i'^{-1} \fr_i''^{-1})([\alpha_i]) = \ft_i'^{-1}([\alpha_i \beta_i])
= \ft_i'^{-1}([\alpha_i]) \ft_i'^{-1}([\beta_i]).\]
A second inversion yields $\ft_i'^{-1}([\beta_i]) = [\beta_i \alpha_i^{-1}]$, so that \[\fs'_i([\alpha_i]) = [\alpha_i] \ft_i'^{-1}([\beta_i]) = [\alpha_i \beta_i \alpha_i^{-1}]\] as asserted. The formula for~$\fs'_i([\beta_i])$ follows from a similar computation. 
\qedhere
\end{parlist}
\end{proof}

\subsection{Dehn twists related to two boundary components} \label{MoreSpecDehn}
\enlargethispage{0.5cm}
According to our conventions in Paragraph~\ref{Dehn}, the curve $\delta_i \delta_j$, for~$i < j$, cannot be used to define a Dehn twist for two reasons: On the one hand, it intersects the boundary, and on the other hand, it is not a simple closed curve. The first problem can be addressed by moving it slightly to the interior with the help of a collar, in the same way as in our treatment of~$\partial_j$ at the end of Paragraph~\ref{Dehn}. As a consequence of the orientation of~$\rho_i$ and~$\rho_j$, a movement to the interior is a movement to the right of the given curve.

The second problem arises not only from the fact that the curve returns to the base point~$x$ at the time of concatenation, when~$\delta_i$ ends and~$\delta_j$ begins, but also from the fact that the paths~$\xi_i$ and~$\xi_j$ are both traced out twice, in opposite directions. To address this problem, we consider the curve~$\gamma_{i,j}$  defined in Paragraph~\ref{Braid}, which does not intersect the boundary, starts and ends in~$x$, and represents the same relative homotopy class as $\delta_i \delta_j$ in the fundamental group~$\pi_1(\Sigma_{g,n},x)$. Passing to~$\gamma_{i,j}'$ as indicated in Paragraph~\ref{FundGroup}, we in addition avoid the base point~$x$. In contrast to~$\delta_i \delta_j$, the modification~$\gamma_{i,j}'$ is a simple closed curve that can be used to define a Dehn twist.
This curve is shown in the picture that appears in the proof below. We denote the Dehn twist along~$\gamma_{i,j}'$ by~$\fd_{i,j}$.

As explained in \cite[Par.~5.1.1, p.~118f]{FM} or \cite[\S~7, p.~63]{PS}, the Dehn twist~$\fd_{i,j}$ is related to the braiding introduced in Paragraph~\ref{Braid}: If we define the double braiding as $\fq_{i,j} \deq \fb_{j,i} \fb_{i,j}$, it is related to our Dehn twist via the formula
\[ [\fq_{i,j}] = [\fd_{i,j} \fd_i^{-1} \fd_{j}^{-1}]. \]
This formula holds because the Dehn twist~$\fd_{i,j}$ not only interchanges the $i$th and the $j$th boundary components twice, but also twists these boundary components themselves, while the double braiding moves the boundary components in a parallel fashion, without introducing a twist. 

The action of~$\fq_{i,j}$ on the generators of the fundamental group is given as follows:
\begin{Proposition} \label{MapClassCyc2}
Suppose that $1 \le i < j \le n$. \vspace{-8pt}
\begin{enumerate}
\item 
We have \[\fq_{i,j}([\delta_i]) = [(\delta_i \delta_j)^{-1} \delta_i (\delta_i \delta_j)] 
= [\delta_j^{-1} \delta_i \delta_j] \quad \text{and} \quad 
\fq_{i,j}([\delta_j]) = [(\delta_i \delta_j)^{-1} \delta_j (\delta_i \delta_j)].\]

\item
If $l \neq i$ and $l \neq j$, we have 
$\fq_{i,j}([\delta_l]) = [\delta_l]$.

\item
The double braiding $\fq_{i,j}$ acts trivially on the homotopy classes of
$\alpha_1,\ldots,\alpha_g$ and $\beta_1,\ldots,\beta_g$ in $\pi_1(\Sigma_{g,n},x)$.
\end{enumerate}
\end{Proposition}
\pagebreak
\begin{proof}
We have already seen in Paragraph~\ref{SpecDehn} that the Dehn twists $\fd_i$ and~$\fd_{j}$ act trivially, so that the action of $\fq_{i,j}$ coincides with the action of~$\fd_{i,j}$. If $l=i$ or $l=j$, then $\gamma_{i,j}$ and~$\delta_l$ intersect only in~$x$, and the intersection is described by the first picture at the end of Paragraph~\ref{FundGroup}. Therefore, the discussion of the first case there yields that
$\fd_{i,j}([\delta_l]) = [\gamma_{i,j}^{-1} \delta_l \gamma_{i,j}]$. Because~$\gamma_{i,j}$ is homotopic 
to~$\delta_i \delta_j$ relative to the base point~$x$, this proves the first assertion.

If $l<i$ or $l>j$, the generator~$\delta_l$ also intersects~$\gamma_{i,j}$ only in~$x$, but now the intersection is described by the second picture at the end of Paragraph~\ref{FundGroup}, so that~$\gamma_{i,j}'$ does not intersect~$\delta_l$, which implies the second assertion in these cases. The case $i < l < j$ is more complicated and is illustrated by the picture
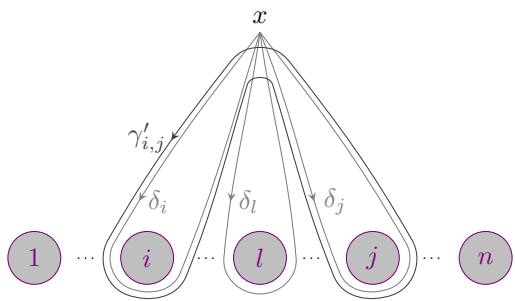
\begin{center} 
\fontsize{10}{12}\selectfont
\begin{tikzpicture}
\node[above] (v1) at (0,0) {$x$}; 

\draw[violet,fill=lightgray] (-3,-3)  circle (10pt) node {1} ;
\draw[violet,fill=lightgray] (-1.5,-3)  circle (10pt) node {$i$} ;
\draw[violet,fill=lightgray] (0,-3) circle (10pt) node {$l$} ;
\draw[violet,fill=lightgray] (1.5,-3) circle (10pt) node {$j$} ;
\draw[violet,fill=lightgray] (3,-3)  circle (10pt) node {$n$} ;

\node[scale=0.6] at (-2.3,-3) {\dots};
\node[scale=0.6] at (-0.7,-3) {\dots};
\node[scale=0.6] at (0.7,-3) {\dots};
\node[scale=0.6] at (2.3,-3) {\dots};

\draw[gray,decoration={ markings, mark=at position 0.3 with {\arrow[>=stealth]{>}}}, postaction={decorate}] (0,0) to[out=262,in=85] (-0.48,-3) to[out=265,in=180] (0,-3.48) to[out=0,in=275] (0.48,-3) to[out=95,in=278] (0,0);

\draw[gray,decoration={ markings, mark=at position 0.333 with {\arrow[>=stealth]{>}}}, postaction={decorate}] (0,0) to[out=232,in=60] (-1.94,-2.8) to[out=240,in=160] (-1.72,-3.41) to[out=340,in=248] (-1.06,-3.2) to[out=68,in=253] (0,0);

\draw[gray,decoration={ markings, mark=at position 0.285 with {\arrow[>=stealth]{>}}}, postaction={decorate}] (0,0) to[out=292,in=112] (1.06,-3.2) to[out=292,in=200] (1.72,-3.41) to[out=20,in=300] (1.94,-2.8) to[out=120,in=308] (0,0);

\draw[darkgray,decoration={ markings, mark=at position 0.11 with {\arrow[>=stealth]{>}}}, postaction={decorate}] (0,-0.2) to[out=180,in=52] (-0.4,-0.4) to[out=232,in=60] (-2.02,-2.76) to[out=240,in=160] (-1.77,-3.49) to[out=340,in=248] (-0.98,-3.24) to[out=68,in=253] (-0.17,-0.7) to[out=73,in=180] (0,-0.6) to[out=0,in=107] (0.17,-0.7) to[out=287,in=112] (0.98,-3.24) to[out=292,in=200] (1.77,-3.49) to[out=20,in=300] (2.02,-2.76) to[out=120,in=308] (0.4,-0.4) to[out=128,in=0] (0,-0.2);

\node[darkgray] at (-1.5,-1.4) {$\gamma_{i,j}'$};
\node[gray] at (-1.35,-2.25) {$\delta_i$};
\node[gray] at (-0.15,-2.25) {$\delta_l$};
\node[gray] at (1.0,-2.25) {$\delta_j$};
\end{tikzpicture}
\end{center}

in which we have replaced the curves~$\delta_i$, $\delta_l$, and~$\delta_j$ by slight modifications within their relative homotopy class. It shows that the curve 
$\fd_{i,j}(\delta_l)$ comes about as follows: It starts at~$x$ with~$\delta_l$, but then soon meets~$\gamma_{i,j}'$. At that point, it turns left and follows~$\gamma_{i,j}'$ against its orientation. Upon returning, it follows~$\delta_l$ again, but soon encounters~$\gamma_{i,j}'$ a second time. Again, it turns left, but now follows~$\gamma_{i,j}'$ in the direction of its orientation. Up to relative homotopy, these two encounters with~$\gamma_{i,j}'$ cancel each other. While continuing along~$\delta_l$, the curve has two similar encounters with~$\gamma_{i,j}'$, which again cancel each other. In summary, the curve is homotopic to~$\delta_l$, which finishes the proof of the second assertion. 

The curves $\alpha_i$ and $\beta_i$ intersect~$\gamma_{i,j}$ again only in~$x$, and the intersection is again described by the second picture at the end of Paragraph~\ref{FundGroup}, so that~$\gamma_{i,j}'$ does not intersect them at all. This yields the third assertion. 
\end{proof}

\subsection{The capping homomorphism} \label{Cap}
In the process of defining~$\Sigma_{g,n}$, we have removed $n$~open disks from~$\Sigma_{g}$. If we assume that our surface is realized via the polygon model, the boundary components come with a natural enumeration, and the $j$th boundary component is parametrized by~$\rho_j$. In the place of the $j$th removed disk, we now glue a punctured disk back, i.e., a disk whose interior contains a marked point denoted by~$y$. If a homeomorphism of~$\Sigma_{g,n}$ restricts to the identity on the image of~$\rho_j$, we can extend it to~$\Sigma_{g,n-1}$ by requiring that the extension restricts to the identity on the newly inserted punctured disk. In this way, we obtain a group homomorphism
\[C_j\colon \Gamma_{g,n}(\rho_j) \to \Gamma_{g,n-1}(y)\]
which we call the $j$th capping homomorphism. Here, we have not only written $\Gamma_{g,n-1}(y)$ for~$\Gamma_{g,n-1}(\{y\})$, as indicated in Paragraph~\ref{MapClGr}, but also 
briefly~$\Gamma_{g,n}(\rho_j)$ for~$\Gamma_{g,n}(\Img(\rho_j))$.

It is a consequence of the Alexander lemma (cf.~\cite[Sec.~2.2, Lem.~2.1, p.~47f]{FM}) that~$\fd_j$ is contained in the kernel of~$C_j$. The following result, which can be found in~\cite[Sec.~3.6, Prop.~3.19, p.~85]{FM} or~\cite[Sec.~3, p.~104f]{Ko}, states that~$\fd_j$ in fact generates the kernel:
\begin{Proposition} \label{CapSeq}
The sequence
\begin{equation*} \label{CapSeqEq}
1 \longrightarrow \langle \fd_{j} \rangle \longrightarrow \Gamma_{g,n}(\rho_j) \overset{C_j}{\longrightarrow} \Gamma_{g,n-1}(y) \longrightarrow 1
\end{equation*}
is short exact.
\end{Proposition} 

Almost always, $\fd_{j} \in \Gamma_{g,n}(\rho_j)$ is a nontrivial mapping class, in which case $\langle \fd_{j} \rangle$ is isomorphic to~$\Z$. The only exception is the case $g=0$ and~$n=1$: Here, it is easy to see that $\Gamma_{0,0}(y)$ is trivial (cf.~\cite[p.~49]{FM}), and~$\Gamma_{0,1}$ is trivial by the Alexander lemma. 

Clearly, we can compose $C_j$ with the forgetful map~$F_y\colon \Gamma_{g,n-1}(y) \to \Gamma_{g,n-1}$ to obtain a homomorphism $D_j \deq F_y \circ C_j$ from $\Gamma_{g,n}(\rho_j)$ to $\Gamma_{g,n-1}$
that caps off the boundary component not with a punctured, but rather with a full disk. The effect of the additional mapping~$F_y$ will be studied in the next paragraph.

\subsection{The Birman sequence} \label{Birm}
Because we can always modify a diffeomorphism by isotopy so that it fixes a point, the forgetful map 
\[F_x \colon \Gamma_{g,n}(x) \to \Gamma_{g,n}\]
is surjective. If $\phi \in \Diffeo^+(\Sigma_{g,n})$ not only fixes~$x$, but in addition represents a mapping class in the kernel of~$F_x$, it can be connected to the identity in $\Diffeo^+(\Sigma_{g,n})$, i.e., there is a path
\[ [0,1] \to \Diffeo^+(\Sigma_{g,n}),~t \mapsto \phi_t\]
with $\phi_0 = \id$ and $\phi_1 = \phi$. Since~$\phi_t$ is not required to fix the point~$x$, we get a closed path
\[\gamma\colon [0,1] \to \Sigma_{g,n},~t \mapsto \phi_t(x)\]
based at~$x$, which represents a homotopy class in the fundamental group $\pi_1(\Sigma_{g,n},x)$.

Using the long exact sequence of a fibration (cf.~\cite[Thm.~11.48, p.~358]{Ro1}), one can show that the homotopy class of~$\gamma$ determines the mapping class of~$\phi$. More precisely, there is a group antihomomorphism
\[P_x \colon \pi_1(\Sigma_{g,n},x) \to \Gamma_{g,n}(x)\]
that maps the class of~$\gamma$ to the class of~$\phi$. This map is called the point-pushing map, or briefly the pushing map. The arising exact sequence
\begin{equation*}\label{eq_Birman sequence}
\pi_1(\Sigma_{g,n},x) \overset{P_x}{\longrightarrow} \Gamma_{g,n}(x) \overset{F_x}{\longrightarrow} \Gamma_{g,n} \longrightarrow 1
\end{equation*}
is called the Birman sequence. If the Euler characteristic of~$\Sigma_{g,n}$ is strictly negative, the Birman sequence is in fact short exact, i.e.,~$P_x$~is injective. The Birman sequence is discussed in greater detail in~\cite[Sec.~4.2, p.~96ff]{FM}.

As we have explained in Paragraph~\ref{FundGroup}, the mapping class group $\Gamma_{g,n}(x)$ acts on the fundamental group~$\pi_1(\Sigma_{g,n},x)$, and so in particular $P_x([\gamma])$ acts on the fundamental group. If~$P_x([\gamma]) = [\phi]$ and~$[\beta] \in \pi_1(\Sigma_{g,n},x)$, it can be shown that~$\phi(\beta)$, i.e., the closed path $t \mapsto  \phi(\beta(t))$, is homotopic to $\gamma^{-1} \beta \gamma$ relative to the basepoint~$x$, so that 
\[P_x([\gamma])([\beta]) = [\gamma^{-1} \beta \gamma].\]
In other words, $P_x([\gamma])$ acts on the fundamental group by an inner automorphism. The Birman sequence therefore implies that we have a homomorphism from~$\Gamma_{g,n}$ to $\Out(\pi_1(\Sigma_{g,n},x))$ that makes the following diagram commutative:
\begin{center}
\begin{tikzcd}
\Gamma_{g,n}(x) \arrow{r}\arrow[d, twoheadrightarrow, "F_x"']  & \Aut(\pi_1(\Sigma_{g,n},x))\arrow[d, twoheadrightarrow]\\
\Gamma_{g,n}\arrow{r} & \Out(\pi_1(\Sigma_{g,n},x))
\end{tikzcd}
\end{center}
Further aspects of this diagram are discussed in~\cite[Par.~8.2.7, p.~235]{FM}.
\pagebreak

The action of the mapping class group $\Gamma_{g,n}(x)$ on the fundamental group~$\pi_1(\Sigma_{g,n},x)$ is compatible with the pushing map in another way: Clearly, $\Gamma_{g,n}(x)$ acts on the kernel of the group homomorphism~$F_x$ by conjugation. This action is compatible with the pushing map in the sense that 
\[P_x([\psi(\gamma)]) = [\psi] P_x([\gamma]) [\psi]^{-1}\]
for $[\psi] \in \Gamma_{g,n}(x)$ (cf.~\cite[Par.~4.2.2, Fact~4.8, p.~99]{FM}). This relation is a direct consequence of our description above.

If a homotopy class in the fundamental group can be represented by a simple closed curve~$\gamma$, there is a formula for its image under the pushing map in terms of Dehn twists: Using the notation from Paragraph~\ref{FundGroup}, we have
\[P_x([\gamma]) = [\fd_{\gamma'} \fd_{\gamma''}^{-1}]\]
(cf.~\cite[Par.~4.2.2, Fact~4.7, p.~99]{FM}).

\subsection{Singular homology} \label{Homol}
If we think of a closed curve as being defined on the interval~$[0,1]$, it can be considered as a singular cycle. In this way, we obtain a group homomorphism
\[\pi_1(\Sigma_{g,n},x) \to H_1(\Sigma_{g,n}, \Z)\]
from the fundamental group to the first singular homology group that is called the (first) Hurewicz homomorphism. Hurewicz' theorem asserts in this situation that the Hurewicz homomorphism is surjective and induces an isomorphism from the commutator factor group of the fundamental group to the first singular homology group (cf.~\cite[Thm.~4.29, p.~83]{Ro1}). 

It follows from its universal property that the commutator factor group of a finitely presented group is presented by the same relations, but now understood as a presentation of an abelian group. Therefore, Hurewicz' theorem and the presentations of the fundamental group discussed in Paragraph~\ref{FundGroupGen} together imply that $H_1(\Sigma_{g}, \Z)$ is a free abelian group of rank~$2g$ with the homology classes of the curves 
$\alpha_1, \beta_1, \alpha_2, \beta_2,  \ldots, \alpha_g, \beta_g$ as a basis. In the presence of boundary components, Hurewicz' theorem yields that $H_1(\Sigma_{g,n}, \Z)$ is generated by the homology classes of the curves $\alpha_1, \beta_1, \alpha_2, \beta_2, \ldots, \alpha_g, \beta_g, \delta_1, \ldots, \delta_n$ subject to the relation that the homology classes of the curves $\delta_1, \ldots, \delta_n$ sum up to zero. By solving for one of these generators, we see that the first homology group is free abelian of rank $2g+n-1$. We note that the homology class of~$\delta_j$ is equal to the homology class of~$\partial_j$, because these curves are freely homotopic.

The algebraic intersection number introduced in Paragraph~\ref{FundGroup} depends only on the homology classes of the curves involved. It therefore defines a skew-symmetric bilinear form on the abelian group~$H_1(\Sigma_{g,n},\Z)$. For the generators, we have on the one hand
\[i_A(\beta_i, \alpha_j) = -i_A(\alpha_j, \beta_i) = \delta_{i,j} \quad \text{and} \quad
i_A(\alpha_i, \alpha_j) = 0 = i_A(\beta_i, \beta_j)\]
(cf.~\cite[Par.~6.1.2, p.~165]{FM}), while on the other hand we have
$i_A(\delta_j, \gamma) = 0$ for every closed curve~$\gamma$, because its free homotopy class has a representative that does not intersect~$\delta_j$. This shows in particular that this bilinear form is nondegenerate if and only if~$n=0$.

Because the outer automorphism group of any group evidently acts on the corresponding commutator factor group, Hurewicz' theorem implies that the commutative square obtained in Paragraph~\ref{Birm} can be enlarged to the diagram
\begin{center}
\begin{tikzcd}
\Gamma_{g,n}(x) \arrow{r}\arrow[d, twoheadrightarrow, "F_x"']  & \Aut(\pi_1(\Sigma_{g,n},x))
\arrow[d, twoheadrightarrow] \\
\Gamma_{g,n} \arrow{r} \arrow[equal]{d} & \Out(\pi_1(\Sigma_{g,n},x)) \arrow[d] \\
\Gamma_{g,n} \arrow{r} & \Aut(H_1(\Sigma_{g,n},\Z))
\end{tikzcd}
\end{center}

It is immediate from the construction that the action of~$\Gamma_{g,n}$ on~$H_1(\Sigma_{g,n}, \Z)$ preserves the intersection form and therefore takes values in the symplectic group. The discussion in Paragraph~\ref{SpecDehn} and Paragraph~\ref{MoreSpecDehn} also implies how the special mapping classes introduced there act on the first homology group. The mapping classes~$[\fd_j]$ and the mapping classes~$[\fq_{i,j}]$ act as the identity on the entire first homology group. From Proposition~\ref{MapClassCyc1}, we see that the mapping classes~$[\ft_i]$, $[\fr_i]$, and~$[\fs_i]$ preserve the two subgroups generated by the homology classes of the curves $\alpha_1, \beta_1, \alpha_2, \beta_2, \ldots, \alpha_g, \beta_g$ on the one hand and the curves $\delta_1, \ldots, \delta_n$ on the other hand, and act as the identity on the second subgroup. As an abelian group, the first homology group is the direct sum of these two subgroups, and the first subgroup is free of rank~$2g$ on the given generating set. We can therefore represent the action on the first subgroup by matrices in~$\Sp(2g,\Z)$. It follows from Proposition~\ref{MapClassCyc1} that these matrix representations are block-diagonal with~$g$~blocks of $2 \times 2$-matrices. For~$\ft_i$, $\fr_i$, and~$\fs_i$, the blocks corresponding to the generators~$\alpha_i$ and $\beta_i$ are
\[
\ft \deq \begin{pmatrix}
1 & 1\\
0 & 1
\end{pmatrix}
\qquad
\fr \deq 
\begin{pmatrix}
1 & 0\\
-1 & 1
\end{pmatrix}
\qquad
\fs \deq \begin{pmatrix}
0 &-1\\
1 & 0
\end{pmatrix}
\]
respectively, while the other blocks are $2 \times 2$-identity matrices. For $i=1,\ldots,g-1$, the matrix representation of~$\fn_i$ is also block-diagonal, but now contains a $4 \times 4$-block
\[\fn \deq
\begin{pmatrix}
1 & 1 & 0 & -1\\
0 & 1 & 0 & 0\\
0 & -1 & 1 & 1\\
0 & 0 & 0 & 1
\end{pmatrix}\]
corresponding to the homology classes of the curves $\alpha_i$, $\beta_i$, $\alpha_{i+1}$, and~$\beta_{i+1}$.

\subsection{The mapping class group of the torus} \label{Torus}
The surface~$\Sigma_{1}$ is a torus. Because the defining relation of the fundamental group discussed in Paragraph~\ref{FundGroupGen} in this case states that the relative homotopy class of 
$\alpha_1 \beta_1 \alpha_{1}^{-1} \beta_1^{-1}$ is trivial, the fundamental group is abelian here, and therefore Hurewicz' theorem mentioned in Paragraph~\ref{Homol} yields that it is isomorphic to the first singular homology group. According to the Dehn-Lickorish theorem~\ref{DehnLick}, the Dehn twists~$[\ft_1]$ and~$[\fr_1]$, or alternatively~$[\ft_1]$ and~$[\fs_1] = [\ft_1^{-1} \fr_1^{-1} \ft_1^{-1}]$, generate the mapping class group~$\Gamma_{1}$. As we have just seen, the actions of these generators on the first homology group are, respectively, represented by the matrices
\[
\ft = \begin{pmatrix}
1 & 1\\
0 & 1
\end{pmatrix}
\qquad
\fr = 
\begin{pmatrix}
1 & 0\\
-1 & 1
\end{pmatrix}
\qquad
\fs = \begin{pmatrix}
0 &-1\\
1 & 0
\end{pmatrix}
\]
with respect to the basis consisting of the singular homology classes of $\alpha_1$ and~$\beta_1$. 
Because~$\Sp(2g,\Z) = \SL(2,\Z)$ if~$g=1$, the group homomorphism $\Gamma_{g,n} \to \Sp(2g,\Z)$ described in Paragraph~\ref{Homol} becomes in this situation a group homomorphism
\[\Gamma_{1} \to \SL(2,\Z)\]
that maps~$[\ft_1]$, $[\fr_1]$, and~$[\fs_1]$ to~$\ft$, $\fr$, and~$\fs$, respectively. It is a classical fact that the modular group $\SL(2,\Z)$ is generated by the matrices~$\fs$ and~$\ft$ and that the relations \mbox{$\fs^4=1$} and $\fs \ft \fs = \ft^{-1} \fs \ft^{-1}$ that they satisfy are defining (cf.~\cite[Thm.~A.2, p.~312]{KT}). These relations imply that $\fs^2$ is central, which can also easily be seen directly. Alternatively, the modular group is generated by the elements~$\fr$ and~$\ft$, and the defining relations for~$\fs$ and~$\ft$ just stated translate into the defining relations $\ft \fr \ft = \fr \ft \fr$ and $(\fr\ft)^6 = 1$ for the generators~$\fr$ and~$\ft$ (cf.~\cite[Prop.~1.1, p.~7]{SZ}). This implies immediately that our group homomorphism $\Gamma_{1} \to \SL(2,\Z)$ is surjective, a fact that is also a special case of a more general result for mapping class groups (cf.~\cite[Par.~6.3.2, Thm.~6.4, p.~170]{FM}). In the case of the torus, the homomorphism is even bijective (cf.~\cite[Par.~2.2.4, Thm.~2.5, p.~53]{FM}), so that $\Gamma_{1} \cong \SL(2,\Z)$.

To discuss the mapping class group~$\Gamma_{1,1}$, we need to introduce the braid group~$B_3$ on three strands, which we define as the group generated by two generators~$r$ and~$t$ subject to the one defining relation \mbox{$rtr = trt$}. We will refer to this relation as the braid relation. Geometrically, $r$ can be interpreted as the interchange of the first two strands, while~$t$ can be interpreted as the interchange of the last two strands (cf.~\cite[Sec.~9.2, p.~246f]{FM}). In view of the second presentation of the modular group given above, there is a surjective group homomorphism
\[B_3 \to \SL(2,\Z)\]
that maps~$r$ to~$\fr$ and~$t$ to~$\ft$. If we define $s \deq (trt)^{-1}$, then on the one hand~$s$ is mapped to~$\fs$ under our homomorphism, and on the other hand the braid relation is equivalent to the relation~$sts = t^{-1} s t^{-1}$ for the generators~$s$ and~$t$ of~$B_3$. The braid relation implies relatively easily that the element $(rt)^3 = (rtr) (trt)$ is central in~$B_3$. This yields, again in view of the second presentation of the modular group, not only that~$(rt)^6$ is contained in the kernel of this homomorphism, but that it even generates the kernel. We therefore have the short exact sequence
\[1 \longrightarrow \langle (rt)^6 \rangle \longrightarrow B_3 \longrightarrow \SL(2,\Z) 
\longrightarrow 1.\]

We want to compare this homomorphism from~$B_3$ to~$\SL(2,\Z)$ with the homomorphism 
$D_1 \colon \Gamma_{1,1} = \Gamma_{1,1}(\rho_1) \to \Gamma_{1}$ that caps off the one boundary component with a full disk. As discussed in Paragraph~\ref{Cap}, $D_1$ is the composition of 
$C_1\colon \Gamma_{1,1} \to \Gamma_{1,0}(y)$, which caps off the boundary component with a punctured disk that contains a marked point~$y$ in its interior, and the forgetful map 
\[F_y\colon \Gamma_{1,0}(y) \to \Gamma_{1,0} = \Gamma_{1}.\]
At this point, there is a small conceptual difficulty with the polygon model. In the polygon model, the capped-off boundary component is parametrized by~$\rho_1$, and the point~$y$ in its interior is different from the base point~$x$ that is the common image of all the vertices of the polygon that are different from the start point and the end point of the edge labeled by~$\rho_1$. To account for this difference, we choose a diffeomorphism~$\phi\colon \Sigma_{1,0} \to \Sigma_{1,0}$ that is isotopic to the identity and satisfies~$\phi(x) = y$. We then have the isomorphism
\[ \Gamma_{1,0}(x) \to \Gamma_{1,0}(y),~[\psi] \mapsto [\phi \circ \psi \circ \phi^{-1}]\]
that satisfies $F_y([\phi \circ \psi \circ \phi^{-1}]) = F_x([\psi])$. Moreover, the map
\[\pi_1(\Sigma_{1,0},x) \to \pi_1(\Sigma_{1,0},y),~[\gamma] \mapsto [\phi \circ \gamma]\]
is an isomorphism of fundamental groups. Because the curves~$\gamma$ and~$\phi \circ \gamma$ are freely homotopic, they map to the same singular homology class under the respective Hurewicz maps. If~$\gamma$ is a smooth simple closed curve, we have 
\[ [\fd_{\phi \circ \gamma}] = [\phi] [\fd_{\gamma}] [\phi^{-1}] =  [\fd_{\gamma}] \]
in $\Gamma_{1}$. These considerations show that the diagram
\begin{center}
\begin{tikzcd}
\Gamma_{1,0}(y) \arrow{r} \arrow[d, twoheadrightarrow, "F_y"'] & \Aut(\pi_1(\Sigma_{1,0},y)) \arrow{r} & \Aut(H_1(\Sigma_{1,0},\Z)) \cong \GL(2,\Z)\arrow[equal]{d} \\
\Gamma_{1,0} \arrow{rr} &  & \Aut(H_1(\Sigma_{1,0},\Z)) 
\cong \GL(2,\Z)
\end{tikzcd}
\end{center}
commutes. Although it does not prove this fact, this diagram suggests that the map
$\Gamma_{1,0}(y) \to \Aut(H_1(\Sigma_{1,0},\Z))$ in the top row might be injective, and this indeed turns out to be the case (cf.~\cite[Par.~2.2.4, p.~54f]{FM}). Therefore the forgetful map~$F_y$ must be bijective in our present case. For the kernel of the homomorphism $D_1 = F_y \circ C_1$, we therefore have $\Ker(D_1) = \Ker(C_1) = \langle \fd_{1} \rangle$ by Proposition~\ref{CapSeqEq}. For the Birman sequence
\begin{equation*}
\pi_1(\Sigma_{1},y) \overset{P_y}{\longrightarrow} \Gamma_{1}(y) \overset{F_y}{\longrightarrow} \Gamma_{1} \longrightarrow 1
\end{equation*}
discussed in Paragraph~\ref{Birm}, this implies that the pushing map~$P_y$ vanishes identically, and is in particular not injective. This does not contradict the claims made there, because the Euler characteristic~$\chi(\Sigma_{1}) = 0$ is not strictly negative.

If we compose the monomorphism $\Gamma_{1,0}(y) \to \Aut(H_1(\Sigma_{1,0},\Z))$ just described with the capping map $C_1\colon \Gamma_{1,1} \to \Gamma_{1,0}(y)$, we obtain a homomorphism from $\Gamma_{1,1}$ to $\Aut(H_1(\Sigma_{1,0},\Z))$. This map can also be constructed in another way. As discussed in Paragraph~\ref{FundGroupGen}, the fundamental group~$\pi_1(\Sigma_{1,1},x)$ of a torus with one boundary component is a free (nonabelian) group whose generators are the relative homotopy classes~$[\alpha_1]$ and~$[\beta_1]$. But by Hurewicz' theorem, we then have that the first homology group~$H_1(\Sigma_{1,1},\Z)$ is again free abelian with the homology classes of~$\alpha_1$ and~$\beta_1$ as generators. This shows that the inclusion map~$\Sigma_{1,1} \to \Sigma_{1}$ induces an isomorphism between the first homology groups, so that also $\Aut(H_1(\Sigma_{1,1},\Z)) \cong \GL(2,\Z)$ with respect to these generators. In view of our considerations above, the Dehn twists~$\ft_1$ and~$\fr_1$ in~$\Gamma_{1,1}$ are then represented by the same matrices as the corresponding Dehn twists in~$\Gamma_{1}$. 

As we said at the beginning of Paragraph~\ref{SpecDehn}, the Dehn twists~$\fr_1$ and~$\ft_1$ satisfy the braid relation, so that we obtain a group homomorphism
\[B_3 \to \Gamma_{1,1}\]
that maps~$r$ to~$\fr_1$ and~$t$ to~$\ft_1$. By the Dehn-Lickorish theorem~\ref{DehnLick}, this homomorphism is surjective. Now the 2-chain relation (cf.~\cite[Par.~4.4.1, p.~107f]{FM}) states that 
$(\fr_1 \ft_1)^6 = \fd_1$ (cf.~\cite[Par.~3.4.1, Fig.~141, p.~124]{S} for an illustration). We therefore have the commutative diagram
\begin{center}
\begin{tikzcd}
1 \arrow[r] & \langle (rt)^6 \rangle \arrow[r] \arrow[d] & B_3 \arrow[r] \arrow[d] & \SL(2,\Z) \arrow[r]  \arrow[d, "\cong"] & 1 \\
1 \arrow[r] & \langle \fd_1 \rangle \arrow[r] & \Gamma_{1,1} \arrow[r, "D_1"] & \Gamma_{1} \arrow[r] & 1
\end{tikzcd}
\end{center}
in which the rightmost vertical map is the isomorphism described above. The five-lemma now implies that \mbox{$\Gamma_{1,1} \cong B_3$} (cf.~\cite[Par.~3.6.4, p.~87f]{FM}).

\subsection{The mapping class group of the sphere} \label{Sphere}
In the preceding paragraph, we have introduced the braid group on three strands. More generally, there is, for $n \ge 2$, a braid group~$B_n$ on $n$~strands, which can be defined as the group generated by elements~$\sigma_1,\ldots,\sigma_{n-1}$ subject to the relations
\[\sigma_i \sigma_{i+1} \sigma_i = \sigma_{i+1} \sigma_i \sigma_{i+1} \]
for $i=1,\ldots,n-2$ and $\sigma_i \sigma_j = \sigma_j \sigma_i$ if~$j>i+1$. 
Geometrically, $\sigma_i$ can be interpreted as the interchange of the $i$th and the $(i+1)$st strand (cf.~\cite[Sec.~9.2, p.~246f]{FM}). In the case $n=3$ considered in the preceding paragraph, we have $r = \sigma_1$ and~$t = \sigma_2$. 

From this definition, we see that there is a group homomorphism from the braid group~$B_n$ to the symmetric group~$S_n$ that maps~$\sigma_i$ to the transposition of~$i$ and~\mbox{$i+1$}. Pulling the natural action of~$S_n$ back along this homomorphism, we get an action of~$B_n$ on the set~$\{1,2,\ldots,n\}$ that we denote by $i \mapsto \sigma.i$. If we view~$\Z^n$ as the set of functions from~$\{1,2,\ldots,n\}$ to~$\Z$, we can form the corresponding wreath product, which we denote by~$\Z^n \rtimes B_n$. If $E_i = (e_i, 1_{B_n})$ denotes the $i$th canonical basis vector, considered as an element of the wreath product, and if we identify $\sigma \in B_n$ with 
$(0, \sigma) \in \Z^n \rtimes B_n$, we have the commutation relation $\sigma E_i = E_{\sigma.i} \sigma$. The wreath product is therefore generated by the elements~$\sigma_1,\ldots,\sigma_{n-1}$ together with~$E_1,\ldots,E_n$. Besides the defining relations of the braid group stated above, these generators satisfy the relations $E_i E_j = E_j E_i$ as well as 
\begin{center}
$\sigma_i E_j = 
\begin{cases}
E_{i+1} \sigma_i &: j=i \\
E_{i} \sigma_i &: j=i+1 \\
E_{j} \sigma_i &: j \neq i \; \text{and} \;j \neq i+1\\
\end{cases}$
\end{center}
It is not too complicated to show that these relations are defining.

It turns out that the mapping class group~$\Gamma_{0,n}$ is a quotient group of this wreath product. If we map~$\sigma_i$ to~$\fb_{i,i+1}$ and~$E_{i}$ to~$\fd_i$, we obtain a surjective group homomorphism from 
$\Z^n \rtimes B_n$ to~$\Gamma_{0,n}$. However, the map is not bijective, because the generators~$\fb_{i,i+1}$ and~$\fd_i$ satisfy additional relations: Besides the relations 
\[\fb_{i,i+1} \fb_{i+1,i+2} \fb_{i,i+1} = \fb_{i+1,i+2} \fb_{i,i+1} \fb_{i+1,i+2} \]
for $i=1,\ldots,n-2$ and $\fb_{i,i+1} \fb_{j,j+1} = \fb_{j,j+1} \fb_{i,i+1}$ if~$j>i+1$ for the braidings as well as $\fd_i \fd_j = \fd_j \fd_i$ and
\begin{center}
$\fb_{i,i+1} \fd_j = 
\begin{cases}
\fd_{i+1} \fb_{i,i+1} &: j=i \\
\fd_i \fb_{i,i+1} &: j=i+1 \\
\fd_j \fb_{i,i+1} &: j \neq i \; \text{and} \;j \neq i+1\\
\end{cases}$
\end{center}
for the Dehn twists and their interaction with the braidings, we have the additional relations
\[\fb_{1,2} \fb_{2,3} \cdots \fb_{n-1,n}^2 \cdots \fb_{2,3} \fb_{1,2} = \fd_1^{-2}\]
and $(\fb_{1,2} \fb_{2,3} \cdots \fb_{n-1,n})^n = \fd_1^{-1} \fd_2^{-1} \cdots \fd_n^{-1}$. These relations all together are again defining (cf.~\cite[Par.~4.2, p.~487]{L1}, see also \cite[Par.~3.1, p.~322f]{L2}).

Clearly, we can embed the wreath product~$\Z^{n-1} \rtimes B_{n-1}$ into the wreath 
product $\Z^n \rtimes B_n$ by sending~$\sigma_i$ to~$\sigma_i$ and~$E_i$ to~$E_i$. By composing this injection with the surjection just described, we get a group homomorphism from~$\Z^{n-1} \rtimes B_{n-1}$ to~$\Gamma_{0,n}$. It clearly takes values in the subgroup~$\Gamma_{0,n}(\rho_n)$ that fixes the last boundary component pointwise. The fundamental fact about this map is the following:
\begin{Proposition} \label{PropSphere}
The group homomorphism
\[\Z^{n-1} \rtimes B_{n-1} \to \Gamma_{0,n}(\rho_n)\]
that maps~$\sigma_i$ to~$\fb_{i,i+1}$ and~$E_i$ to~$\fd_i$ is an isomorphism.
\end{Proposition} 

This proposition follows from the fact that we can view a sphere with one boundary component as a closed disk, say of radius~$1$. The mapping class group~$\Gamma_{0,n}(\rho_n)$ is therefore isomorphic to the mapping class group of a disk with $n-1$~holes in which the diffeomorphisms are required to restrict to the identity on the boundary unit circle. One of the standard topological descriptions of the braid group yields that this mapping class group is isomorphic to $\Z^{n-1} \rtimes B_{n-1}$ in the indicated way (cf.~\cite[\S~7, p.~64]{PS}, see also~\cite[Par.~9.1.3, p.~243f]{FM}).

\section{Tensor categories} \label{Sec:TensCat}
\subsection{Finiteness} \label{Fin}
Let~$\K$ be an algebraically closed field of arbitrary characteristic, and let~$\CC$ be an essentially small abelian $\K$-linear category. Here, $\CC$ is called essentially small if there is a set, not a class, of objects with the property that every object in~$\CC$ is isomorphic to an object in this set. We assume that~$\CC$ is finite in the sense of~\cite[Def.~1.8.6, p.~9]{EGNO}, i.e., that it has finite-dimensional spaces of morphisms, that every object has finite length, that it has enough projectives, and that there are only finitely many isomorphism classes of simple objects. That it has enough projectives means that for every object~$X$, there is an epimorphism \mbox{$f\colon P \to X$} from a projective object~$P$ (cf.~\cite[Sec.~II.14, p.~70]{Mi}). A standard, but not completely trivial argument shows that, under the assumption on finite length, the epimorphism~$f$ can be chosen to be essential, where an epimorphism~\mbox{$f\colon P \to X$} is called essential if, given a morphism $g\colon Y \to P$, we can conclude that~$g$ is an epimorphism if~$f \circ g$ is an epimorphism. An essential epimorphism 
\mbox{$f\colon P \to X$} from a projective object~$P$ is usually called a projective cover of~$X$; it follows from the assumption on finite length that this definition is equivalent to the one given in~\cite[Def.~1.6.6, p.~6]{EGNO}. 

As explained in~\cite[Prop.~1.4, p.~3]{DSS}, an essentially small abelian $\K$-linear category is finite if and only if it is equivalent to the category of finite-dimensional modules over a finite-dimensional $\K$-algebra (cf.~also~\cite[p.~9f]{EGNO}). We note that finite categories are called bounded categories in~\cite{KL}.

We will assume that~$\CC$ is a strict tensor category in the sense of~\cite[Def.~XI.2.1, p.~282]{Ka}, so that we have in particular a tensor product functor~$\ot$ from~$\CC \times \CC$ to~$\CC$ and a unit object~$\1$. We note that tensor categories are called monoidal categories in~\cite[Def.~2.2.8, p.~25]{EGNO}; the term `tensor category' is used there for a category that satisfies additional restrictions (cf.~\cite[Def.~4.1.1, p.~65]{EGNO}), all of which will be satisfied in our situation, as we start to discuss now: First, we require the tensor product to be $\K$-bilinear. Second, we require that the unit object~$\1$ is simple and that $\End(\1)$ is one-dimensional over~$\K$.

In addition, we require that~$\CC$ is left rigid; i.e., that every object~$X$ has a left dual~$X^*$ (cf.~\cite[Def.~2.10.1, p.~40]{EGNO}; \cite[Def.~XIV.2.1, p.~342]{Ka}). The corresponding evaluation and coevaluation morphisms will be denoted by
\[\ev_X\colon X^* \ot X \to \1 \qquad \text{and} \qquad \coev_X\colon \1 \to X \ot X^*.\]
We also assume that~$\CC$ is braided with braiding
\[c_{X,Y}\colon X \ot Y \to Y \ot X\]
(cf.~\cite[Def.~8.1.1, p.~195]{EGNO}; \cite[Def.~XIII.1.1, p.~315]{Ka}) and that it carries a ribbon structure
\[\theta_X \colon X \to X\]
(cf.~\cite[Def.~8.10.1, p.~216]{EGNO}; \cite[Def.~XIV.3.2, p.~349]{Ka}), which by definition satisfies the equations
\[\theta_{X\ot  Y}=(\theta_X\ot \theta_Y)\circ c_{Y,X}\circ c_{X,Y}\] 
and $\theta_{X^*} = \theta_X^*$. Under these assumptions, every object~$X$ does not only have a left dual~$X^*$, but also a right dual~$\prescript{*}{} X$, which is defined by using the same object~\mbox{$\prescript{*}{} X \deq X^*$}, but introducing the evaluation morphism
\[\ev'_X \deq \ev_X \circ c_{X,X^*} \circ (\theta_X \ot \id_{X^*})\colon X \ot \prescript{*}{} X \to \1 \]
and the coevaluation morphism
\[\coev'_X \deq (\id_{X^*} \ot \theta_X) \circ c_{X,X^*} \circ \coev_X\colon \1 \to \prescript{*}{} X \ot X\]
(cf.~\cite[Prop.~XIV.3.5, p.~352]{Ka}). The axioms of left and right duality imply that for every object~$X$, the functor~$X \ot \textbf{--}$ has the left adjoint~$X^* \ot \textbf{--}$ and the right adjoint~$\prescript{*}{} X \ot \textbf{--}$, so that 
\[\Hom(X^* \ot Y, Z) \cong \Hom(Y, X \ot Z) \quad \text{and} \quad
\Hom(X \ot Y, Z) \cong \Hom(Y, \prescript{*}{}{X} \ot Z). \]
Similarly, the functor~$\textbf{--} \ot X$ has the right adjoint~$\textbf{--} \ot X^*$ and the left adjoint~$\textbf{--} \ot \prescript{*}{} X$, so that 
\[\Hom(Y \ot X, Z)\cong \Hom(Y, Z \ot X^*) \quad \text{and} \quad
\Hom(Y \ot \prescript{*}{}{X}, Z) \cong \Hom(Y, Z \ot X)\]
(cf.~\cite[Prop.~2.10.8, p.~42]{EGNO}; \cite[Prop.~XIV.2.2, p.~343f and p.~346f]{Ka}). 
This implies that the functors~$X \ot \textbf{--}$ and~$\textbf{--} \ot X$ are exact (cf.~\cite[Chap.~II, Thm.~7.7, p.~68]{HS}; \cite[Par.~2.7, Satz~3, p.~72]{P}). This in turn implies that $X \ot P$ and $P \ot X$ are projective if~$P$ is projective (cf.~\cite[Chap.~II, Prop.~10.2, p.~82]{HS}). We will also need the following related result on adjunctions:

\begin{Proposition} \label{RightAdj}
Suppose that $F\colon \CC \to \DD$ is a $\K$-linear functor between finite $\K$-linear categories. Then~$F$ has a right adjoint if and only if it is right exact.
\end{Proposition}
\begin{proof}
If~$F$ has a right adjoint, it is right exact (cf.~\cite[Chap.~II, Thm.~7.7, p.~68]{HS}). As pointed out above, we can assume for the converse that~$\CC$ and~$\DD$ are the categories of finite-dimensional left modules over finite-dimensional algebras~$A$ and~$B$, respectively. Then the Eilenberg-Watts theorem (cf.~\cite[Thm.~5.45, p.~261]{Ro2}; \cite[Thm.~2.4, p.~677]{I}) states that~$F$ is naturally equivalent to the functor~$M \ot_A -$ for a $B$-$A$-bimodule~$M$. But by the standard adjunction (cf.~\cite[Thm.~2.76, p.~93]{Ro2}), this functor has the right adjoint $\Hom_B(M,-)$.
\end{proof}

A different proof of this proposition that connects it to the adjoint functor theorem is given in~\cite[Cor.~1.9, p.~6]{DSS}.

An object~$X$ is called transparent if $c_{Y,X} \circ c_{X,Y} = \id_{X \ot Y}$ for all objects~$Y \in \CC$ (cf.~\cite[p.~224]{B}). Transparent objects are called central in \cite[Rem.~2.10, p.~296]{Mue}. The full subcategory consisting of all transparent objects is called the M\"uger center of~$\CC$.

The unit object is transparent (cf.~\cite[Exerc.~8.1.6, p.~196]{EGNO}; \cite[Prop.~XIII.1.2, p.~316]{Ka}), and the direct sum of two transparent objects is transparent. Therefore, an object that is isomorphic to a finite direct sum of copies of the unit object is contained in the M\"uger center. We say that~$\CC$ is modular if the M\"uger center does not contain any other objects. In this definition, we do not require that~$\CC$ be semisimple. In the semisimple case, this condition is equivalent to the invertibility of the modular 
\mbox{S-matrix} (cf.~\cite[Prop.~8.20.12, p.~243]{EGNO}; \cite[Cor.~2.16, p.~297]{Mue}; \cite[Sec.~5.1, p.~24]{Sh1}), which is in this case usually taken as the definition of modularity (cf.~\cite[Def.~8.13.4, p.~224]{EGNO}).

\subsection{Factorizable Hopf algebras} \label{Fact}
A fundamental example for categories with the properties that we have just listed are the representation categories of factorizable ribbon Hopf algebras. Suppose that~$A$ is a Hopf algebra with coproduct~$\Delta$, counit~$\varepsilon$, and antipode~$S$. For the coproduct, we will use the sigma notation of R.~Heyneman and M.~Sweedler in the modified form 
$\Delta(a) = a_{(1)} \ot a_{(2)}$. If we take for~$\CC$ the category of finite-dimensional left $A$-modules, it is not only a $\K$-linear abelian category, but also a tensor category, where the tensor product of two $A$-modules~$X$ and~$Y$ becomes an $A$-module via
\[a.(x \ot y) \deq a_{(1)}.x \ot a_{(2)}.y\]
for~$x \in X$, $y \in Y$, and~$a \in A$, where we have used a dot to indicate the module action. The unit object~$\1$ is the base field~$\K$, endowed with the trivial $A$-module structure coming from the counit. The left dual~$X^*$ of~$X$ is the dual vector space $X^* \deq \Hom_\K(X,\K)$, endowed with structures described in~\cite[Examp.~XIV.2.1, p.~347]{Ka}. Although this category is not strict, our considerations below apply to this category, because we have the notion of the tensor product 
$X_1 \ot X_2 \ot \dots \ot X_n$ of~$n$ vector spaces or $A$-modules, which is not built as an iteration of the tensor product of two vector spaces and therefore does not depend on the insertion of parentheses. The tensor products that appear below have to be understood in this way.

We assume that $A$ is quasitriangular with R-matrix~$R = R^1 \ot R^2 \in A \ot A$, where we follow the spirit of the notation for the coproduct above and write this tensor, which is in general not decomposable, as if it were decomposable (cf.~\cite[Def.~VIII.2.1, p.~173]{Ka}). From the R-matrix, we obtain a braiding on~$\CC$ via the formula
\[c_{X,Y}(x \ot y) \deq R^2.y \ot R^1.x\]
for~$x \in X$ and $y \in Y$ (cf.~\cite[Prop.~XIII.1.2, p.~318]{Ka}). From the R-matrix, we also obtain the monodromy matrix~$Q \deq R'R$, where $R'\deq R^2 \ot R^1$. As in the case of the R-matrix, we write $Q = Q^1 \ot Q^2$. $A$~is called factorizable if  the map
\[\Phi\colon A^* \to A,~\varphi \mapsto (\id_A \ot \varphi)(Q)\]
is bijective (cf.~\cite[Par.~3.2, p.~26]{SZ} and the references given there). Note that this condition implies that~$A$ is finite-dimensional, so that~$\CC$ is finite. We will require that~$A$ is factorizable, which is equivalent to the requirement that the category~$\CC$ is modular, as we will explain in Paragraph~\ref{CoendHopf}.

Furthermore, we assume that~$A$ contains a ribbon element, i.e., a nonzero central element~$v$ that satisfies 
\[\Delta(v) = Q (v \ot v) \qquad \text{and} \qquad S(v) = v\]
(cf.~\cite[Par.~4.3, p.~37]{SZ}; note the difference to~\cite[Def.~XIV.6.1, p.~361]{Ka}). Our ribbon element gives rise to a ribbon structure of~$\CC$ by defining
\[\theta_X\colon X \to X,~x \mapsto v.x\]
(cf.~\cite[Prop.~XIV.6.2, p.~361]{Ka}).

\subsection{Coends} \label{Coend}
\enlargethispage{11mm}
If~$\CC$ and~$\DD$ are categories and $F\colon \CC^{\op} \times \CC \to \DD$ is a bifunctor, a coend of~$F$ is an object~$L$ of~$\DD$ together with a morphism $\iota_X\colon F(X,X) \to L$ for every object~$X \in \CC$ which is dinatural in the sense that, for every morphism $f\colon X \to Y$, the diagram 
\begin{center}
\begin{tikzcd}[column sep=large]
F(Y,X) \arrow[r, "{F(\id_{Y}, f)}"] \arrow[d, "{F(f, \id_X)}"'] & F(Y,Y) \arrow[d, "\iota_Y"] \\
F(X,X) \arrow[r, "\iota_X"'] & L
\end{tikzcd}
\end{center}
commutes (cf.~\cite[Chap.~IX, Sec.~6, p.~226f]{ML2}). Moreover, this dinatural transformation is required to be universal, which means that for another dinatural transformation $\kappa_X\colon F(X,X) \to Z$, there is a unique morphism $g\colon L \to Z$ that makes for all $X\in \CC$ the diagram  
\begin{center}
\begin{tikzcd}
F(X,X) \arrow[r, "\iota_X"]  \arrow["\kappa_X"']{dr} & L \arrow[d, "g"] \\ & Z
\end{tikzcd}
\end{center}
commutative. 
\vspace{6mm}
\pagebreak

These conditions determine a coend up to isomorphism. If a coend exists, we will always assume that a particular coend has been chosen, for which we will use the notation
\[L = \int^{X} F(X, X)\]
as in~\cite[loc.~cit.]{ML2}.

From its universal property, we see that coends behave well with respect to natural transformations:
\begin{Lemma} \label{LemCoendNat}
Suppose that $F'\colon \CC^{\op} \times \CC \to \DD$ is a second bifunctor with coend
\mbox{$\iota'_X\colon F'(X,X) \to L'$}, and that $\eta_{X,Y} \colon F(X,Y) \to F'(X,Y)$ is a natural transformation. Then there is a unique morphism $g\colon L \to L'$ such that the diagram
\begin{center}
\begin{tikzcd}
F(X,X) \arrow[r, "\iota_X"] \arrow[d, "\eta_{X,X}"'] & L \arrow[d, "g"] \\
F'(X,X) \arrow[r, "\iota'_X"'] & L'
\end{tikzcd}
\end{center}
commutes for all objects~$X \in \CC$.
\end{Lemma}
\begin{proof}
This lemma is the version of~\cite[Chap.~IX, Sec.~7, Prop.~1, p.~228]{ML2} for coends instead of ends.
\end{proof}

When we speak of a coend without further specifications, we think of the case where~$\CC=\DD$ is our base category satisfying the requirements stated in Paragraph~\ref{Fin}, and the bifunctor~$F$ is given by $F(X,Y) = X^* \ot Y$. It is shown in \cite[Par.~5.1.3, p.~266ff]{KL} that, under our assumptions on~$\CC$, a coend for this bifunctor~$F$ exists. 

It is not too difficult to see that a functor that possesses a right adjoint preserves coends. This together with the Fubini theorem for coends (cf.~\cite[Chap.~IX, Sec.~8, p.~230f]{ML2}) implies that the dinatural transformation
\[\iota_X \ot \iota_Y\colon X^* \ot X \ot Y^* \ot Y \to L \ot L \]
is a coend. Comparing this transformation with the dinatural transformation
\[(\ev_X \ot \ev_Y) \circ (\id_{X^*} \ot (c_{Y^*, X} \circ c_{X, Y^*}) \ot \id_Y)\colon 
X^* \ot X \ot Y^* \ot Y \to \1, \]
the universal property yields a morphism $\omega_L \colon L \ot L \to \1$ that, via the adjunctions stated in Paragraph~\ref{Fin}, determines homomorphisms $\omega'_L\colon L \to L^*$ and 
$\omega''_L\colon L \to \prescript{*}{}{L}$, which are duals of each other. It is shown in~\cite[Thm.~1.1, p.~3]{Sh1} that the modularity of~$\CC$ is equivalent to the property that~$\omega'_L$, or alternatively~$\omega''_L$, is an isomorphism. This property is used in~\cite[Def.~5.2.7, p.~276]{KL} as the definition of modularity. 

The same method can be used to introduce other morphisms that will be needed in the sequel. The easiest of these morphisms arises from the dinatural transformation
$\iota_X \circ(\id_{X^*} \ot \, \theta_{X}) \colon X^* \ot X \to L$. Applying the universal property of the coend yields a morphism $\fT \colon L \to L$ that makes all diagrams of the form
\begin{center}
\begin{tikzcd}
X^* \ot X \arrow[r, "\iota_X"]  \arrow["\iota_X \circ(\id_{X^*} \ot \, \theta_{X})"']{dr} & L \arrow[d, "\fT"] \\ & L
\end{tikzcd}
\end{center}
commutative. Similarly, the universal property of~$L \ot L$ explained above can be used to obtain a morphism $\fN'\colon L \ot L \to L \ot L$ that makes all diagrams of the form
\begin{center}
\begin{tikzcd}
X^* \ot X \ot Y^* \ot Y \arrow[r, "\iota_X \ot \iota_Y"]  \arrow["(\iota_X \ot \iota_Y) \circ (\id_{X^*} \ot (c_{Y^*,X} \circ c_{X,Y^*}) \ot \id_Y)"']{dr} & L \ot L \arrow[d, "\fN'"] \\ & L \ot L
\end{tikzcd}
\end{center}
commutative. Both of these morphisms are used when defining $\fN \deq \fN' \circ (\fT \ot \fT)$. An in a sense hybrid form of the dinatural transformations used in the definition of~$\omega_L$ and of~$\fN'$ is used in the definition of $\fS'$ from $L \ot L$ to $L$: It arises by applying the universal property of~$L \ot L$ to a dinatural transformation in such a way that the diagram
\begin{center}
\begin{tikzcd}
X^* \ot X \ot Y^* \ot Y \arrow[r, "\iota_X \ot \iota_Y"]  
\arrow["(\ev_X \ot \iota_Y) \circ (\id_{X^*} \ot (c_{Y^*,X} \circ c_{X,Y^*}) \ot \id_Y)"']{dr} & L \ot L \arrow[d, "\fS'"] \\ &  L
\end{tikzcd}
\end{center}
becomes commutative.

Given an object~$Z \in \CC$, we define another morphism that is similar to $\fN'$: 
Again, because a functor that possesses a right adjoint preserves coends, the dinatural transformation $\id_Z \ot \iota_X: Z \ot X^* \ot X \to Z \ot L$ is a coend. Its universal property yields a morphism $\fN^l_{Z,L}\colon Z \ot L \to Z \ot L$ that makes all diagrams of the form
\vspace{-2mm}
\begin{center}
\begin{tikzcd}[column sep=large]
Z \ot X^* \ot X  \arrow[r, "\id_Z \ot \iota_X"]  
\arrow["(\id_{Z}\ot \iota_X) \circ ((c_{X^*,Z}\circ c_{Z,X^*})\ot \id_{X})"']{dr} & Z \ot L \arrow[d, "\fN^l_{Z,L}"] \\ &  Z \ot L
\end{tikzcd}
\end{center}
commutative. 

The coend~$L$ is a Hopf algebra inside the category~$\CC$. We will not need the arising product, coproduct, unit, counit, and antipode, which are constructed in a similar way as the morphisms~$\omega_L$, $\fT$, $\fN'$, and $\fS'$ by using the universal property of a coend and are described in~\cite[Par.~1.6, p.~478f]{V} (cf.~also \cite[Par.~5.2.2, p.~271ff]{KL}; note that the conventions are slightly different there). We will, however, need that, as a consequence of the Hopf algebra structure, there are two-sided integrals $\Lambda_L\colon \1 \to L$ and $\lambda_L\colon L \to \1$. Here, the assumption on modularity is used for two points, namely on the one hand for the fact that the unit object~$\1$ can indeed be used as the domain of~$\Lambda_L$ and the codomain of~$\lambda_L$, and on the other hand for the fact that~$\lambda_L$ is two-sided (cf.~\cite[Sec.~5.2, p.~270ff]{KL}). The integral $\Lambda_L$ is used to define 
$\fS \in \End(L)$ as the composition
\[L\xlongrightarrow{\cong} L\ot \1 \xlongrightarrow{\id_L\ot \Lambda_L} L\ot L\xlongrightarrow{\fS'} L\]
(cf.~\cite[Par.~1.3, p.~473]{L1}; the setting there is slightly more general as the modularity hypothesis is weakened).

\subsection{Coends from Hopf algebras} \label{CoendHopf}
In the case where~$\CC$ is the category of finite-dimensional left modules over a factorizable ribbon Hopf algebra~$A$, the coend can be described explicitly: The dual vector space $L \deq A^*$, viewed as a left $A$-module via the left coadjoint action 
\[(a.\varphi)(a') \deq \varphi(S(a_{(1)}) a' a_{(2)})\]
for $a, a' \in A$ and $\varphi \in A^*$, becomes a coend when endowed with the dinatural transformation
\[\iota_X(\xi \ot x)(a) = \xi(a.x) \]
for a left $A$-module~$X$ and elements $\xi \in X^*$, $x \in X$, and~$a \in A$ (cf.~\cite[Lem.~4.3, p.~498]{V}, see also~\cite[Thm.~7.4.13, p.~331]{KL}). It is not difficult to find the explicit form of the morphisms introduced in Paragraph~\ref{Coend} in this model of the coend:

\begin{Proposition} \label{CoendHopfProp}
Suppose that $\varphi, \psi \in A^*$ and that $a,a' \in A$. Then we have \vspace{-10pt}
\begin{enumerate}
\item
$\fT(\varphi)(a)=\varphi(av)$

\item
$\fN'(\varphi \ot \psi)(a \ot a') = \varphi(a Q^1) \psi(S(Q^2) a')$

\item
$\fN(\varphi \ot \psi)(a \ot a') = \varphi(a v_{(1)}) \psi(S(v_{(2)})a')$

\item
$\fS'(\varphi \ot \psi)(a) = \varphi(Q^1) \psi(S(Q^2)a)$

\item 
$\omega_L(\varphi \ot \psi) = \varphi(Q^1) \psi(S(Q^2))$
\end{enumerate}
\end{Proposition}
\begin{proof}
We prove the first assertion by showing that the right-hand side indeed has the required universal property. This holds since we have for an $A$-module~$X$ and $\xi \in X^*, x \in X$ that
\begin{align*}
((\iota_X \circ (\id_{X^*} \ot \theta_{X}))(\xi \ot x))(a) &= 
\iota_X(\xi \ot v.x)(a) = \xi(av.x) = \fT(\iota_X(\xi\ot x))(a).
\end{align*}
The proof of the second assertion is similar: If~$Y$ is another $A$-module and 
$\zeta \in Y^*$, $y \in X$ are elements, we have
\begin{align*}
&([(\iota_X \ot \iota_Y) \circ (\id_{X^*} \ot (c_{Y^*,X} \circ c_{X,Y^*}) \ot \id_Y)]
(\xi \ot x \ot \zeta \ot y)) (a \ot a')\\
&\quad=((\iota_X \ot \iota_Y)(\xi \ot Q^1.x \ot Q^2.\zeta \ot y))(a \ot a') =
\xi(aQ^1.x) (Q^2.\zeta)(a'.y) \\
&\quad= \xi(aQ^1.x) \zeta(S(Q^2)a'.y) = 
\fN'  (\iota_X(\xi \ot x) \ot \iota_Y(\zeta \ot y)) (a\ot a').
\end{align*}
From the first two assertions, we obtain that
\begin{align*}
\fN(\varphi \ot \psi)(a \ot a') &= \fN'(\fT(\varphi) \ot \fT(\psi))(a \ot a') \\
&= \varphi(a Q^1 v) \psi(S(Q^2) a' v) = \varphi(a Q^1 v) \psi(S(Q^2 v) a')
\end{align*}
where the last equality uses that~$v$ is central and invariant under the antipode. Using that~$\Delta(v) = Q (v \ot v)$, we arrive at the third assertion. The proofs of the fourth and the fifth assertion are again very similar to the proofs of the first and the second assertion.
\end{proof}
Because the antipode of a finite-dimensional Hopf algebra is bijective (cf.~\cite[Thm.~2.1.3, p.~18]{Mo}), the fifth assertion shows that~$\omega_L$ is nondegenerate if and only if the map~$\Phi$ introduced in Paragraph~\ref{Fact} is bijective. This shows that the modularity of~$\CC$ is equivalent to the factorizability of~$A$ (cf.~\cite[Par.~7.4.6, p.~332]{KL}).

\pagebreak

We note that, because these maps are defined via the universal property of the coend, they must automatically be morphisms in $\CC$, i.e., they must be $A$-linear. We also note that in the second and the third assertions, we have chosen a special identification of~$A^* \ot A^*$ 
with~$(A \ot A)^*$, which is determined by
$(\varphi \ot \psi)(a \ot a') = \varphi(a) \psi(a')$.

It is important to stress that the Hopf algebra structure on~$L$ is not the usual Hopf algebra structure on~$A^*$. In fact, $A^*$ is a Hopf algebra in the category of vector spaces, whereas~$L$ is a Hopf algebra in the category~$\CC$, which has a different braiding. The Hopf algebra structure on~$L$ is rather dual to the so-called transmuted Hopf algebra structure of~$A$ (cf.~\cite[Examp.~9.4.9, p.~504]{Maj}; \cite[Lem.~4.4, p.~499]{V}). However, there is a very direct relation between the integrals of~$A^*$ and~$L$: Because a factorizable Hopf algebra is unimodular (cf.~\cite[Prop.~3.7.4, p.~482]{L1} and~\cite[Prop.~12.4.2, p.~405f]{Ra}), there is a nonzero two-sided integral~$\Lambda_A \in A$. This leads to the two-sided integral
\[\lambda_L\colon L \to \1 = \K,~\varphi \mapsto \varphi(\Lambda_A).\]

On the other hand, a right integral~$\rho\colon A \to \K$ is by definition contained in~$L=A^*$, and it can be shown that the unique morphism $\Lambda_L\colon \1 = \K \to L$ that maps~$1_\K$ to~$\rho$ is even a two-sided integral (cf.~\cite[Prop.~6.6, p.~153]{BKLT}). As a consequence, the morphism~$\fS$ is given by the formula
\[\fS(\varphi)(a) = \varphi(Q^1) \rho(S(Q^2)a).\]


\subsection{The block spaces} \label{BlockSpc}
In topological field theory, one associates to a modular category projective representations of mapping class groups of surfaces. There is a vast literature on this topic; our approach here is based on a construction given by V.~Lyubashenko in his articles~\cite{L1} and~\cite{L2}. We will now briefly review some key aspects of his construction to the extent that we need them. Suppose that~$\CC$ is a category that satisfies the requirements stated in Paragraph~\ref{Fin}. We assume that each of the~$n$ boundary components of the surface~$\Sigma_{g,n}$ described in Paragraph~\ref{Surf} is labeled by an object~$X_i$ from~$\CC$. Associated with such a labeled surface is the space
\[\TFT(\Sigma_{g,n}^{X_1,\ldots, X_n}) \deq
\Hom_{\CC}(X_n \ot \cdots \ot X_1, L^{\ot g}),\]
where~$L$ denotes the coend defined in Paragraph~\ref{Coend}. This space, which we call the space of chiral conformal blocks or briefly the block space, obviously depends functorially on the labels and can therefore be viewed as the value of a left exact contravariant functor $\TFT(\Sigma_{g,n})$ on the object $(X_1,\ldots, X_n) \in \CC^n$, by which we mean that the functor is left exact in each argument (cf.~\cite[Par.~4.6, p.~130f]{P}).

The construction of a projective action of the mapping class group then starts from an oriented net (cf.~\cite[Sec.~6, p.~355]{L2} and \cite[Par.~4.1, p.~486f]{L1}). 
\begin{samepage}
We consider the following oriented net that encodes the structure of our surface: 
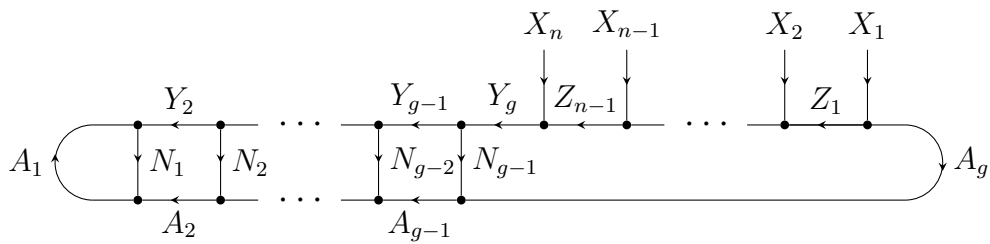
\begin{center}
\begin{tikzpicture}
\node[circle,fill,scale=0.3] (x1) at (-1,0) {};
\node[circle,fill,scale=0.3] (x2) at (0.1,0) {};
\node[circle,fill,scale=0.3] (x3) at (2.2,0) {};
\node[circle,fill,scale=0.3] (x4) at (3.3,0) {};
\node[circle,fill,scale=0.3] (x5) at (4.4,0) {};
\node[circle,fill,scale=0.3] (x6) at (5.5,0) {};
\node[circle,fill,scale=0.3] (x7) at (7.6,0) {};
\node[circle,fill,scale=0.3] (xe) at (8.7,0) {};
\node[circle,fill,scale=0.3] (x8) at (-1,-1) {};
\node[circle,fill,scale=0.3] (x9) at (0.1,-1) {};
\node[circle,fill,scale=0.3] (x10) at (2.2,-1) {};
\node[circle,fill,scale=0.3] (x11) at (3.3,-1) {};

\draw[->-] (4.4,1) to node[pos=0, above]{$X_n$} (x5);
\draw[->-] (5.5,1) to node[pos=0, above]{$X_{n-1}$} (x6);
\draw[->-] (7.6,1) to node[pos=0, above]{$X_2$} (x7);
\draw[->-] (8.7,1) to node[pos=0, above]{$X_1$} (xe);
\draw[-<-] (x1) to node[pos=0.5, above]{$Y_2$} (x2);
\draw (x2) to (0.6,0);
\node[scale=1.2] at (1.2,0) {$\cdots$};
\draw (1.7,0) to (x3);
\draw[-<-] (x3) to node[pos=0.5, above]{$Y_{g-1}$} (x4);
\draw[-<-] (x4) to node[pos=0.5, above]{$Y_g$} (x5);
\draw[-<-] (x5) to node[pos=0.5, above]{$Z_{n-1}$} (x6);
\draw (x6) to (6,0);
\node[scale=1.2] at (6.6,0) {$\cdots$};
\draw (7.1,0) to (x7);
\draw[-<-] (x7) to node[pos=0.5, above]{$Z_1$} (xe);
\draw[->-] (x1) to node[pos=0.5, right]{$N_1$} (x8);
\draw[->-] (x2) to node[pos=0.5, right]{$N_2$} (x9);
\draw[->-] (x3) to node[pos=0.5, right]{$N_{g-2}$} (x10);
\draw[->-] (x4) to node[pos=0.5, right]{$N_{g-1}$} (x11);
\draw[-<-] (x8) to node[pos=0.5, below]{$A_2$} (x9);
\draw (x9) to (0.6,-1);
\node[scale=1.2] at (1.2,-1) {$\cdots$};
\draw (1.7,-1) to (x10);
\draw[-<-] (x10) to node[pos=0.5, below]{$A_{g-1}$} (x11);
\draw (x1) to (-1.6,0);
\draw[-<-] (-1.6,0) arc (90:270:5mm)  node[pos=0.5, left]{$A_1$};
\draw (-1.6,-1) to (x8);
\draw (x7) to (9.2,0);
\draw[-<-] (9.2,-1) arc (-90:90:5mm) node[pos=0.5, right]{$A_g$};
\draw (x11) to (9.2,-1);
\end{tikzpicture}
\end{center}
\end{samepage}

The block space arises from the oriented net by taking suitable coends over internal edges (cf.~\cite[Par.~8.2, p.~374]{L2}). For our chosen net, the construction gives the space
\begin{align*} 
&T_1 \deq \\
&\int^{A_i,N_j,Z_l,Y_m} \Hom(A_1 \ot Y_2,N_1) \ot \Hom(Y_3,Y_2 \ot N_2)\ot \dots\ot  \Hom(Y_g,Y_{g-1}\ot  N_{g-1}) \\
&\qquad \ot  \Hom(X_n\ot  Z_{n-1},Y_g)\ot \Hom(X_{n-1} \ot Z_{n-2},Z_{n-1})\ot \dots \ot  \Hom(X_2 \ot Z_1, Z_2) \\
&\qquad \ot \Hom(X_1,Z_1\ot  A_g)\ot  \Hom(N_{g-1} \ot A_g,A_{g-1}) \ot \dots \ot \Hom(N_1\ot  A_2,A_1) 
\end{align*}
where in the upper limit of the integral sign the abbreviation~$A_i$ has been used for~$A_1, A_2, \dots, A_g$. Similarly, the abbreviations
$N_j$, $Z_l$, and~$Y_m$ have been used for the objects with indices $j=1,\dots,g-1$, $l=1,\dots,n-1$, and $m=2,\dots,g$. The coends are taken in the category of left exact functors:
Obviously, the argument of the coend depends functorially on the objects, and the corresponding functor is left exact. A more detailed discussion of this aspect can be found in~\cite[App.~B, p.~398]{L2} and \cite[Sec.~3, p.~72ff]{FS}.

In view of \cite[Lem.~B.1, p.~398]{L2} and the functor adjunctions discussed in Paragraph~\ref{Fin}, the space~$T_1$ is naturally isomorphic to
\begin{align*}
T_2 \deq &\int^{A_i, N_j, Y_2, Y_g, Z_1} 
\Hom(Y_2,\prescript{*}{}{A_1}\ot  N_1) \ot \Hom(Y_g,Y_2\ot  N_2 \ot \dots\ot  N_{g-1}) \\
&\qquad\ot \Hom(X_n \ot \dots \ot X_2 \ot  Z_1,Y_g)\ot \Hom(X_1\ot \prescript{*}{}{A_g},Z_1) \\
&\qquad\ot  \Hom(N_{g-1},A_{g-1} \ot A_g^*) \ot \dots \ot \Hom(N_1,A_1 \ot  A_2^*). 
\end{align*}
Taking also the coend over~$Y_g$, we obtain a natural isomorphism with the space
\begin{align*}
T_3 \deq & \int^{A_i,N_j,Z_1,Y_2} \Hom(Y_2,\prescript{*}{}{A_1}\ot  N_1) \nonumber\\
&\qquad\ot \Hom(X_n\ot \dots\ot  X_2\ot  Z_1,Y_2\ot  N_2\ot \dots\ot  N_{g-1})\ot  \Hom(X_1\ot \prescript{*}{}{A_g},Z_1) \nonumber\\
&\qquad\ot  \Hom(N_{g-1},A_{g-1}\ot  A_g^*)\ot \dots\ot \Hom(N_1,A_1\ot  A_2^*). 
\end{align*}
Treating~$N_1, \ldots, N_{g-1}$ in the same way and moreover taking the coend over~$Z_1$, we obtain the space 
\begin{align*}
T_4 \deq & \int^{A_i,Y_2}\Hom(Y_2,\prescript{*}{}{A_1}\ot  A_1\ot  A_2^*) \\
&\qquad \ot \Hom(X_n \ot \dots \ot X_1\ot  \prescript{*}{}{A_g},Y_2\ot  A_2\ot  A_3^* \ot \dots \ot A_{g-1}\ot  A_g^*). 
\end{align*}
Finally, if we also take the coend over~$Y_2$ and apply the functor adjunctions to~$A_g$, we arrive at the space 
\begin{align*}
&T_5 \deq \int^{A_i}\Hom(X_n \ot \dots \ot X_1, 
\prescript{*}{}{A_1} \ot A_1 \ot A_2^* \ot A_2 \ot \dots \ot A_{g-1} \ot A_g^* \ot A_g). 
\end{align*}
It is now important to realize that we have introduced right duals in Paragraph~\ref{Fin} in such a way that left and right duals have the same underlying object. They differ only in their  evaluation and coevaluation morphisms, which affect the functor adjunctions, but not the objects themselves. The space~$T_5$ is therefore also equal to 
\begin{align*}
&T_5 = \int^{A_i}\Hom(X_n \ot \dots \ot X_1, A_1^* \ot A_1 \ot \dots \ot A_g^* \ot A_g).
\end{align*}
Because the coends are taken in the category of left exact functors, the space~$T_5$ is naturally isomorphic to our block space~$\Hom(X_n \ot \dots \ot X_1, L^{\ot  g})$ defined above (cf.~\cite[Prop.~3.4, p.~74]{FS}).

We have claimed that the specific oriented net given above encodes the structure of the surface introduced in Paragraph~\ref{Surf}. We now explain this relation in the special case~$g=3$ and~$n=2$; the general case is not essentially different. The first step consists of the application of the fattening functor (cf.~\cite[p.~341]{L2}), which turns the net into a so-called ribbon graph by replacing each trivalent vertex by a hexagon in which every second side is represented by a double line. These hexagons are arranged so that their double lines come to lie on the edges of the net, and if an edge connects two vertices, the corresponding hexagons are glued along their now common double line. 

\pagebreak

If we apply this procedure to our net above, we obtain the ribbon graph
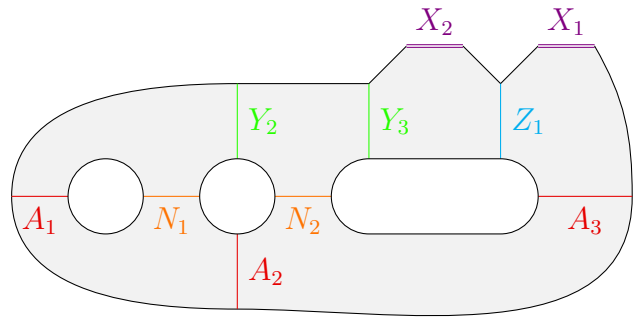
\begin{center}
\begin{tikzpicture}
\fill[gray!10] (-3,-0.5) to[out=90,in=180] (0,1) to (1.75,1) to (2.25,1.5)
to (3,1.5) to (3.5,1) to (4,1.5) to (4.75,1.5) to[out=300,in=90] (5.25,-0.5) to[out=270, in=0] (0,-2) to[out=180, in=270] (-3,-0.5);
\draw (-3,-0.5) to[out=90,in=180] (0,1) to (1.75,1) to (2.25,1.5);
\draw[double, violet] (2.25,1.5) to node[pos=0.5, above]{$X_2$} (3,1.5);
\draw (3,1.5) to (3.5,1) to (4,1.5);
\draw[double, violet] (4,1.5) to node[pos=0.5, above]{$X_1$} (4.75,1.5);
\draw (4.75,1.5) to[out=300,in=90] (5.25,-0.5) to[out=270, in=0] (0,-2) to[out=180, in=270] (-3,-0.5);

\draw[green!80!lime] (0,1) to node[pos=0.5, right]{$Y_2$} (0,0);
\draw[green!80!lime] (1.75,1) to node[pos=0.5, right]{$Y_3$} (1.75,0);
\draw[cyan] (3.5,1) to node[pos=0.5, right]{$Z_1$} (3.5,0);
\draw[red!10!orange] (-1.25,-0.5) to node[pos=0.5, below]{$N_1$} (-0.5,-0.5);
\draw[red!10!orange] (0.5,-0.5) to node[pos=0.5, below]{$N_2$} (1.25,-0.5);
\draw[black!10!red] (-3,-0.5) to node[pos=0.5, below]{$A_1$} (-2.25,-0.5);
\draw[black!10!red] (0,-1) to node[pos=0.5, right]{$A_2$} (0,-2);
\draw[black!10!red] (4,-0.5) to node[pos=0.5, below]{$A_3$} (5.25,-0.5);

\draw[fill=white] (1.25,-0.5) to[out=90,in=180] (1.75,0) to (3.5,0) to[out=0, in=90] (4,-0.5) to[out=270, in=0] (3.5,-1) to (1.75,-1) to[out=180, in=270] (1.25,-0.5);

\draw[fill=white] (-2.25,-0.5) to[out=90, in= 180] (-1.75,0) to[out=0,in=90] (-1.25,-0.5) to[out=270, in=0] (-1.75,-1) to[out=180, in=270] (-2.25,-0.5);
\draw[fill=white] (-0.5,-0.5) to[out=90, in= 180] (0,0) to[out=0,in=90] (0.5,-0.5) to[out=270, in=0] (0,-1) to[out=180, in=270] (-0.5,-0.5);
\end{tikzpicture}
\end{center}
in which the colored lines indicate the glued double lines.

The second step consists of the application of the duplication functor (cf.~\cite[Par.~2.2 and Par.~2.3, p.~315ff]{L2}). The duplication functor turns a ribbon graph into a surface by taking two copies of the graph and gluing them along the boundary components that do not arise from the double lines mentioned above. Applied to our ribbon graph, it yields a surface of genus~$3$ that has two boundary components, which arise from the two double lines still present in our ribbon graph and are labeled by~$X_1$ and~$X_2$. If we deform the arising surface slightly and bring the first and the third handle closer together by pulling the middle handle to the bottom, we obtain the standard surface
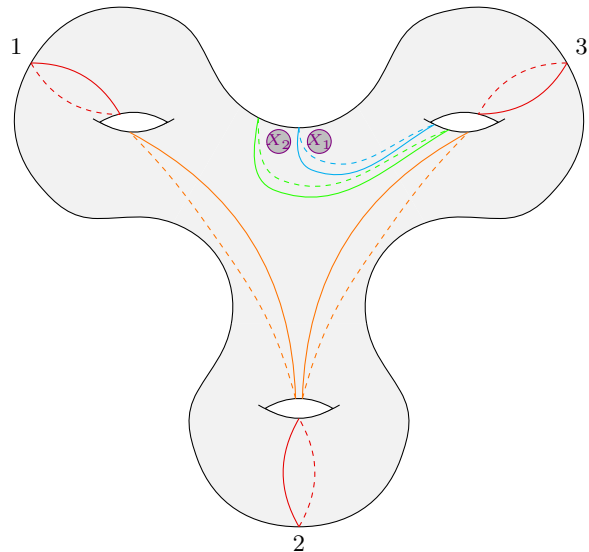
\begin{center}
\fontsize{9}{9}\selectfont
\begin{tikzpicture}[scale=0.45]
\pic[scale=0.45] (lower) at (0,-pi) {handle}; 
\pic[rotate=120,scale=0.45] (tr) at (30:pi) {handle};
\pic[rotate=-120,scale=0.45] (tl) at (150:pi) {handle};
\fill[gray!10]  (lower-right) to[out=100,in=200] (tr-left)-- 
(tr-right) to[out=-140,in=-40] (tl-left)
-- (tl-right) to[out=-20,in=80] (lower-left) -- cycle;
\draw (lower-right) to[out=100,in=200] (tr-left);
\draw (tr-right) to[out=-140,in=-40] (tl-left);
\draw (tl-right) to[out=-20,in=80] (lower-left);

\node[below] at (lower-num) {2}; 
\node[above right] at (tr-num) {3}; 
\node[above left] at (tl-num) {1}; 

\draw[black!10!red] 
(0,-5.94) to[bend right] (lower-num);
\draw[black!10!red,dash pattern=on 2pt off 2pt] (lower-num) to[bend right](0,-5.94);

\draw[rotate=120,black!10!red] 
(0,-6.12) to[bend right] (tr-num);
\draw[rotate=120,black!10!red,dash pattern=on 2pt off 2pt] (tr-num) to[bend right](0,-6.12);

\draw[rotate=240,black!10!red] 
(0,-6.1) to[bend right] (tl-num);
\draw[rotate=240,black!10!red,dash pattern=on 2pt off 2pt] (tl-num) to[bend right](0,-6.1);

\draw[violet,rotate=-180, 
postaction={decorate},fill=lightgray] (-0.6,-2.25)  circle (10pt) node {$\scriptstyle{X_1}$} ;
\draw[violet,rotate=-180,
postaction={decorate},fill=lightgray] (0.6,-2.25) circle (10pt) node {$\scriptstyle{X_2}$} ;

\draw[red!10!orange] (0.1,-5.35) to[bend left] (5,2.52); 
\draw[red!10!orange,dash pattern=on 2pt off 2pt] (0.1,-5.35) to[out=75, in=235] (2.1,-1.1) to[out=55, in=225] (5,2.52);
\draw[red!10!orange] (-0.1,-5.35) to[bend right] (-5,2.52);
\draw[red!10!orange,dash pattern=on 2pt off 2pt] (-0.1,-5.35) to[out=105, in=305] (-2.1,-1.1) to[out=125, in=315] (-5,2.52);

\draw[cyan] (0,2.67) to[out=260, in=160] (0.4,1.4) to[out=340, in=210] (4,2.75);
\draw[cyan,dash pattern=on 2pt off 2pt] (0,2.67) to[out=270, in=160] (0.6,1.7) to[out=340, in=200] (4,2.75);
\draw[green!80!lime] (-1.2,2.95) to[out=260, in=160] (-0.7, 0.8) to[out=340, in=210] (4.43,2.6);
\draw[green!80!lime,dash pattern=on 2pt off 2pt] (-1.2,2.95) to[out=270, in=160] (-0.5, 1.1) to[out=340, in=200] (4.43,2.6);

\scalebox{-1}[1]{
\draw[green!80!lime] (2.2,3.8) to[out=270, in=60] (0.9, -0.5) to[out=240, in=90] (0,-5.33);
\draw[green!80!lime,dash pattern=on 2pt off 2pt] (2.2,3.8) to[out=270, in=60] (0.5, -0.5) to[out=240, in=90] (0,-5.33);}
\end{tikzpicture}
\end{center}
described in Paragraph~\ref{Surf}, in which the curves arise from the duplication of the lines in the ribbon graph that have the same color. It should be noted that the red curves are freely homotopic to the curves denoted by $\alpha_1$, $\alpha_2$, and $\alpha_3$ in Paragraph~\ref{Surf}, and the orange curves are freely homotopic to the curves~$\mu_1$ and~$\mu_2$ from Paragraph~\ref{FundGroupGen}, respectively.

It is possible to compare this discussion with the corresponding one in \cite[Par.~4.5, p.~494f]{L1}. To do that, it is necessary to deform the surface given in \cite[Fig.~6, p.~495]{L1} by pulling down the middle handles to the bottom in order to bring the first and the $g$th handle together at the top, exactly as in our discussion of the relation between the net and the surface above. This leads to a surface in which the handles are labeled counterclockwise, but the boundary components are labeled clockwise, in contrast to the surface used in Paragraph~\ref{Surf}, where the boundary components arose from the polygon and therefore were also labeled counterclockwise. As a consequence, the order of the labels~$X_1,\ldots,X_n$ is reversed. The following table explains the relation between our surface and the surface in \cite[Fig.~2, p.~490 and Fig.~6, p.~495]{L1}:
\begin{center}
\renewcommand{\arraystretch}{1.21}
\begin{tabular}{|c|c|c|c|c|c|} \hline
Curves & $\alpha_{i}$ & $\beta_i$ & $\mu_j$ & $\zeta_l$ & $\partial_k$ \\ \hline 
Color & \textcolor{black!10!red}{red} & \textcolor{black!10!blue}{blue} & \textcolor{red!10!orange}{orange} & \textcolor{cyan}{turquoise} & \textcolor{violet}{purple}\\ \hline 
Label in the net & $A_i$ & &  $N_j$ & $Z_l$ & $X_k$ \\ \hline 
Dehn twists & $\ft_i$ & $\fr_i$ & $\fn_j$ & $\fz_l$ & $\fd_k$\\ \hline 
\cite{L1} & $e_{i}$ & $b_i$ & $a_{j+1}$ & $t_{n-l,g}$ & $R_{g-k+1}$\\ \hline 
\end{tabular}
\end{center}
Here, the index ranges are $i=1,\dots,g$, $j=1,\dots,g-1$, $k=1,\dots,n$, and \mbox{$l=1,\dots,n-1$}. The light green curves in our surface correspond to the curves labeled~$d_j$ in \cite[loc.~cit.]{L1}.

\subsection{Mapping class group representations} \label{MapClRep}
As we just explained, the theory associates with a surface a certain left exact functor. But in addition, the theory associates with a mapping class~$[\psi] \in \Gamma_{g,n}$ a projective class~$[\TFT(\psi)] = P(\TFT(\psi))$ of natural equivalences between two of these functors, namely the functor
\[(X_1, \ldots, X_n) \mapsto \TFT(\Sigma_{g,n}^{X_1, \ldots, X_n})\]
on the one hand and the functor
\[(X_1, \ldots, X_n) \mapsto \TFT(\Sigma_{g,n}^{X_{\tau(1)},\ldots, X_{\tau(n)}})\]

on the other hand, where~$\tau \deq p([\psi])^{-1} \in S_n$ is the inverse permutation of the marked points introduced in Paragraph~\ref{MapClGr}. A representative~$\TFT(\psi)$ of this projective class is a natural equivalence between these functors, i.e., a family of linear isomorphisms
\[\TFT(\psi) \colon \TFT(\Sigma_{g,n}^{X_1, \ldots, X_n}) \to 
\TFT(\Sigma_{g,n}^{X_{\tau(1)},\ldots, X_{\tau(n)}})\]
that are natural in the following sense: If, for $i=1,\ldots,n$, we are given morphisms $f_i\colon X_i \to Y_i$, then the diagram 
\begin{center}
\begin{tikzcd}
\TFT(\Sigma_{g,n}^{Y_1,\ldots, Y_n}) \arrow[r, "\TFT(\psi)"]
\arrow[d, "\circ (f_n \ot \dots \ot f_1)"']  & 
\TFT(\Sigma_{g,n}^{Y_{\tau(1)},\ldots, Y_{\tau(n)}}) \arrow[d, "\circ (f_{\tau(n)} \ot \dots \ot f_{\tau(1)})"]\\
\TFT(\Sigma_{g,n}^{X_1,\ldots, X_n}) \arrow[r, "\TFT(\psi)"] & 
\TFT(\Sigma_{g,n}^{X_{\tau(1)},\ldots, X_{\tau(n)}})
\end{tikzcd}
\end{center}
commutes. Clearly, a nonzero scalar multiple of such a natural equivalence is again a natural equivalence, and the set of all its nonzero scalar multiples constitutes its projective class~$[\TFT(\psi)] = P(\TFT(\psi))$. 

This assignment is also compatible with composition in the sense that, for a second mapping class~$[\phi]$, we have $[\TFT(\phi \circ \psi)] = [\TFT(\phi) \circ \TFT(\psi)]$. However, this does not imply that the mapping class group~$\Gamma_{g,n}$ acts projectively on each block space $\TFT(\Sigma_{g,n}^{X_1,\ldots, X_n})$, because the block spaces are not preserved by~$\TFT(\psi)$ if~$[\psi]$ permutes two boundary components with different labels. In general, only the pure mapping class group~$\Pu\Gamma_{g,n}$ acts projectively on each block space. However, the entire mapping class group~$\Gamma_{g,n}$ acts projectively on the direct sum
\[\bigoplus_{\tau \in S_n} \TFT(\Sigma_{g,n}^{X_{\tau(1)},\ldots, X_{\tau(n)}})\]
of block spaces. In the case where all boundary components are labeled with the same object~$X$, we also have an action of~$\Gamma_{g,n}$ on the space~$\Hom_\CC(X^{\ot n}, L^{\ot g})$. We can in fact realize the direct sum above as a subspace of this space by setting 
$X \deq X_1 \oplus \dots \oplus X_n$.

Via a sophisticated system of rules laid out in~\cite{L1} and~\cite{L2}, the description of a surface in terms of a net makes it possible to describe the actions of the elements of the mapping class group on the block spaces explicitly. According to this formalism, the elements of the mapping class group introduced in Paragraph~\ref{Dehn}, Paragraph~\ref{Braid}, Paragraph~\ref{SpecDehn}, and Paragraph~\ref{MoreSpecDehn} act on the block spaces as follows:
\begin{parlist}
\item \label{D}
For $j = 1,\ldots,n$, the Dehn twist~$\fd_j$ acts by precomposition with the morphism
$\id_{X_n \ot \dots \ot X_{j+1}} \ot \theta_{X_j} \ot \id_{X_{j-1} \ot \dots \ot X_1}$,
i.e., via the map
\begin{align*}
\TFT(\fd_j)\colon \Hom_{\CC}(&X_n \ot \cdots \ot  X_1, L^{\ot g}) \to 
\Hom_{\CC}(X_n \ot \cdots \ot  X_1, L^{\ot g}) \\
&f \mapsto f \circ (\id_{X_n \ot \dots \ot X_{j+1}} \ot \theta_{X_j} \ot \id_{X_{j-1} \ot \dots \ot X_1}).
\end{align*}

\item \label{B}
For $j = 1,\ldots,n-1$, the braiding~$\fb_{j,j+1}$ acts by precomposition with the morphism
$\id_{X_n \ot \dots \ot X_{j+2}} \ot c_{X_{j}, X_{j+1}} \ot \id_{X_{j-1} \ot \dots \ot X_1}$,
i.e., it induces a map
\begin{align*}
\TFT(\fb_{j,j+1})\colon \Hom_{\CC}(&X_n \ot \cdots \ot X_{j+1} \ot X_j \ot  \cdots \ot  X_1, L^{\ot g}) \\
&\mspace{100mu} \to 
\Hom_{\CC}(X_n \ot \cdots \ot X_{j} \ot X_{j+1} \ot  \cdots \ot  X_1, L^{\ot g}) \\
&f \mapsto f \circ (\id_{X_n \ot \dots \ot X_{j+2}} \ot c_{X_{j}, X_{j+1}} \ot \id_{X_{j-1} \ot \dots \ot X_1}).
\end{align*}

\item \label{DD}
For $j = 1,\ldots,n-1$, the element~$\fd_{j,j+1}$ acts by precomposition with the morphism
$\id_{X_n \ot \dots \ot X_{j+2}} \ot \theta_{X_{j+1} \ot X_j} \ot \id_{X_{j-1} \ot \dots \ot X_1}$, a map that we denote by~$\TFT(\fd_{j,j+1})$.

\item \label{Q}
For $j = 1,\ldots,n-1$, the element~$\fq_{j,j+1}$ acts by precomposition with the morphism
$\id_{X_n \ot \dots \ot X_{j+2}} \ot (c_{X_j,X_{j+1}}\circ c_{X_{j+1},X_j}) \ot \id_{X_{j-1} \ot \dots \ot X_1}$, a map that we denote by~$\TFT(\fq_{j,j+1})$.

\item \label{T}
For $i = 1,\ldots,g$, the Dehn twist $\ft_i$ acts by postcomposition with the morphism
$\id_{L^{\ot (i-1)}} \ot \fT \ot \id_{L^{\ot (g-i)}}$, i.e., via the map
\begin{align*}
\TFT(\ft_i)\colon \Hom_{\CC}(&X_n \ot \cdots \ot  X_1, L^{\ot g}) \to 
\Hom_{\CC}(X_n \ot \cdots \ot  X_1, L^{\ot g}) \\
&f \mapsto (\id_{L^{\ot (i-1)}} \ot \fT \ot \id_{L^{\ot (g-i)}}) \circ f.
\end{align*}

\item \label{S}
For $i = 1,\ldots,g$, the element $\fs_i$ acts by postcomposition with the morphism 
$\id_{L^{\ot (i-1)}} \ot \fS \ot \id_{L^{\ot (g-i)}}$, a map that we denote by~$\TFT(\fs_i)$.

\item \label{N}
For $i=1,\ldots,g-1$, the Dehn twist $\fn_i$ acts by postcomposition with the morphism
$\id_{L^{\ot (i-1)}} \ot \fN \ot \id_{L^{\ot (g-i-1)}}$, a map that we denote by~$\TFT(\fn_i)$.

\item \label{Z}
For $j=1,\dots,n-1$, the Dehn twist $\fz_j$ acts via the map 
\[\TFT(\fz_j) \colon \Hom_{\CC}(X_n \ot \cdots \ot  X_1, L^{\ot g}) \to 
\Hom_{\CC}(X_n \ot \cdots \ot  X_1, L^{\ot g})\]
which is the unique linear map that makes the diagram
\begin{center}
\begin{tikzcd}
\Hom_{\CC}(X_j \! \ot \cdots \ot \! X_1, 
\prescript{*}{}{(X_n \ot \dots \ot X_{j+1})} \ot  L^{\ot g}) \arrow[r] \arrow[d] & \Hom_{\CC}(X_n \ot \cdots \ot  X_1, L^{\ot g}) \arrow[d, "\TFT(\fz_j)"] \\
\Hom_{\CC}(X_j \! \ot \cdots \ot \! X_1, 
\prescript{*}{}{(X_n \ot \dots \ot X_{j+1})} \ot  L^{\ot g}) \arrow[r] & \Hom_{\CC}(X_n \ot \cdots \ot  X_1, L^{\ot g})
\end{tikzcd}
\end{center}
commutative. In this diagram, the horizontal arrows are the functor adjunctions from Paragraph~\ref{Fin}, and the left vertical map is given by
\begin{align*}
&h \mapsto (\theta_{\prescript{*}{}{(X_n \ot \dots \ot X_{j+1})} \ot L^{\ot (g-1)}} \ot \fT)
\circ\fN^l_{\prescript{*}{}{(X_n\ot \dots \ot X_{j+1})} \ot L^{\ot (g-1)},L} \circ h.
\end{align*}
\end{parlist}

We now explain how some of the formulas come about. For any edge of a trivalent net, there is an automorphism called the twist (cf.~\cite[Par.~4.3, p.~350]{L2}) that corresponds to a Dehn twist around the corresponding curve on the surface (cf.~\cite[Prop.~2.2, p.~319]{L2}). Recall that, as already mentioned in Paragraph~\ref{Dehn}, Dehn twists as defined here are inverse Dehn twists as defined in~\cite{L1} and~\cite{L2}. By~\cite[Par.~8.1, No.~(x), p.~374]{L2}, the twist acts by applying the ribbon twist to the corresponding variable. Because the isomorphisms between the spaces~$T_1$ to~$T_5$ introduced in Paragraph~\ref{BlockSpc} and our block space are natural, this implies Claim~\ref{D}.
Similarly, the generators~$\ft_i$,~$\fn_j$, and~$\fz_{l}$ act by applying a ribbon twist to the internal variables $A_i$, $N_j$, and~$Z_l$ (cf.~\cite[p.~493]{L1}). We calculate the action of these generators explicitly in the case $g=3$ and $n=2$ to explain Claim~\ref{T}, Claim~\ref{N}, and Claim~\ref{Z}.

We first calculate the action of $\ft_1$: On the space~$T_1$, it acts by postcomposition with the twist~$\theta_{A_1}$ on the last tensor factor, which corresponds to postcomposition with $\theta_{A_1} \ot \id_{A_2^*}$ on the last tensor factor on the spaces~$T_2$ and~$T_3$. Under the isomorphism to the space~$T_4$, this corresponds to postcomposition with  $\id_{\prescript{*}{}{A_1}}\ot \theta_{A_1}\ot \id_{A_2^*}$ on the first tensor factor. 
This in turn corresponds  on the space~$T_5$ to postcomposition with 
$\id_{\prescript{*}{}{A_1}} \ot \theta_{A_1} \ot \id_{A_2^* \ot A_2 \ot A_3^* \ot A_3}$. In view of the definition of~$\fT$ in Paragraph~\ref{Coend}, this becomes postcomposition with $\fT\ot \id_{L\ot  L}$ on our block space $\Hom(X_2 \ot X_1, L \ot L \ot L)$. Very similar calculations yield the claim for~$\ft_2$ and~$\ft_3$.

Next, we calculate the action of $\fn_1$: On the space~$T_1$, it acts by postcomposition with the twist~$\theta_{N_1}$ on the first tensor factor. Under the isomorphism to the space~$T_2$, this corresponds to postcomposition with $\id_{\prescript{*}{}{A_1}} \ot \theta_{N_1}$ on the first tensor factor, which does not change under the isomorphism to the space~$T_3$. On the space~$T_4$, it becomes postcomposition with $\id_{\prescript{*}{}{A_1}} \ot \theta_{A_1 \ot A_2^*}$ on the first tensor factor. Under the isomorphism to the space~$T_5$, this corresponds to postcomposition with 
$\id_{\prescript{*}{}{A_1}} \ot \theta_{A_1 \ot A_2^*} \ot \id_{A_2 \ot A_3^*}$. This is equal to postcomposition with 
$\id_{\prescript{*}{}{A_1}} \ot ((\theta_{A_1} \ot \theta_{A_2^*}) \circ c_{A_2^*,A_1} \circ c_{A_1,A_2^*})\ot \id_{A_2\ot  A_3^*\ot  A_3}$,
which corresponds on our block space~$\Hom(X_2 \ot X_1, L \ot L \ot L)$ to postcomposition with $\fN \ot \id_L$. The calculations for $\fn_2$ are similar.

For the action of $\fz_{1}$, we proceed differently: On the space~$T_1$, it acts by postcomposition with~$\theta_{Z_1} \ot \id_{A_3}$ on the 
tensor factor~$\Hom(X_1, Z_1 \ot A_3)$, which on the spaces~$T_2$ and~$T_3$ becomes postcomposition with~$\theta_{Z_1}$  on the fourth and the third tensor factor, respectively. Under the isomorphism to the space~$T_4$, this corresponds to precomposition with $\id_{X_2} \ot \theta_{X_1 \ot \prescript{*}{}{A_3}}$. This space is isomorphic to 
\[\int^{A_i}\Hom(X_1\ot \prescript{*}{}{A_3},\prescript{*}{}{X_2} \ot A_1^* \ot A_1 \ot A_2^* \ot A_2 \ot A_3^*),\]
where the action becomes precomposition with $\theta_{X_1\ot \prescript{*}{}{A_3}}$, which equals postcomposition with $\theta_{\prescript{*}{}{X_2} \ot A_1^* \ot A_1 \ot A_2^* \ot A_2 \ot A_3^*}$. Under the isomorphism to the space 
\[\int^{A_i}\Hom(X_1,\prescript{*}{}{X_2} \ot A_1^* \ot A_1 \ot A_2^* \ot A_2 \ot A_3^* \ot A_3),\]
it becomes postcomposition with $\theta_{\prescript{*}{}{X_2} \ot A_1^* \ot A_1 \ot A_2^* \ot A_2 \ot A_3^*}\ot \id_{A_3}$, which equals postcomposition with 
\[((\theta_{\prescript{*}{}{X_2} \ot A_1^* \ot A_1 \ot A_2^* \ot A_2}\ot \theta_{A_3^*})\circ c_{A_3^*,\prescript{*}{}{X_2} \ot A_1^* \ot A_1 \ot A_2^* \ot A_2}\circ c_{\prescript{*}{}{X_2} \ot A_1^* \ot A_1 \ot A_2^* \ot A_2,A_3^*}) \ot \id_{A_3}.\]
Under the isomorphism to the space $\Hom(X_1,\prescript{*}{}{X_2} \ot L^{ \ot 3})$, this equals postcomposition with $(\theta_{\prescript{*}{}{X_2} \ot L \ot L} \ot \fT)\circ \fN^l_{\prescript{*}{}{X_2} \ot L \ot L,L}$. But this means precisely that $\TFT(\fz_1)$ acts as asserted on our block space $\Hom(X_2 \ot X_1,L^{\ot 3})$.

Using a fusing morphism (cf.~\cite[Eq.~(6.8), p.~359]{L2}), the braiding of two neighboring boundary components corresponds to a braiding morphism of a trivalent net (cf.~\cite[Par.~4.3, p.~350f]{L2}), which in turn corresponds to the braiding of the corresponding labels of the boundary components on our block space (cf.~\cite[Par.~8.1, No.~(viii)~and~(ix), p.~373f]{L2}). In this way, we arrive at Claim~\ref{B} and Claim~\ref{Q}. Claim~\ref{DD} follows from Claim~\ref{D}, because we have seen in Paragraph~\ref{MoreSpecDehn} that 
$[\fq_{j,j+1}] = [\fd_{j,j+1} \fd_j^{-1} \fd_{j+1}^{-1}]$, and in a ribbon category we have
\[c_{X_j,X_{j+1}} \circ c_{X_{j+1},X_j} = 
(\theta_{X_{j+1}}^{-1} \ot \theta_{X_j}^{-1}) \circ \theta_{X_{j+1} \ot X_j}
= \theta_{X_{j+1} \ot X_j} \circ (\theta_{X_{j+1}}^{-1} \ot \theta_{X_j}^{-1}). \]

The element~$\fs_i$ corresponds to a switch (cf.~\cite[Prop.~2.2, p.~319]{L2}) that acts by~$\fS$ on the corresponding tensor factor of~$L^{\ot g}$ (cf.~\cite[Prop.~8.8, p.~391]{L2}). This establishes Claim~\ref{S} and concludes our discussion about how the different elements of the mapping class group act.

It is important to note that, as we use a different net for the description of the surface, our formulas for the actions of the generators of the mapping class group on the block space are different from those in~\cite[Par.~4.5, p.~494f]{L1}. Starting from the net in~\cite[Fig.~5, p.~491]{L1} and using the internal variables called there~$E_i$ instead of those called there~$D_i$ to form the coends in the last step would give formulas similar to ours. 
\pagebreak

To facilitate the comparison, we include a small table that relates our notation to the one used in~\cite{L1}:
\begin{center}
\begin{tabular}{|c|c|c|c|c|c|c|c|c|c|} \hline
Present notation & $\Gamma_{g,n}$ & $L$ & $\fN'$ & $\fN^l_{X,L}$ & $\fS$ & $\fT$ & $\fb_{l,l+1}$ & $\fq_{l,l+1}$  & $\fs_i$ \\ \hline 
\cite{L1} & $M'_{g,n}$ & ${\mathbf f}$ & $\Omega$ & $\Omega^l_{X,\mathbf{f}}$ & $S$ & $T$ & $\omega_l$ & $\omega_l^2$ &  $S_i$ \\ \hline 
\end{tabular}
\end{center}
It must be emphasized that the morphisms in the first row are only analogous, but not strictly equal, to the ones in the second row. 

In the case where~$n=1$, i.e., when there is only one boundary component, the projective action on the block space can be obtained by postcomposition from a projective action 
of~$\Gamma_{g,1}$ on~$L^{\ot g}$: For a mapping class~$[\psi] \in \Gamma_{g,1}$, a representative~$\TFT(\psi)$ of the associated projective class~$[\TFT(\psi)]$ is a natural equivalence from the contravariant Hom-functor 
\[X \mapsto \TFT(\Sigma_{g,1}^{X}) = \Hom_{\CC}(X, L^{\ot g})\]
to itself, which according to the Yoneda lemma (cf.~\cite[Chap.~III, Sec.~2, p.~61]{ML2}; \cite[Par.~1.15, p.~37]{P}) is given by postcomposition with an automorphism of~$L^{\ot g}$. For the generators of the mapping class group whose projective action is described in the list above, this can be seen explicitly: The only generator that appears in this case and is not already given by postcomposition is~$\fd_1$, which according to Claim~\ref{D} is given by precomposition with~$\theta_X$. However, by the naturality of the twist, this is equal to postcomposition with~$\theta_{L^{\ot g}}$.

\subsection{Modular functors} \label{ModFunct}
The various mapping class group representations just described are not unrelated, but rather together form a modular functor (cf.~\cite[Par.~8.1, p.~372ff]{L2}). In particular, they are compatible with gluing of surfaces (cf.~\cite[Par.~8.1, No.~(xi), p.~374]{L2}). We will need only one very special instance of this general property: Consider the last two boundary components of the surface~$\Sigma_{g,n+2}$. They carry an orientation that is induced from the orientation of the surface, which is opposite to the orientation of the curves~$\rho_{n+1}$ and~$\rho_{n+2}$ in Paragraph~\ref{Surf}. Up to isotopy, there is a unique orientation-reversing diffeomorphism between these two boundary components that maps the distinguished point on one boundary component to the distinguished point on the other. If we identify the two boundary components along this diffeomorphism, the arising quotient space is diffeomorphic to~$\Sigma_{g+1,n}$, so that we have effectively attached a handle, which we consider as the first one. A diffeomorphism 
\mbox{$\varphi \colon \Sigma_{g,n+2} \to \Sigma_{g,n+2}$} that restricts to the identity on the last two boundary components then induces a diffeomorphism \mbox{$\psi \colon \Sigma_{g+1,n} \to \Sigma_{g+1,n}$}. For~$\tau \deq p([\varphi])^{-1} \in S_{n+2}$, our hypothesis means that~$\tau$ fixes~$n+1$ and~$n+2$.

Suppose we are given $n$ objects~$X_1,\ldots,X_n$ and another object~$X$. From the adjunctions recalled in Paragraph~\ref{Fin}, we get an isomorphism
\[\Hom_\CC(X^{*} \ot X^{**} \ot X_n \ot \dots \ot X_1, L^{\ot g}) \to 
\Hom_\CC(X_n \ot \dots \ot X_1, X^* \ot X \ot L^{\ot g}).\]
If we postcompose with $\iota_X \ot \id_{L^{\ot g}}$, we obtain a morphism
\[\Hom_\CC(X^* \ot X^{**} \ot X_n \ot \dots \ot X_1, L^{\ot g}) \to 
\Hom_\CC(X_n \ot \dots \ot X_1, L^{\ot (g+1)})\]
that we call the handle gluing homomorphism. By construction, this homomorphism is 
natural with respect to the objects~$X_1, \ldots, X_n$ of~$\CC$ in the sense that for given morphisms $f_i\colon X_i \to Y_i$, the diagram 
\begin{center}
\begin{tikzcd}
\TFT(\Sigma_{g,n+2}^{Y_1,\ldots, Y_n, X^{**}, X^*}) \arrow{r} 
\arrow[d, "\circ (\id_{X^* \ot X^{**}} \ot f_n \ot \dots \ot f_1)"']  & \TFT(\Sigma_{g+1,n}^{Y_1,\ldots, Y_n}) 
\arrow[d, "\circ (f_n \ot \dots \ot f_1)"]\\
\TFT(\Sigma_{g,n+2}^{X_1,\ldots, X_n, X^{**}, X^*})
\arrow[r] & \TFT(\Sigma_{g+1,n}^{X_1,\ldots, X_n})
\end{tikzcd}
\end{center}
commutes.

The special case of the gluing property under consideration here then states that the diagram
\begin{center}
\begin{tikzcd}
\TFT(\Sigma_{g,n+2}^{X_1,\ldots, X_n, X^{**}, X^*}) \arrow{r} \arrow[d, "\TFT(\varphi)"]  & \TFT(\Sigma_{g+1,n}^{X_1,\ldots, X_n}) \arrow[d, "\TFT(\psi)"]\\
\TFT(\Sigma_{g,n+2}^{X_{\tau(1)},\ldots, X_{\tau(n)}, X^{**}, X^*})
\arrow[r] & \TFT(\Sigma_{g+1,n}^{X_{\tau(1)},\ldots, X_{\tau(n)}})
\end{tikzcd}
\end{center}
commutes for suitably chosen representatives~$\TFT(\varphi)$ and~$\TFT(\psi)$ within their respective projective classes. As above, $\TFT(\varphi)$ and~$\TFT(\psi)$ are viewed here as natural transformations of the corresponding functors, so that the nonzero scalar used in passing to a different representative does not depend on the spaces~$X_1,\ldots, X_n$.  

To illustrate this property, we consider the example $\varphi = \fd_{n+1}$, which also plays an important role later. Then the corresponding map on the glued surface is $\psi=\ft_1$. The actions of~$\varphi$ and~$\psi$ are described in the list in Paragraph~\ref{MapClRep}. By its definition in Paragraph~\ref{Coend}, we have
$\fT \circ\iota_X = \iota_X \circ (\id_{X^*} \ot \theta_X) 
= \iota_X \circ (\theta_{X}^* \ot \id_X)$, where the second equality follows from the dinaturality of~$\iota$. Therefore, the diagram
\begin{center}
\begin{tikzcd}[column sep=huge]
\Hom_\CC(X_n \ot \dots \ot X_1, X^* \ot X\ot L^{\ot g}) 
\arrow[d, "(\theta_{X}^* \ot \id_X\ot \id_{L^{\ot g}}) \circ"] 
\arrow[r, "(\iota_X\ot \id_{L^{\ot g}}) \circ"]  & 
\Hom_\CC(X_n \ot \dots \ot X_1, L^{\ot (g+1)}) \arrow[d, "(\fT\ot \id_{L^{\ot g}}) \circ"]\\
\Hom_\CC(X_n \ot \dots \ot X_1, X^* \ot X \ot L^{\ot g}) 
\arrow[r, "(\iota_X\ot \id_{L^{\ot g}}) \circ"'] & 
\Hom_\CC(X_n \ot \dots \ot X_1, L^{\ot (g+1)})
\end{tikzcd}
\end{center}
commutes.

Combining this with the naturality of the adjunction, we see that the gluing property holds in this case. 

We note that the choice $\varphi = \fd_{n+2}$ leads to the same map on the glued surface, namely~$\psi=\ft_1$. A very similar reasoning shows that the gluing property also holds in this case.

\subsection{The case of the torus} \label{ActTorus}
In order to understand how the results of this article generalize the results of our previous one (cf.~\cite{LMSS1}), it will be important to consider the projective representation reviewed in Paragraph~\ref{MapClRep} in the case of the torus, where~$g=1$, and the case where $\CC$ is the category of representations of a Hopf algebra $A$ with the properties described in Paragraph~\ref{Fact}. Important ingredients of this projective representation are the endomorphisms~$\fS$ and~$\fT$ of~$L=A^*$, whose explicit form we have determined in Paragraph~\ref{CoendHopf}. In our previous article, we have used the same symbols for endomorphisms of~$A$, which we will now denote by $\hS$ and~$\hT$ and which are given by
\[\hS(a) = \rho(a Q^1) S(Q^2) \qquad \text{and} \qquad \hT(a) = va\]
(cf.~\cite[Sec.~3, p.~410]{LMSS1}). In order to understand how these endomorphisms are related, we use a variant of the Radford map 
\[\iota \colon A \to A^*,~a \mapsto \iota(a)\]
defined by $\iota(a)(a') = \rho(aa')$, which was also introduced on the page just cited. Here~$\rho$ is the right integral already considered in Paragraph~\ref{CoendHopf}. As recalled in~\cite[loc.~cit.]{LMSS1}, $\rho$ is contained in the space 
\begin{align*}
\bar{C}(A) \deq \{\varphi \in A^* \mid \varphi(aa') = 
\varphi(S^{2}(a') a) \; \text{for all} \; a,a' \in A\},
\end{align*}
which can be viewed as the set of invariants $(A^*)^A$ for the left coadjoint action discussed in Paragraph~\ref{CoendHopf}. The variant of the Radford map that we will need is the map
$\bar{\iota} \deq \iota \circ S^2$, which then satisfies
\[\bar{\iota}(a)(a') = \iota(S^2(a))(a') = \rho(S^2(a)a') = \rho(a'a).\]
With the help of~$\bar{\iota}$, we see that $\hS$ and~$\hT$ are conjugate to $\fS$ and~$\fT$:
\begin{Lemma} \label{LemEquiv} 
The diagrams
\begin{center}
\begin{tikzcd}
A^* \arrow[r, "\fS"] & A^* \\
A \arrow[u, "\bar{\iota}"] \arrow[r, "\hS"'] & A \arrow[u, "\bar{\iota}"']
\end{tikzcd}
\qquad \text{and} \qquad
\begin{tikzcd}
A^* \arrow[r, "\fT"] & A^* \\
A \arrow[u, "\bar{\iota}"] \arrow[r, "\hT"'] & A \arrow[u, "\bar{\iota}"']
\end{tikzcd}
\end{center}
commute. Moreover, $\bar{\iota}$ is an isomorphism of $A$-modules, where~$A^*$ carries the left coadjoint action and $A$ carries the left adjoint action of the coopposite Hopf algebra, given by $a.a' \deq a_{(2)} a' S^{-1}(a_{(1)})$.
\end{Lemma}
\begin{proof}
To show the commutativity of the first diagram, we use that $\rho \in \bar{C}(A)$ to compute
\begin{align*}
\bar{\iota}(\hS(a))(a') &= \rho(a Q^1) \bar{\iota}(S(Q^2))(a') = \rho(a Q^1) \rho(a' S(Q^2)) \\
&= \rho(S^2(Q^1)a) \rho(S^3(Q^2)a') = \rho(Q^1 a) \rho(S(Q^2)a') \\
&= \bar{\iota}(a)(Q^1) \rho(S(Q^2)a') = \fS(\bar{\iota}(a))(a').
\end{align*}
The proof of the commutativity of the second diagram is even simpler: We have
\begin{align*}
\bar{\iota}(\hT(a))(a') &= \bar{\iota}(va)(a') = \rho(a'va) 
= \bar{\iota}(a)(a'v) = \fT(\bar{\iota}(a))(a') 
\end{align*}
by Proposition~\ref{CoendHopfProp}. Because $\rho$ is a Frobenius homomorphism, $\iota$ and~$\bar{\iota}$ are bijective. To see that~$\bar{\iota}$ is linear with respect to the specified $A$-actions, we compute
\begin{align*}
\bar{\iota}(a.a')(a'') &= \rho(a''(a.a')) = 
\rho(a'' a_{(2)} a' S^{-1}(a_{(1)})) = \rho(S(a_{(1)}) a'' a_{(2)} a')\\
&= \bar{\iota}(a')(S(a_{(1)}) a'' a_{(2)}) = (a.(\bar{\iota}(a')))(a'')
\end{align*}
and get the assertion.
\end{proof}

By~\cite[Prop.~4.3, p.~37]{SZ}, we have 
$\hS \circ \hT \circ \hS = \rho(v) \; \hT^{-1} \circ \hS \circ \hT^{-1}$. The preceding lemma therefore implies that we also have
$\fS \circ \fT \circ \fS = \rho(v) \; \fT^{-1} \circ \fS \circ \fT^{-1}$. It then follows from our discussion of the braid group~$B_3$ in Paragraph~\ref{Torus} that the assignment
\[s \mapsto \fS \qquad \text{and} \qquad t \mapsto \fT \]
yields a projective representation of~$B_3$ in~$\Aut_A(L)$. If~$X$ is an $A$-module, postcomposition with~$\fS$ and~$\fT$ therefore leads to a projective representation of~$B_3$ on the space~$\Hom_A(X, L)$. On the other hand, we have seen in Paragraph~\ref{Torus} that the assignment $r \mapsto \fr_1$ and $t \mapsto \ft_1$ yields an isomorphism between the braid group~$B_3$ and the mapping class group~$\Gamma_{1,1}$. By construction, this isomorphism maps~$s$ to~$\fs_1$. Therefore, if we transport the projective action of~$B_3$ to~$\Gamma_{1,1}$ along this isomorphism, we obtain exactly the projective representation considered in Paragraph~\ref{MapClRep} in the case~$g=1$ and~$n=1$. 

It is instructive to see why the 2-chain relation $(\fr_1 \ft_1)^6 = \fd_1$ in the mapping class group~$\Gamma_{1,1}$, which we discussed in Paragraph~\ref{Torus}, is satisfied in our situation. We have already used in Paragraph~\ref{Torus} that $(\fr_1 \ft_1)^6 = \fs_1^{-4}$.
Now suppose that \mbox{$f \in \Hom_A(X,L)$} is an $A$-linear map. For $x\in X$, we choose 
$a \in A$ so that $f(x) = \bar{\iota}(a)$. For example from~\cite[Prop.~3.2, p.~411]{LMSS1}
combined with~\cite[Lem.~3.3, p.~412]{LMSS1}, we know that
\[\hS^4(a) = ((\rho \ot \rho)(Q))^2 S(v^{-1}_{(1)}) a v^{-1}_{(2)}. \]
From the lemma above and the fact that~$\rho \in \bar{C}(A)$, we then get
\begin{align*}
((\fs_1^{4}.f)(x))(a') &= \fS^4(f(x))(a') = \fS^4(\bar{\iota}(a))(a') 
= \bar{\iota}(\hS^4(a))(a') = \rho(a' \hS^4(a)) \\
&= ((\rho \ot \rho)(Q))^2 \; \rho(a'S(v^{-1}_{(1)}) a v^{-1}_{(2)}) \\
&= ((\rho \ot \rho)(Q))^2 \; \rho(S^2(v^{-1}_{(2)}) a' S(v^{-1}_{(1)})a) \\
&= ((\rho \ot \rho)(Q))^2 \; f(x)(S^2(v^{-1}_{(2)}) a' S(v^{-1}_{(1)})) \\
&= ((\rho \ot \rho)(Q))^2 \; (S(v^{-1}).f(x))(a') 
= ((\rho \ot \rho)(Q))^2 \; (v^{-1}.f(x))(a') \\
&= ((\rho \ot \rho)(Q))^2 \; f(v^{-1}.x)(a'). 
\end{align*}
On the other hand, we have
\[((\fd_1^{-1}.f)(x))(a') = ((f \circ \theta_X^{-1})(x))(a') = f(v^{-1}.x)(a'),\]
so that the actions of $\fs_1^{4}$ and~$\fd_1^{-1}$ agree up to a scalar, as required.

In the case $g=1$, but $n=0$, the definition in Paragraph~\ref{BlockSpc} is to be understood in such a way that the action is on the block space~$\Hom_A(\K, L)$, because the base field is the unit object of the category, as explained in Paragraph~\ref{Fact}. For any left \mbox{$A$-module~$Y$}, the map $\Hom_A(\K, Y) \to Y^A,~f \mapsto f(1_\K)$ yields a bijection with the space of invariants 
\[Y^A = \{y \in Y \mid a.y = \varepsilon(a) y \; \text{for all} \; a \in A\}.\]
As we said above, in the case $Y=L$ we have $L^A = \bar{C}(A)$, while in the case $Y=A$, endowed with the action described in Lemma~\ref{LemEquiv}, we have $Y^A = Z(A)$, the center of~$A$. In view of its $A$-linearity, $\bar{\iota}$ restricts to an isomorphism between~$Z(A)$ and~$\bar{C}(A)$.

In the case $X=\K$, we have that $\theta_X$ is the identity map, because~$\varepsilon(v)=1$. This means that~$\fd_1$ acts trivially on~$\Hom_A(\K, L)$. By the 2-chain relation that we have just checked explicitly, this implies that~$\fs_1^{4}$ acts trivially. This in turn means, in view of the discussion in Paragraph~\ref{Torus}, that the projective action of the braid group~$B_3$ descends to a projective action of the modular group~$\SL(2,\Z)$ on~$\bar{C}(A)$, as it is required for the construction in Paragraph~\ref{MapClRep}, because
$\Gamma_1 \cong \SL(2,\Z)$. Via~$\bar{\iota}$, this projective action on~$\bar{C}(A)$ is isomorphic to an action of~$\SL(2,\Z)$ on the center~$Z(A)$. This is the action that was considered in~\cite[Cor.~3.4, p.~412]{LMSS1}.

Let us mention that using~$\bar{\iota}$ is not the only way to relate the projective representations on~$\bar{C}(A)$ and~$Z(A)$. Other ways are discussed in~\cite[Par.~9.1, p.~87ff]{SZ}.

\section{Derived functors} \label{Sec:DerFunct}
\subsection{Projective resolutions} \label{Ext}
There are various approaches to the definition of the Ext-functors in abelian categories. We assume here that our category satisfies the assumptions listed in Paragraph~\ref{Fin}; one of these assumptions was that it has enough projectives. In this case, we can use the approach via projective resolutions described in~\cite[Chap.~VII, \S~7, p.~182ff]{Mi} and denoted there by~$\overline{\Ext}$, but here just denoted by~$\Ext$. In this approach, the group~$\Ext^m(X,Y)$ for two objects~$X$ and~$Y$ of~$\CC$ is defined by choosing a projective resolution 
\[X \xlongleftarrow{\xi} P_0 \xlongleftarrow{d_1} P_1 \xlongleftarrow{d_2} P_2
\xlongleftarrow{d_3} \cdots\]
of~$X$, which exists by the above hypothesis. The group~$\Ext^m(X,Y)$ is then defined as the $m$-cohomology group of the cochain complex 
\[\Hom_\CC(P_0,Y) \xlongrightarrow{} \Hom_\CC(P_1,Y) \xlongrightarrow{} \Hom_\CC(P_2,Y)
\xlongrightarrow{} \cdots\]
of abelian groups. Although the notation does not reflect this, the definition depends on the chosen resolution; different resolutions lead to $\Ext$-groups that are canonically isomorphic, but not equal (cf.~\cite[Chap.~XII, \S~9, p.~390]{ML1}; \cite[Sec.~1, p.~402]{LMSS1}). Note that 
$\Ext^0(X,Y) \cong \Hom_\CC(X,Y)$, because the Hom-functor is left exact.

The following lemma about the ribbon structure~$\theta$ will be important in the sequel:
\begin{Lemma} \label{NatHomotop}
Suppose that 
\[\1 \xlongleftarrow{\xi} P_0 \xlongleftarrow{d_1} P_1 \xlongleftarrow{d_2} P_2
\xlongleftarrow{d_3} \cdots\]
is a projective resolution of the unit object. Then there exists, for every object~$X \in \CC$ and every $m \ge 0$, a morphism $h_m(X)\colon P_m \ot X \to P_{m+1} \ot X$ such that
\[\theta_{P_m \ot X} - \id_{P_m} \ot \theta_{X} 
= (d_{m+1} \ot \id_X) \circ h_m(X) + h_{m-1}(X) \circ (d_m \ot \id_X)\]
for all $m \ge 1$, and $\theta_{P_0 \ot X} - \id_{P_0} \ot \theta_{X} = (d_1 \ot \id_X) \circ h_0(X)$. These morphisms are natural in~$X$ in the sense that, for a morphism $f\colon X \to Y$, the diagram
\begin{center}
\begin{tikzcd}[column sep = huge]
P_m \ot X  \arrow[d, "h_m(X)"'] \arrow[r, "\id_{P_m} \ot f"] & P_m \ot Y \arrow[d, "h_m(Y)"]  \\
P_{m+1} \ot X  \arrow[r, "\id_{P_{m+1}} \ot f"'] & P_{m+1} \ot Y 
\end{tikzcd}
\end{center}
commutes.
\end{Lemma}
\begin{proof}
We denote by~$\Rex(\CC, \CC)$ the category of $\K$-linear right exact functors from~$\CC$ to~$\CC$, whose morphisms are natural transformations. By~\cite[Cor.~2.6, p.~466]{Sh2}, the functor $X \mapsto P_m \ot X$ is a projective object in the $\K$-linear category~$\Rex(\CC, \CC)$, and therefore these functors form a projective resolution of the functor~$X \mapsto \1 \ot X \cong X$ in the category~$\Rex(\CC, \CC)$. The result now follows from the comparison theorem (cf.~\cite[Thm.~6.16, p.~340]{Ro2}) applied to the category~$\Rex(\CC, \CC)$.
\end{proof}

\vspace{2mm}

If~$\CC$ is the category of finite-dimensional modules over a factorizable Hopf algebra discussed in Paragraph~\ref{Fact}, the preceding proof can be given a more explicit form. By the Eilenberg-Watts theorem already mentioned in the proof of Proposition~\ref{RightAdj}, right exact functors can in this case be represented by tensoring with bimodules. In particular, the functor~$X \mapsto P_m \ot X$ is represented as~$X \mapsto (P_m \ot A) \ot_A X$, where the space~$P_m \ot A$ carries the $A$-bimodule structure
\vspace{1mm}
\[a.(p \ot a').a'' = a_{(1)}.p \ot a_{(2)}a'a'' \vspace{1mm} \]
for $p \in P_m$ and~$a, a', a'' \in A$. We then have two liftings of~$\theta_A$, which appear in the diagram
\begin{center}
\begin{tikzcd}[column sep = large] {}
A \arrow{d}{\theta_A} & P_0 \ot A \arrow{l}[swap]{\xi \ot \id_A} 
\arrow[d, xshift=0.7ex, "\id_{P_0} \ot \theta_{A}"] 
\arrow[d, xshift=-0.6ex, "\theta_{P_0 \ot A}"'] &
P_1 \ot A \arrow{l}[swap]{d_1 \ot \id_A} 
\arrow[d, xshift=0.7ex, "\id_{P_1} \ot \theta_{A}"] 
\arrow[d, xshift=-0.6ex, "\theta_{P_1 \ot A}"'] &
\cdots \arrow{l}[swap]{d_2 \ot \id_A}\\
A & P_0 \ot A \arrow{l}[swap]{\xi \ot \id_A} & P_1 \ot A \arrow{l}[swap]{d_1 \ot \id_A} & \cdots
\arrow{l}[swap]{d_2 \ot \id_A}
\end{tikzcd}
\end{center}

\vspace{2mm}

and have the explicit form
\vspace{1mm}
\[\theta_{P_m \ot A}(p \ot a) = v_{(1)}.p \ot v_{(2)}a \qquad \text{and}\qquad
(\id_{P_m} \ot \theta_{A})(p \ot a) = p \ot va \vspace{1mm} \]
in terms of the ribbon element~$v$.

\vspace{3mm}

Now we can apply the more standard comparison theorem for bimodules, considered as 
\mbox{$A \ot A^{\op}$}-modules, to see that these two liftings are chain-homotopic via a chain homotopy $h'_m\colon P_m \ot A \to P_{m+1} \ot A$.
The mappings $h_m(X)$ are then induced from the mappings $h'_m \ot_A \id_X$ via the isomorphism
$(P_m \ot A) \ot_A X \cong P_m \ot X$.

\subsection{Derived block spaces} \label{MappClassExt}
Our goal now is to extend the projective action of the pure mapping class group~$\Pu\Gamma_{g,n}$ on the block spaces
\[\TFT(\Sigma_{g,n}^{X_1,\ldots, X_n}) \deq
\Hom_{\CC}(X_n \ot \dots \ot X_1, L^{\ot g})\]
to the derived block spaces
\[\TFT^m(\Sigma_{g,n}^{X_1,\ldots, X_n}) \deq
\Ext^m(X_n \ot \dots \ot X_1, L^{\ot g}).\]
To do this, we consider the surface~$\Sigma_{g,n+1}$ with one additional boundary component, and choose a projective resolution
\[\1 \xlongleftarrow{\xi} P_0 \xlongleftarrow{d_1} P_1 \xlongleftarrow{d_2} P_2
\xlongleftarrow{d_3} \cdots\]
of the unit object. As we mentioned in Paragraph~\ref{Fin}, tensoring is exact and preserves projectives, so that 
\[\1 \ot X_n \ot \dots \ot X_1 \xlongleftarrow{} P_0\ot X_n \ot \dots \ot X_1 \xlongleftarrow{} P_1\ot X_n \ot \dots \ot X_1 \xlongleftarrow{} \cdots\]
is a projective resolution of~$X_n \ot \cdots \ot X_1 \cong \1\ot X_n \ot \cdots \ot X_1$. 
If we use the abbreviation~$\Gamma_{g,n+1}(n+1)$ for the 
mapping class group~$\Gamma_{g,n+1}(\Img(\rho_{n+1}))$, an element~$[\psi]$ of this group fixes the $(n+1)$st boundary component. Then the permutation $\tau \deq p([\psi])^{-1} \in S_{n+1}$ introduced in Paragraph~\ref{MapClGr} fixes~$n+1$, and therefore each representative~$\TFT(\psi)$ of the associated projective class~$[\TFT(\psi)]$ yields by naturality a cochain homomorphism between the cochain complexes 
\[\TFT(\Sigma_{g,n+1}^{X_1, \ldots, X_n, P_m}) = 
\Hom_{\CC}(P_m \ot X_n \ot \cdots \ot X_1, L^{\ot g})\] 
and
\[\TFT(\Sigma_{g,n+1}^{X_{\tau(1)}, \ldots, X_{\tau(n)}, P_m}) = 
\Hom_{\CC}(P_m \ot X_{\tau(n)} \ot \cdots \ot X_{\tau(1)}, L^{\ot g})\]
and so induces a homomorphism
\[\TFT^m(\psi) \colon \Ext^m(X_n \ot \dots \ot X_1, L^{\ot g}) \to 
\Ext^m(X_{\tau(n)} \ot \cdots \ot X_{\tau(1)}, L^{\ot g})\]
between the derived block spaces. Choosing a different representative of the projective class clearly rescales~$\TFT^m(\psi)$ by a nonzero scalar, so that the projective class~$[\TFT^m(\psi)]$ is well-defined. 

However, we have associated this homomorphism between derived block spaces with a mapping class
$[\psi] \in \Gamma_{g,n+1}(n+1)$, whereas the homomorphisms between the original block spaces were associated with a mapping class $[\psi] \in \Gamma_{g,n}$. As we will show now, we can also associate homomorphisms between the derived block spaces with a mapping class~$[\psi] \in \Gamma_{g,n}$, namely by choosing a preimage in~$\Gamma_{g,n+1}(n+1)$ under the homomorphism~$D_{n+1}$ defined at the end of Paragraph~\ref{Cap}. For this, we obviously need to show that the arising homomorphisms are independent of the chosen preimage. We begin with a few auxiliary results: 
\begin{Lemma} \label{LemDehnTw1}
For the Dehn twist~$\fd_{n+1} \in \Gamma_{g,n+1}(n+1)$, we have~$[\TFT^m(\fd_{n+1})] = [\id]$.
\end{Lemma}
\begin{proof}
According to Paragraph~\ref{MapClRep}, the Dehn twist~$\fd_{n+1}$ acts by precomposition with the morphism $\theta_{X_{n+1}} \ot \id_{X_n \ot \dots \ot X_{1}}$. By~\cite[Lem.~XIV.3.3, p.~350]{Ka}, we have $\theta_\1 = \id_\1$. 
As we have discussed already at the end of Paragraph~\ref{Ext}, the diagram
\begin{center}
\begin{tikzcd} {}
\1 \arrow{d}{\theta_\1} & P_0 \arrow{l}[swap]{\xi}\arrow{d}{\theta_{P_0}} &
P_1 \arrow{l}[swap]{d_1} \arrow{d}{\theta_{P_1}} &
\cdots\arrow{l}[swap]{d_2}\\
\1 & P_0\arrow{l}[swap]{\xi} & P_1 \arrow{l}[swap]{d_1} & \cdots
\arrow{l}[swap]{d_2}
\end{tikzcd}
\end{center}
commutes by the naturality of the twist. On the other hand, it is obvious that the diagram
\begin{center}
\begin{tikzcd} {}
\1 \arrow{d}{\id_\1} & P_0 \arrow{l}[swap]{\xi}\arrow{d}{\id_{P_0}} &
P_1 \arrow{l}[swap]{d_1} \arrow{d}{\id_{P_1}} &
\cdots\arrow{l}[swap]{d_2}\\
\1 & P_0\arrow{l}[swap]{\xi} & P_1 \arrow{l}[swap]{d_1} & \cdots
\arrow{l}[swap]{d_2}
\end{tikzcd}
\end{center}
commutes. Therefore, both the family~$(\theta_{P_m})$ and the family~$(\id_{P_m})$ lift the morphism $\theta_{\1} = \id_{\1}$ to the projective resolution. By the comparison theorem, the two lifts are chain-homotopic. This chain homotopy induces a cochain homotopy on the cochain complex 
$\Hom_\CC(P_m\ot X_n \ot \cdots \ot X_1, L^{\ot g})$, which yields that the two
maps induce the same map in cohomology, namely the identity. 
\end{proof}

We will need another lemma of a similar nature:
\begin{Lemma} \label{LemDehnTw2} 
For the two Dehn twists~$\fd_n$ and~$\fd_{n,n+1}$ in~$\Gamma_{g,n+1}(n+1)$, we have $[\TFT^m(\fd_{n})] = [\TFT^m(\fd_{n,n+1})]$.
\end{Lemma}
\begin{proof}
From Lemma~\ref{NatHomotop}, we know that the chain maps 
$\theta_{P_m\ot X_n}$ and~$\id_{P_m}\ot \theta_{X_{n}}$ are chain-homotopic. This implies that the chain maps 
\[ \theta_{P_m\ot X_n}\ot \id_{X_{n-1} \ot \dots \ot X_{1}} \qquad \text{and} \qquad
\id_{P_m}\ot \theta_{X_n} \ot \id_{X_{n-1} \ot \cdots \ot X_{1}} \] 
are chain-homotopic. Because~$\fd_{n,n+1}$ acts by precomposition with the first one and~$\fd_{n}$ acts by precomposition with the second one, this implies the assertion.
\end{proof}

We now consider the curve~$\alpha_i$ introduced in Paragraph~\ref{Surf}, which begins and ends in the base point~$x$ for the surface~$\Gamma_{g,n}$. If we push it off the base point to the left and to the right as described in Paragraph~\ref{FundGroup}, we can cut out a small disk centered at~$x$ without intersecting these two curves, which we consider as the $(n+1)$st boundary component. In this way, we obtain two curves~$\alpha'_i$ and~$\alpha''_i$ on the surface~$\Gamma_{g,n+1}$:

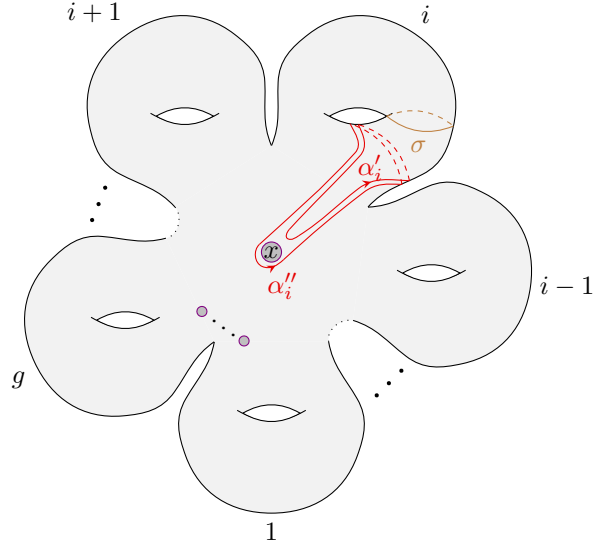
\begin{center}
\fontsize{10}{10}\selectfont
\begin{tikzpicture}[scale=0.38]
\pic[scale=0.38] (lower) at (0,-pi) {handle}; 
\pic[rotate=82.5,scale=0.38] (lr) at (-7.5:pi) {handle};
\pic[rotate=147.5,scale=0.38] (tr) at (57.5:pi) {handle};
\pic[rotate=212.5,scale=0.38] (tl) at (122.5:pi) {handle};
\pic[rotate=295,scale=0.38] (ll) at (205:pi) {handle};
\fill[gray!10]  (lower-right) to[out=100,in=170] (lr-left)-- 
(lr-right) to (tr-left)
-- (tr-right) to (tl-left)
-- (tl-right) to[out=310,in=10] (ll-left) 
--(ll-right) to (lower-left) -- cycle;
\draw[dotted] (lower-right) to[out=100,in=170] (lr-left);
\draw[dotted] (tl-right) to[out=310,in=10] (ll-left);

\node[rotate=68,scale=1.5] at (-6,1.8) {\dots};
\node[rotate=45,scale=1.5] at (4.2,-4.5) {\dots};

\draw[violet,fill=lightgray] (0,0) circle (10pt);
\node[] (v1) at (0,0) {$x$}; 
\node[below] at (lower-num) {1}; 
\node[right] at (lr-num) {$i-1$}; 
\node[above right] at (tr-num) {$i$}; 
\node[above left] at (tl-num) {$i+1$}; 
\node[below left] at (ll-num) {$g$}; 

\draw[violet,rotate=-35,fill=lightgray] (-0.8,-3.1) circle (5pt) ;
\node[rotate=-35] at (-1.6,-2.7) {\dots};
\draw[violet,rotate=-35,fill=lightgray] (1,-3.1) circle (5pt) ;

\draw[rotate=147.5,black!10!red] 
(-0.1,-1) to[out=260,in=80] node[pos=0.85, right]{$\alpha_i'$} (-0.8,-4.5) to[out=260,in=150]  (tr-B);
\draw[rotate=147.5,black!10!red, dash pattern= on 2pt off 2pt] (-2.73,-4.7) to[bend right] (tr-B);
\draw[rotate=147.5,black!10!red,decoration={ markings, mark=at position 0.7 with {\arrow{stealth}}}, postaction={decorate}]
(-0.5,-1) to[out=245,in=70] (-1.37,-3.5)  to[out=250,in=30] (-2.73,-4.7);
\draw[rotate=147.5,black!10!red] (-0.1,-1) to[out=80,in=10] (-0.24,-0.7) to[out=190,in=65] (-0.5,-1) ;

\draw[rotate=147.5,black!10!red] 
(0.5,0) to[out=255,in=80] (-0.6,-4.5) to[out=260,in=150] (0.1,-5.27);
\draw[rotate=147.5,black!10!red, dash pattern= on 2pt off 2pt] (-2.58,-4.45) to[bend right] (0.1,-5.27);
\draw[rotate=147.5,black!10!red]
(-0.5,0) to[out=255,in=70] (-1.6,-3.5)  to[out=250,in=30] (-2.58,-4.45);
\draw[rotate=147.5,black!10!red,decoration={ markings, mark=at position 0.8 with {\arrow{stealth}}}, postaction={decorate}] (0.5,0) to[out=75,in=10] node[pos=0.8, below right]{$\alpha_i''$} (0,0.6)  to[out=190,in=75] (-0.5,0) ;

\draw[brown] (tr-D) to[bend right] node[pos=0.5, below]{$\sigma$} (tr-k);
\draw[brown,dash pattern= on 2pt off 2pt] (tr-D) to[bend left] (tr-k);

\end{tikzpicture}
\end{center}
We now label the new boundary component with the elements of our projective resolution; i.e., for a given~$m$, we set $X_{n+1} \deq P_m$. We want to show that the two Dehn twists~$\ft'_i \deq \fd_{\alpha'_i}$ and~$\ft''_i \deq \fd_{\alpha''_i}$ induce the same map in cohomology. As we will see below, it is sufficient to consider the case~$i=1$:
\begin{Lemma} \label{LemFirstGen} 
For the two Dehn twists~$\ft'_1 \deq \fd_{\alpha'_1}$ and~$\ft''_1 \deq \fd_{\alpha''_1}$ in~$\Gamma_{g,n+1}(n+1)$, we have $[\TFT^m(\ft'_1)] = [\TFT^m(\ft''_1)]$.
\end{Lemma}
\begin{proof}
\begin{parlist}
\item
If we cut the surface along the curve that is denoted by~$\sigma$ in the picture above, we obtain a surface that is diffeomorphic to~$\Sigma_{g-1,n+3}$:
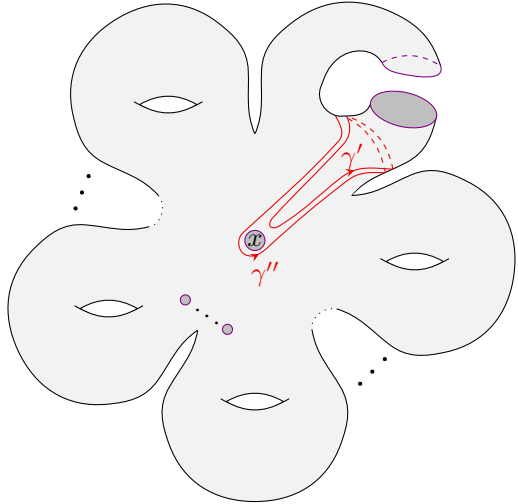
\begin{center}
\fontsize{10}{10}\selectfont
\begin{tikzpicture}[scale=0.38]
\pic[scale=0.38] (lower) at (0,-pi) {handle}; 
\pic[rotate=82.5,scale=0.38] (lr) at (-7.5:pi) {handle};
\pic[rotate=147.5,scale=0.38] (tr) at (57.5:pi) {rightpart};
\pic[rotate=212.5,scale=0.38] (tl) at (122.5:pi) {handle};
\pic[rotate=295,scale=0.38] (ll) at (205:pi) {handle};
\fill[gray!10]  (lower-right) to[out=100,in=170] (lr-left)-- 
(lr-right) to (tr-left) to[out=50, in=255] (tr-k) to (tr-D)
to[out=90, in=40] (tr-right) 
to (tl-left)
-- (tl-right) to[out=310,in=10] (ll-left) 
--(ll-right) to (lower-left) -- cycle;
\draw[dotted] (lower-right) to[out=100,in=170] (lr-left);
\draw[dotted] (tl-right) to[out=310,in=10] (ll-left);

\node[rotate=68,scale=1.5] at (-6,1.8) {\dots};
\node[rotate=45,scale=1.5] at (4.2,-4.5) {\dots};

\draw[violet,fill=lightgray] (0,0) circle (10pt);
\node[] (v1) at (0,0) {$x$}; 

\draw[violet,rotate=-35,fill=lightgray] (-0.8,-3.1) circle (5pt) ;
\node[rotate=-35] at (-1.6,-2.7) {\dots};
\draw[violet,rotate=-35,fill=lightgray] (1,-3.1) circle (5pt) ;

\draw[rotate=147.5,black!10!red] 
(-0.1,-1) to[out=260,in=80] node[pos=0.85, right]{$\gamma'$} (-0.8,-4.5) to[out=260,in=150]  ($(tr-B)-(0.2,0)$);
\draw[rotate=147.5,black!10!red, dash pattern= on 2pt off 2pt] (-2.73,-4.7) to[bend right] coordinate[pos=1] (X) ($(tr-B)-(0.2,0)$);
\draw[rotate=147.5,black!10!red,decoration={ markings, mark=at position 0.7 with {\arrow{stealth}}}, postaction={decorate}]
(-0.5,-1) to[out=245,in=70] (-1.37,-3.5)  to[out=250,in=30] (-2.73,-4.7);
\draw[rotate=147.5,black!10!red] (-0.1,-1) to[out=80,in=10] (-0.24,-0.7) to[out=190,in=65] (-0.5,-1) ;

\draw[rotate=147.5,black!10!red] 
(0.5,0) to[out=255,in=80] (-0.6,-4.5) to[out=260,in=150] (0.1,-5.27);
\draw[rotate=147.5,black!10!red, dash pattern= on 2pt off 2pt] (-2.58,-4.45) to[bend right] (0.1,-5.27);
\draw[rotate=147.5,black!10!red]
(-0.5,0) to[out=255,in=70] (-1.6,-3.5)  to[out=250,in=30] (-2.58,-4.45);
\draw[rotate=147.5,black!10!red,decoration={ markings, mark=at position 0.8 with {\arrow{stealth}}}, postaction={decorate}] (0.5,0) to[out=75,in=10] node[pos=0.8, below right]{$\gamma''$} (0,0.6)  to[out=190,in=75] (-0.5,0) ;

\draw[fill=white](tr-D)to[out=250,in=0] (X) to[out=180,in=-10] ++(-0.3,0)
to[out=170,in=250] ++(-0.6,1) to[out=70,in=220] ++(0.6,1) to[out=40, in=120] (tr-a);

\coordinate[] (M) at ($(tr-D)!0.5!(tr-k)$);
\draw[violet, fill=lightgray, rotate=-13]  let \p1 = ($(tr-D) - (tr-k)$) in (M) circle [x radius= {0.5*veclen(\x1,\y1)}, y radius= 0.6cm];

\draw[violet,dash pattern= on 2pt off 2pt] (tr-a) to[bend left] (tr-b);
\end{tikzpicture}
\end{center}

Here, we consider the lower boundary component arising from the cut in the picture above as the~$(n+2)$nd one and the upper boundary component as the~$(n+3)$rd one.

The surface~$\Sigma_{g,n+1}$ can be reconstructed from the surface~$\Sigma_{g-1,n+3}$ by gluing in a handle, as described in Paragraph~\ref{ModFunct}. Upon gluing, the Dehn twists along the curves denoted by~$\gamma'$ and~$\gamma''$ in the second picture become the Dehn twists along the curves denoted by~$\alpha'_1$ and~$\alpha''_1$ in the first picture, respectively. As we saw in Paragraph~\ref{ModFunct}, the diagram
\begin{center}
\begin{tikzcd}
\TFT(\Sigma_{g-1,n+3}^{X_1, \dots, X_n, P_m, X^{**}, X^*}) \arrow{r} \arrow[d, "\TFT(\fd_{\gamma'})"']  & \TFT(\Sigma_{g,n+1}^{X_1, \dots, X_n, P_m}) \arrow[d, "\TFT(\ft'_1)"]\\
\TFT(\Sigma_{g-1,n+3}^{X_1, \dots, X_n, P_m,  X^{**}, X^*})
\arrow[r] & \TFT(\Sigma_{g,n+1}^{X_1, \dots, X_n, P_m})
\end{tikzcd}
\end{center}
commutes for our choices of the representatives~$\TFT(\fd_{\gamma'})$ and~$\TFT(\ft'_1)$ of the projective classes, because upon appropriate labeling~$\fd_{\gamma'}$ is~$\fd_{n+2}$ and~$\ft'_1$ is~$\ft_1$ in the notation used there. From the discussion in Paragraph~\ref{ModFunct}, we know that a similar diagram commutes for~$\fd_{\gamma''}$ and~$\ft''_1$. Our goal is to show that~$\TFT(\ft'_1)$ and~$\TFT(\ft''_1)$ induce the same map in cohomology.

\item
To see this, we apply Lemma~\ref{NatHomotop} to the reverse category, in which tensor products are taken in the opposite order (cf.~\cite[Examp.~2.5, p.~39]{JS}). In this way, we obtain for each object~$X \in \CC$ a chain homotopy
$h_m(X)\colon X\ot P_m \to X\ot P_{m+1} $ between the chain maps~$(\theta_{X\ot P_m })$ and~$(\theta_{X}\ot \id_{P_m})$ that is natural in~$X$. For two objects~$X$ and~$Y$ of~$\CC$, the Dehn twists~$\fd_{\gamma'}$ and~$\fd_{\gamma''}$ act on the space 
$\Hom_\CC(Y^* \ot X^{**} \ot P_m \ot X_n \ot \dots \ot X_1, L^{\ot (g-1)})$ by precomposition with~$\id_{Y^*} \ot \, \theta_{X^{**}} \ot \id_{P_m\ot Z} $ 
and~$ \id_{Y^*} \ot \, \theta_{X^{**} \ot P_m} \ot\id_{Z}$, respectively, where, for brevity, we have used the notation
$Z \deq X_n \ot \dots \ot X_1$. By naturality, the adjunction isomorphism
\[\Hom_\CC(Y^* \ot X^{**} \ot P_m \ot Z, L^{\ot (g-1)}) \to 
\Hom_\CC(P_m \ot Z, X^* \ot Y \ot L^{\ot (g-1)})\]
is an isomorphism of cochain complexes, so that we obtain a cochain map~$(f_{X,Y}'^m)$ of the cochain complex on the right-hand side whose defining property is that the diagram
\begin{center}
\begin{tikzcd}
\Hom_\CC(Y^* \ot X^{**} \ot P_m \ot Z, L^{\ot (g-1)}) \arrow{r}
\arrow[d, "= \, \circ (\id_{Y^*} \ot \, \theta_{X^{**}} \ot\id_{P_m\ot Z})","\TFT(\fd_{\gamma'})"']  & \Hom_\CC(P_m \ot Z, X^* \ot Y \ot L^{\ot (g-1)}) \arrow[d, "f_{X,Y}'^m"]\\
\Hom_\CC(Y^* \ot X^{**} \ot P_m\ot Z, L^{\ot (g-1)}) \arrow{r} & 
\Hom_\CC(P_m\ot Z, X^* \ot Y \ot L^{\ot (g-1)})
\end{tikzcd}
\end{center}
commutes. For~$\gamma''$, there is a second cochain map~$(f_{X,Y}''^m)$ that makes a very similar diagram commutative. For~$\gamma'$, the naturality of the adjunction isomorphism implies that~$f_{X,Y}'^m$ is given by postcomposition with~$\theta_{X^*} \ot \id_Y \ot \id_{L^{\ot (g-1)}}$. The chain homotopy~$(\id_{Y^*} \ot \, h_m(X^{**}) \ot \id_Z)$ induces a cochain homotopy
\[h'_m(X,Y) \colon \Hom_\CC(P_m \ot Z, X^* \ot Y \ot L^{\ot (g-1)}) \to 
\Hom_\CC(P_{m-1}\ot Z, X^* \ot Y \ot L^{\ot (g-1)})\]
that is natural in~$X$ and~$Y$ and makes the diagram
\begin{center}
\begin{tikzcd}
\Hom_\CC(Y^* \ot X^{**} \ot P_m \ot Z, L^{\ot (g-1)}) \arrow{r}
\arrow[d, "\circ (\id_{Y^*} \ot \, h_{m-1}(X^{**}) \ot \id_Z)"]  & 
\Hom_\CC(P_m\ot Z,  X^* \ot Y \ot L^{\ot (g-1)}) \arrow[d, "{h'_m(X,Y)}"]\\
\Hom_\CC(Y^* \ot X^{**} \ot P_{m-1} \ot Z, L^{\ot (g-1)}) \arrow{r} & 
\Hom_\CC(P_{m-1}\ot Z, X^* \ot Y \ot L^{\ot (g-1)})
\end{tikzcd}
\end{center} 
commutative.

\item
Now the functor~$\Hom_\CC(P_m\ot Z, -\ot L^{\ot (g-1)})$ is an exact functor from the category~$\CC$ to the category~$\V$ of finite-dimensional vector spaces. By Proposition~\ref{RightAdj}, it has a right adjoint and therefore preserves coends, as already mentioned in Paragraph~\ref{Coend}. Therefore, the family of morphisms
\[(\iota_X\ot \id_{L^{\ot (g-1)}}) \circ \colon 
\Hom_\CC(P_m \ot Z, X^* \ot X \ot L^{\ot (g-1)} ) 
\to \Hom_\CC(P_m\ot Z, L^{\ot g}) \]
is a coend for the bifunctor
\[\CC^{\op} \times \CC \to \V,~(X,Y) \mapsto 
\Hom_\CC(P_m \ot Z, X^* \ot Y \ot L^{\ot (g-1)}). \]
So we can apply Lemma~\ref{LemCoendNat} to this bifunctor and the corresponding bifunctor with~$m-1$ instead of~$m$ to obtain a $\K$-linear map
\[h'_m\colon \Hom_\CC(P_m\ot Z, L^{\ot g}) \to \Hom_\CC(P_{m-1}\ot Z, L^{\ot g})\]
that makes the diagram
\begin{center}
\begin{tikzcd}
\Hom_\CC(P_m\ot Z, X^*\ot X\ot L^{\ot (g-1)}) \arrow{r}
\arrow[d, "{h'_m(X,X)}"']  & 
\Hom_\CC(P_m\ot Z, L^{\ot g}) \arrow[d, "{h'_m}"]\\
\Hom_\CC(P_{m-1}\ot Z, X^*\ot X\ot L^{\ot (g-1)}) \arrow{r} & 
\Hom_\CC(P_{m-1}\ot Z, L^{\ot g})
\end{tikzcd}
\end{center}
commutative, i.e., satisfies
\[h'_m((\iota_X\ot \id_{L^{\ot (g-1)}}) \circ k) = 
(\iota_X\ot \id_{L^{\ot (g-1)}}) \circ h'_m(X,X)(k)\]
for all morphisms $k\colon P_m\ot Z \to X^*\ot X\ot L^{\ot (g-1)}$.

\item
By the definition of~$h_m(X^{**})$, we have
\[\theta_{X^{**} \ot P_m} - \theta_{X^{**}} \ot \id_{P_m} 
= (\id_{X^{**}} \ot d_{m+1}) \circ h_m(X^{**}) + h_{m-1}(X^{**}) \circ (\id_{X^{**}} \ot d_m).\]
For a morphism $k' \colon Y^* \ot X^{**} \ot P_m \ot Z \to L^{\ot (g-1)}$, this implies that
\begin{align*}
&k' \circ [(\id_{Y^*} \ot \theta_{X^{**} \ot P_m} \ot \id_Z) 
- (\id_{Y^*} \ot \theta_{X^{**}} \ot \id_{P_m \ot Z})]  \\
& \qquad = k' \circ  [(\id_{Y^* \ot X^{**}} \ot d_{m+1} \ot \id_Z) 
\circ (\id_{Y^*} \ot h_m(X^{**}) \ot \id_Z) \\
&\qquad \quad +  (\id_{Y^*} \ot h_{m-1}(X^{**}) \ot \id_Z) 
\circ (\id_{Y^* \ot X^{**}} \ot d_m \ot \id_{Z})].
\end{align*}

In view of the definition of~$h'_m(X,Y)$ and the naturality of the adjunction isomorphism, this yields
\begin{align*}
&f''^m_{X,Y}(k) - f'^m_{X,Y}(k) = 
h'_{m+1}(X,Y)(k \circ (d_{m+1}\ot \id_{Z})) + h'_{m}(X,Y)(k) \circ (d_{m}\ot \id_{Z})
\end{align*}
for all morphisms $k \colon P_m \ot Z \to X^* \ot Y \ot L^{\ot (g-1)}$. If we set $X=Y$ and compose with~$\iota_X \ot \id_{L^{\ot (g-1)}}$, this equation becomes
\begin{align*}
&\TFT(\ft''_1) \circ (\iota_X \ot \id_{L^{\ot (g-1)}}) \circ k 
- \TFT(\ft'_1) \circ (\iota_X \ot \id_{L^{\ot (g-1)}})\circ  k\\ 
&\;= h'_{m+1}((\iota_X \ot \id_{L^{\ot (g-1)}}) \circ k \circ (d_{m+1}\ot \id_{Z}))
+ h'_{m}((\iota_X \ot \id_{L^{\ot (g-1)}}) \circ k) \circ (d_m\ot \id_{Z}).
\end{align*}
Both sides of this equation define a dinatural transformation from the bifunctor
\[\CC^{\op} \times \CC \to \V,~(X,Y) \mapsto 
\Hom_\CC(P_m \ot Z, X^* \ot Y \ot L^{\ot (g-1)}) \]
already considered above to its coend $\Hom_\CC(P_m \ot Z, L^{\ot g})$. Our dinatural transformation factors over this coend, and the corresponding homomorphism, which is unique, can be read off directly from both the left and the right-hand side of the equation above. We get
\[\TFT(\ft''_1) \circ k'' - \TFT(\ft'_1) \circ k'' = 
h'_{m+1}(k'' \circ (d_{m+1}\ot \id_{Z})) + h'_m(k'') \circ (d_m\ot \id_{Z}) \]
for all~$k'' \in \Hom_\CC(P_m\ot Z, L^{\ot g})$. Therefore, the family $(h'_m)$ constitutes a cochain homotopy between the cochain maps induced by~$\TFT(\ft'_1)$ and~$\TFT(\ft''_1)$. When saying that, it should be noted that these maps are only determined up to a scalar; we have taken here for~$\TFT(\ft'_1)$ the representative that corresponds to the dinatural transformation 
$k \mapsto (\iota_X\ot \id_{L^{\ot (g-1)}}) \circ f'^m_{X,X}(k)$ via the universal property of the coend of our bifunctor, and have made a similar choice for the representative of~$\TFT(\ft''_1)$ using~$f''^m_{X,X}$.
\qedhere
\end{parlist}
\end{proof}

By putting these auxiliary results together, we can now associate with a mapping class in~$\Gamma_{g,n}$ a projective class of morphisms not only between the original block spaces, but rather between the derived block spaces. As already stated above, we do this by choosing a preimage under the epimorphism~$D_{n+1}$ defined at the end of Paragraph~\ref{Cap}. The key fact that we need to prove is therefore the following:
\begin{Theorem} \label{ThmDesc}
Suppose that $[\psi], [\psi'] \in \Gamma_{g,n+1}(n+1)$ satisfy 
$D_{n+1}([\psi]) = D_{n+1}([\psi'])$. Then we have $[\TFT^m(\psi)] = [\TFT^m(\psi')]$.
\end{Theorem}
\begin{proof}
\begin{parlist}
\item 
Clearly, $[\psi]$ and~$[\psi']$ differ by an element in the kernel of~$D_{n+1}$, which is in particular an element in the pure mapping class group~$\Pu\Gamma_{g,n+1}$. It therefore suffices to show that $[\TFT^m(\psi)] = [\id]$ for each mapping class~$[\psi]$ in the kernel of~$D_{n+1}$. Recall that~$D_{n+1}$ is the composition 
\[\Gamma_{g,n+1}(n+1) \overset{C_{n+1}}{\longrightarrow} \Gamma_{g,n}(y) \overset{F_y}{\longrightarrow} \Gamma_{g,n} \]
of the capping homomorphism, which arises from gluing a punctured disk in place of the missing disk, and the forgetful map~$F_y$ that, depending on the perspective, either fills this puncture or `forgets' that the point was marked. According to Proposition~\ref{CapSeq}, the kernel of~$C_{n+1}$ is generated by the Dehn twist~$\fd_{n+1}$. But by Lemma~\ref{LemDehnTw1} above, we have~$[\TFT^m(\fd_{n+1})] = [\id]$. This implies that the assignment
$[\psi] \mapsto [\TFT^m(\psi)]$ is well-defined for $[\psi] \in \Gamma_{g,n}(y)$, not only for
$[\psi] \in \Gamma_{g,n+1}(n+1)$. 

\item
To treat the second morphism~$F_y$ in the above composition, we use the Birman sequence from Paragraph~\ref{Birm}. In order to do this, we proceed as in Paragraph~\ref{Torus} and first replace the additional puncture~$y$ by the base point~$x$ of the fundamental group that comes from the polygon model of the surface as the identification of all the vertices of the polygon, except for those that correspond to the marked points on the boundary components. The map that forgets the base point~$x$ instead of the puncture will be denoted by~$F_x$ instead of~$F_y$. The Birman sequence then takes the form
\[\pi_1(\Sigma_{g,n},x) \overset{P_x}{\longrightarrow} \Gamma_{g,n}(x) \overset{F_x}{\longrightarrow} \Gamma_{g,n} \longrightarrow 1,\vspace{1mm}\]
where, as before,~$P_x$ denotes the pushing map. To complete the proof, we therefore have to show that $[\TFT^m(P_x([\gamma]))] = [\id]$ for all homotopy classes~$[\gamma] \in \pi_1(\Sigma_{g,n},x)$. Because~$P_x$ is a group antihomomorphism, it suffices to prove this for all~$[\gamma]$ in a generating set of the fundamental group. 

\vspace{2mm}

\item
As discussed in Paragraph~\ref{FundGroupGen}, the fundamental group~$\pi_1(\Sigma_{g,n},x)$ is generated by the homotopy classes of the simple closed curves $\alpha_1,\beta_1,\ldots,\alpha_g,\beta_g$, all of which are nonseparating, together with the curves~$\delta_1,\ldots,\delta_n$, all of which are separating. From a formula stated at the end of Paragraph~\ref{Birm}, we know that
$P_x([\alpha_1]) = [\fd_{\alpha'_1} \fd_{\alpha''_1}^{-1}]$. Therefore Lemma~\ref{LemFirstGen} yields that $[\TFT^m(P_x([\alpha_1]))] = [\id]$. 

\vspace{2mm}

\item
For any of the other generators that correspond to nonseparating curves, say~$\beta_i$, we use the change of coordinates principle (cf.~\cite[Par.~1.3.1, p.~37]{FM}) to obtain a diffeomorphism \mbox{$\varphi\colon \Sigma_{g,n} \to \Sigma_{g,n}$} satisfying~$\varphi(\alpha_1) = \beta_i$. The discussion in~\cite[loc.~cit.]{FM} shows that we can assume that~$\varphi$ is orientation-preserving and satisfies~$\varphi(x) = x$. From Paragraph~\ref{Birm}, we know that $[P_x([\beta_i])] = [\varphi] [P_x([\alpha_1])] [\varphi^{-1}]$. Now the compatibility with composition described in Paragraph~\ref{MapClRep} implies that
\begin{align*}
[\TFT^m(P_x([\beta_i]))] = [\TFT^m(\varphi)] [\TFT^m(P_x([\alpha_1]))] [\TFT^m(\varphi^{-1})]
= [\TFT^m(\varphi)] [\TFT^m(\varphi^{-1})] = [\id].
\end{align*}

\vspace{1mm}

\item
For the separating curves~$\delta_1,\ldots,\delta_n$, the formula already mentioned above yields that~$P_x([\delta_j]) = [\fd_{\delta'_j} \fd_{\delta''_j}^{-1}]$.
We have $C_{n+1}([\fd_n]) = [\fd_{\delta'_n}]$ and $C_{n+1}([\fd_{n,n+1}]) = [\fd_{\delta''_n}]$. Therefore, Lemma~\ref{LemDehnTw2} shows that
$[\TFT^m(P_x([\delta_n]))]  = [\id]$. If $j\neq n$, we proceed as above and use the general change of coordinate principle (cf.~\cite[loc.~cit.]{FM}) to obtain a diffeomorphism \mbox{$\varphi\colon \Sigma_{g,n} \to \Sigma_{g,n}$} with~$\varphi(\delta_n) = \delta_j$ that is orientation-preserving and satisfies~$\varphi(x) = x$. As above, we get
$[P_x([\delta_j])] = [\varphi] [P_x([\delta_n])] [\varphi^{-1}]$, and a very similar computation as there then yields $[\TFT^m(P_x([\delta_j]))] = [\id]$.
\qedhere
\end{parlist}
\end{proof}

By construction, a mapping class that permutes the boundary components does not in general carry a derived block space into itself. However, the pure mapping class group preserves these spaces:
\begin{Cor} \label{CorDesc}
There is a projective action of $\Pu\Gamma_{g,n}$ on
\[\TFT^m(\Sigma_{g,n}^{X_1,\ldots, X_n}) = \Ext^m(X_n \ot \dots \ot X_1, L^{\ot g})\]
that agrees with the original action on 
$\TFT(\Sigma_{g,n}^{X_1,\ldots, X_n}) = \Hom_\CC(X_n \ot \dots \ot X_1, L^{\ot g})$ if~$m=0$.
\end{Cor}

As in Paragraph~\ref{MapClRep}, the entire mapping class group acts projectively on the direct sum
\[\bigoplus_{\tau \in S_n} \TFT^m(\Sigma_{g,n}^{X_{\tau(1)},\ldots, X_{\tau(n)}})\]
of derived block spaces. If all boundary components have the same label~$X$, the mapping class group indeed acts projectively on a single derived block space:
\begin{Cor} \label{CorEqual}
For~$X \in \CC$, there is a projective action of $\Gamma_{g,n}$ 
on~$\Ext^m(X^{\ot n}, L^{\ot g})$ that agrees with the original action on 
$\Hom_\CC(X^{\ot n}, L^{\ot g})$ if~$m=0$.
\end{Cor}

\subsection{The case of the sphere} \label{CaseSphere}
The construction of the mapping class group representations in the preceding paragraph was not based on generators and relations for the mapping class group, as such a presentation of the mapping class group is notoriously difficult. In the case~$g=0$, however, we gave such a presentation in Paragraph~\ref{Sphere}. It is instructive to see why the defining relations are satisfied in this case. In order to verify these relations, we will need a couple of lemmas. As stated in Paragraph~\ref{Fin}, we are assuming that our category is strict.
\begin{Lemma} \label{LemCaseSphere}
Suppose that $X_1,\ldots,X_n$ are objects of~$\CC$. Then we have
\begin{enumerate}
\item 
$\displaystyle
c_{X_1 \ot \dots \ot X_{n-1}, X_n} = (c_{X_1,X_n} \ot \id_{X_2 \ot \dots \ot X_{n-1}}) \circ 
(\id_{X_1} \ot \, c_{X_2,X_n} \ot \id_{X_3 \ot \dots \ot X_{n-1}})  $ \\
\phantom{x} $\displaystyle  
\mspace{120mu} \circ \cdots \circ (\id_{X_1 \ot \dots \ot X_{n-2}} \ot \, c_{X_{n-1},X_n})$

\item 
$\displaystyle
c_{X_1, X_2 \ot \dots \ot X_{n}} = (\id_{X_2 \ot \dots \ot X_{n-1}} \ot \, c_{X_1,X_{n}}) \circ \cdots 
\circ (\id_{X_2} \ot \, c_{X_1,X_3} \ot \id_{X_4 \ot \dots \ot X_{n}}) $ \\
\phantom{x} $\displaystyle
\mspace{105mu}
\circ (c_{X_1,X_2} \ot \id_{X_3 \ot \dots \ot X_{n}})$

\item 
$\displaystyle
c_{X_n, X_1 \ot \dots \ot X_{n-1}} \circ \cdots 
\circ c_{X_2, X_3 \ot \dots \ot X_n \ot X_{1}} \circ c_{X_1, X_2 \ot \dots \ot X_{n}} = 
\theta_{X_1 \ot \dots \ot X_{n}} \circ (\theta_{X_1}^{-1} \, \ot \dots \ot \, \theta_{X_n}^{-1})$

\item 
$\displaystyle
c_{X_2 \ot \dots \ot X_{n}, X_1} \circ
c_{X_3 \ot \dots \ot X_n \ot X_{1}, X_2} \circ \cdots \circ
c_{X_1 \ot \dots \ot X_{n-1}, X_n}     = 
\theta_{X_1 \ot \dots \ot X_{n}} \circ (\theta_{X_1}^{-1} \, \ot \dots \ot \, \theta_{X_n}^{-1})$
\end{enumerate}
\end{Lemma}
\begin{proof}
\begin{parlist}
\item
The first assertion is clearly correct for $n=2$ and then follows inductively from the equation
\begin{align*}
&c_{X_1 \ot \dots \ot X_{n-1}, X_n} = 
(c_{X_1 \ot \dots \ot X_{n-2}, X_{n}} \ot \id_{X_{n-1}}) \circ
(\id_{X_1 \ot \dots \ot X_{n-2}} \ot \, c_{X_{n-1},X_n}).
\end{align*}

\item
Associated with every braiding is another braiding in which $c_{X,Y}$ is replaced by~$c^{-1}_{Y,X}$. If we apply the first assertion to this braiding instead, take inverses, and permute~$X_1,\ldots,X_n$ cyclically, we obtain the second assertion. 

\item
To prove the third assertion, we need the auxiliary statement that
\begin{align*}
&c_{X_k, X_{k+1} \ot \dots \ot X_n \ot X_1 \ot \dots \ot X_{k-1}} \circ \cdots 
\circ c_{X_2, X_3 \ot \dots \ot X_n \ot X_{1}} \circ c_{X_1, X_2 \ot \dots \ot X_{n}} = \\ 
&c_{X_1 \ot \dots \ot X_{k}, X_{k+1} \ot \dots \ot X_{n}} \circ 
(\theta_{X_1 \ot \dots \ot X_{k}} \ot \id_{X_{k+1} \ot \dots \ot X_{n}}) \circ 
(\theta_{X_1}^{-1} \ot \dots \ot \theta_{X_k}^{-1} \ot \id_{X_{k+1} \ot \dots \ot X_{n}})
\end{align*}
for $k=1,\ldots,n-1$. For $k=1$, this is obvious. For $k \le n-2$, we have inductively
\begin{align*}
&c_{X_{k+1}, X_{k+2} \ot \dots \ot X_{n} \ot X_1 \ot \cdots \ot X_{k}} \circ c_{X_1 \ot \dots \ot X_{k}, X_{k+1} \ot \dots \ot X_{n}} \\ 
&\qquad\circ (\theta_{X_1 \ot \dots \ot X_{k}} \ot \id_{X_{k+1} \ot \dots \ot X_{n}}) \circ 
(\theta_{X_1}^{-1} \ot \dots \ot \theta_{X_k}^{-1} \ot \id_{X_{k+1} \ot \dots \ot X_{n}})  \\
&\quad= (\id_{X_{k+2} \ot \dots \ot X_{n}} \ot \, c_{X_{k+1}, X_1 \ot \cdots \ot X_{k}}) \circ 
(c_{X_{k+1}, X_{k+2} \ot \dots \ot X_{n}} \ot \id_{X_1 \ot \cdots \ot X_{k}}) \\
&\qquad\circ (\id_{X_{k+1}} \ot \, c_{X_1 \ot \cdots \ot X_{k}, X_{k+2} \ot \dots \ot X_{n}}) \circ 
(c_{X_1 \ot \cdots \ot X_{k}, X_{k+1}} \ot \id_{X_{k+2} \ot \dots \ot X_{n}}) \\
&\qquad\circ (\theta_{X_1 \ot \dots \ot X_{k}}  \ot \theta_{X_{k+1}} \ot \id_{X_{k+2} \ot \dots \ot X_{n}}) \circ 
(\theta_{X_1}^{-1} \ot \dots \ot \theta_{X_{k+1}}^{-1} 
\ot \id_{X_{k+2} \ot \dots \ot X_{n}}).
\end{align*}
By using the Yang-Baxter equation (cf.~\cite[Thm.~XIII.1.3, p.~317]{Ka}) on the first three terms, this expression becomes
\begin{align*}
&(c_{X_1 \ot \cdots \ot X_{k}, X_{k+2} \ot \dots \ot X_{n}} \ot \id_{X_{k+1}}) \circ 
(\id_{X_1 \ot \cdots \ot X_{k}} \ot \, c_{X_{k+1}, X_{k+2} \ot \dots \ot X_{n}}) \\
&\quad \circ (c_{X_{k+1}, X_1 \ot \cdots \ot X_{k}} \ot \id_{X_{k+2} \ot \dots \ot X_{n}}) \circ 
(c_{X_1 \ot \cdots \ot X_{k}, X_{k+1}} \ot \id_{X_{k+2} \ot \dots \ot X_{n}}) \\
&\quad\circ (\theta_{X_1 \ot \dots \ot X_{k}}  \ot \theta_{X_{k+1}} \ot \id_{X_{k+2} \ot \dots \ot X_{n}}) \circ 
(\theta_{X_1}^{-1} \ot \dots \ot \theta_{X_{k+1}}^{-1} 
\ot \id_{X_{k+2} \ot \dots \ot X_{n}}),
\end{align*}
which by the defining property of a ribbon structure is equal to
\begin{align*}
&(c_{X_1 \ot \cdots \ot X_{k}, X_{k+2} \ot \dots \ot X_{n}} \ot \id_{X_{k+1}}) \circ 
(\id_{X_1 \ot \cdots \ot X_{k}} \ot \, c_{X_{k+1}, X_{k+2} \ot \dots \ot X_{n}}) \\
&\qquad\circ (\theta_{X_1 \ot \dots \ot X_{k} \ot X_{k+1}} \ot \id_{X_{k+2} \ot \dots \ot X_{n}}) \circ 
(\theta_{X_1}^{-1} \ot \dots \ot \theta_{X_{k+1}}^{-1} 
\ot \id_{X_{k+2} \ot \dots \ot X_{n}}) \\
&\quad = c_{X_1 \ot \dots \ot X_{k+1}, X_{k+2} \ot \dots \ot X_{n}} \circ 
(\theta_{X_1 \ot \dots \ot X_{k+1}} \ot \id_{X_{k+2} \ot \dots \ot X_{n}}) \\
&\qquad\circ 
(\theta_{X_1}^{-1} \ot \dots \ot \theta_{X_{k+1}}^{-1} 
\ot \id_{X_{k+2} \ot \dots \ot X_{n}}),
\end{align*}
completing our inductive step.

\item
To derive the third assertion, we now use the case $k=n-1$ of the auxiliary statement above and apply
$c_{X_n, X_1 \ot \dots \ot X_{n-1}}$ to obtain
\begin{align*}
&c_{X_n, X_1 \ot \dots \ot X_{n-1}} \circ c_{X_{n-1}, X_n \ot X_1 \ot \dots \ot X_{n-2}} \circ\cdots 
\circ c_{X_2, X_3 \ot \dots \ot X_n \ot X_{1}} \circ c_{X_1, X_2 \ot \dots \ot X_{n}}\\
&\quad = c_{X_n, X_1 \ot \dots \ot X_{n-1}} \circ c_{X_1 \ot \dots \ot X_{n-1}, X_{n}} \circ 
(\theta_{X_1 \ot \dots \ot X_{n-1}} \ot \id_{X_{n}}) \\
&\qquad\circ 
(\theta_{X_1}^{-1} \ot \dots \ot \theta_{X_{n-1}}^{-1} \ot \id_{X_{n}}) \\
&\quad= c_{X_n, X_1 \ot \dots \ot X_{n-1}} \circ c_{X_1 \ot \dots \ot X_{n-1}, X_{n}} \circ 
(\theta_{X_1 \ot \dots \ot X_{n-1}} \ot \theta_{X_n}) \circ 
(\theta_{X_1}^{-1} \ot \dots \ot \theta_{X_n}^{-1}) \\
&\quad= \theta_{X_1 \ot \dots \ot X_{n}} \circ (\theta_{X_1}^{-1} \ot \dots \ot \theta_{X_n}^{-1})
\end{align*}
again by the defining property of a ribbon structure.

\item
A ribbon structure for the braiding~$c^{-1}_{Y,X}$ considered in the proof of the second assertion is given by the morphisms~$\theta_{X}^{-1} \colon X \to X$. If we apply the third assertion to this braiding and this ribbon structure instead and then take inverses, we obtain the fourth assertion. 
\qedhere
\end{parlist}
\end{proof}

Another lemma that we will need concerns precomposition with the twist:
\begin{Lemma} \label{TwistLem}
For all $f \in \Hom_\CC(X \ot Y, \1)$, we have
$f \circ (\theta_X \ot \id_Y) = f \circ (\id_X \ot \, \theta_Y)$.
More generally, precomposition with $\theta_X \ot \id_Y$ and precomposition with 
$\id_X \ot \, \theta_Y$ induce the same homomorphism on $\Ext^m(X \ot Y, \1)$.
\end{Lemma}
\begin{proof}
\begin{parlist}
\item
As we recorded in Paragraph~\ref{Fin}, the functor~$\textbf{--} \ot Y^*$ is the right adjoint of the functor~$\textbf{--} \ot Y$. The adjunction is given by
\[\eta\colon \Hom_\CC(X \ot Y, \1) \to \Hom_\CC(X, Y^*),~f \mapsto (f \ot \id_{Y^*}) \circ (\id_X \ot \coev_{Y}).\]
Therefore, the diagram
\vspace{-1.5mm}
\begin{center}
\begin{tikzcd}[column sep= large]
\Hom_\CC(X \ot Y, \1) \arrow[r, "\circ (\theta_X \ot \id_Y)"] \arrow[d, "\eta"'] & \Hom_\CC(X \ot Y, \1) \arrow[d, "\eta"] \\
\Hom_\CC(X, Y^*) \arrow[r, "\circ \theta_{X}"'] & \Hom_\CC(X, Y^*)
\end{tikzcd}
\end{center}
commutes. On the other hand, also the diagram
\begin{center}
\begin{tikzcd}[column sep= large]
\Hom_\CC(X \ot Y, \1) \arrow[r, "\circ (\id_X \ot \, \theta_Y)"] \arrow[d, "\eta"'] & \Hom_\CC(X \ot Y, \1) \arrow[d, "\eta"] \\
\Hom_\CC(X, Y^*) \arrow[r, "(\theta_{Y^*}) \circ"'] & \Hom_\CC(X, Y^*)
\end{tikzcd}
\end{center}
\vspace{-1.5mm}
commutes, because
\begin{align*}
\eta(f \circ (\id_X \ot \, \theta_Y)) 
&= (f \ot \id_{Y^*}) \circ (\id_X \ot \, \theta_Y \ot \id_{Y^*}) \circ (\id_X \ot \coev_{Y}) \\
&= (f \ot \id_{Y^*}) \circ (\id_X \ot \id_Y \ot \, \theta_{Y}^*) \circ (\id_X \ot \coev_{Y}) \\
&= (f \ot \, \theta_{Y}^*) \circ (\id_X \ot \coev_{Y}) 
= \theta_{Y}^* \circ (f \ot \id_{Y^*}) \circ (\id_X \ot \coev_{Y}) \\
&= \theta_{Y}^* \circ \eta(f) = \theta_{Y^*} \circ \eta(f).
\end{align*}
But by the naturality of the twist, we have $\theta_{Y^*} \circ g = g \circ \theta_X $ for 
$g \in \Hom_\CC(X, Y^*)$, so that the bottom rows in the two diagrams are equal. By comparing the two diagrams, we obtain the first assertion. 

\enlargethispage{2.8mm}

\item
From our projective resolution 
\[\1 \xlongleftarrow{} P_0 \xlongleftarrow{} P_1 \xlongleftarrow{} P_2 \xlongleftarrow{} \cdots\]
of the unit object, we get a projective resolution of~$X$ by tensoring with~$X$, i.e., by defining 
$Q_m \deq P_m\ot X$, and even a projective resolution of~$X \ot Y$ by tensoring again with~$Y$. Then
\begin{center}
\begin{tikzcd} {}
X \ot Y \arrow{d}{\id_X \ot \, \theta_Y} & Q_0\ot Y \arrow{l}[swap]{}\arrow{d}{\id_{Q_0} \ot \, \theta_{Y}} &
Q_1\ot Y \arrow{l}[swap]{} \arrow{d}{\id_{Q_1} \ot \, \theta_{Y}} & \cdots \arrow{l}[swap]{} \\
X \ot Y & Q_0\ot Y \arrow{l}[swap]{} & Q_1\ot Y  \arrow{l}[swap]{} & \cdots \arrow{l}[swap]{}
\end{tikzcd}
\vspace{1mm}
\end{center}
is a lift of $\id_X \ot \, \theta_Y$ to this projective resolution. On the other hand, a lift of 
$\theta_X \ot \id_Y$ is given by 
\begin{center}
\begin{tikzcd} {}
X \ot Y \arrow{d}{\theta_X \ot \id_Y} & Q_0\ot Y \arrow{l}[swap]{}\arrow{d}{\theta_{Q_0} \ot \id_{Y}} &
Q_1\ot Y \arrow{l}[swap]{} \arrow{d}{\theta_{Q_1} \ot \id_{Y}} & \cdots \arrow{l}[swap]{} \\
X \ot Y & Q_0\ot Y \arrow{l}[swap]{} & Q_1\ot Y  \arrow{l}[swap]{} & \cdots \arrow{l}[swap]{}
\end{tikzcd}
\vspace{1mm}
\end{center}
So the action of $\id_X \ot \, \theta_Y$ on $\Ext^m(X \ot Y, \1)$ is induced by precomposition with
\mbox{$\id_{Q_m} \ot \, \theta_{Y}$} on $\Hom_\CC(Q_m\ot Y, \1)$, while the action of $\theta_X \ot \id_Y$ is induced by precomposition with $\theta_{Q_m} \ot \id_{Y}$ on this space. But by the preceding step, these two precomposition maps are equal.
\qedhere
\end{parlist}
\end{proof}

The results of Paragraph~\ref{MappClassExt} yield in the case $g=0$ that the projective action of $\Gamma_{0,n+1}(n+1)$ on the direct sum
\[\bigoplus_{\tau \in S_n} \Ext^m(X_{\tau(n)} \ot \dots \ot X_{\tau(1)}, \1)\]
descends to a projective action of $\Gamma_{0,n}$.
In fact, the action is in this case not only projective, but rather an ordinary linear action. As we will explain now, we can see this explicitly from the presentations of these groups that we gave in Paragraph~\ref{Sphere}. According to Proposition~\ref{PropSphere}, we have
\[\Gamma_{0,n+1}(n+1) \cong \Z^{n} \rtimes B_{n}.\]
It is not difficult to check that the defining relations of $\Z^{n} \rtimes B_{n}$ discussed there are strictly, not only projectively, satisfied on the direct sum 
\[\bigoplus_{\tau \in S_n} \Hom_\CC(P_m \ot X_{\tau(n)} \ot \dots \ot X_{\tau(1)}, \1).\]
To show that the action descends to an action of $\Gamma_{0,n}$, we have to show that the additional relations
\[\fb_{1,2} \fb_{2,3} \cdots \fb_{n-1,n}^2 \cdots \fb_{2,3} \fb_{1,2} = \fd_1^{-2}\]
and $(\fb_{1,2} \fb_{2,3} \cdots \fb_{n-1,n})^n = \fd_1^{-1} \fd_2^{-1} \cdots \fd_n^{-1}$ hold on this space. It is sufficient to verify the relations on a single summand of the direct sum:
\begin{Proposition} \label{SphereProp}
If $f \in \Hom_\CC(P_m \ot X_n \ot \dots \ot X_1, \1)$ represents the cohomology class~$\bar{f}$, we have 
\[\TFT^m(\fb_{1,2} \fb_{2,3} \cdots \fb_{n-1,n}^2 \cdots \fb_{2,3} \fb_{1,2})(\bar{f}) 
= \TFT^m(\fd_1^{-2})(\bar{f}) \]
as well as
\[\TFT^m((\fb_{1,2} \fb_{2,3} \cdots \fb_{n-1,n})^n)(\bar{f}) = 
\TFT^m(\fd_1^{-1} \fd_2^{-1} \cdots \fd_n^{-1})(\bar{f}). \]
\end{Proposition}
\begin{proof}
\begin{parlist}
\item 
It follows from the second assertion in Lemma~\ref{LemCaseSphere} that
\begin{align*}
&\TFT(\fb_{n-1,n} \cdots \fb_{2,3} \fb_{1,2})(f) \\
&\quad = f \circ (\id_{P_m \ot X_n \ot \dots \ot X_3} \ot \, c_{X_1,X_2}) 
\circ (\id_{P_m \ot X_n \ot \dots \ot X_4} \ot \, c_{X_1,X_3} \ot \id_{X_2}) \\
&\qquad \circ \dots \circ (\id_{P_m} \ot \, c_{X_1,X_n} \ot \id_{X_{n-1} \ot \dots \ot X_2})\\
&\quad = f \circ (\id_{P_m} \ot \, c_{X_1, X_n \ot \dots \ot X_2})
\end{align*}
and then from the first assertion there that
\begin{align*}
&\TFT(\fb_{1,2} \fb_{2,3} \cdots \fb_{n-1,n}^2 \cdots \fb_{2,3} \fb_{1,2})(f) \\
&\quad = f \circ (\id_{P_m} \ot \, c_{X_1, X_n \ot \dots \ot X_2}) 
\circ  (\id_{P_m} \ot \, c_{X_n, X_1} \ot \id_{X_{n-1} \ot \dots \ot X_2}) \\
&\qquad \circ \dots \circ (\id_{P_m \ot X_n \ot \dots \ot X_4} \ot \, c_{X_3, X_1} \ot \id_{X_2}) \circ (\id_{P_m \ot X_n \ot \dots \ot X_3} \ot \, c_{X_2, X_1}) \\
&\quad = f \circ (\id_{P_m} \ot \, c_{X_1, X_n \ot \dots \ot X_2}) \circ 
(\id_{P_m}\ot \, c_{X_n\ot\dots\ot X_2,X_1}) \\
&\quad =f \circ (\id_{P_m}\ot \, \theta_{X_n\ot \dots \ot X_1}) \circ 
(\id_{P_m}\ot \, \theta_{X_n\ot\dots\ot X_2}^{-1} \ot \, \theta_{X_1}^{-1}),
\end{align*}
while
$\TFT(\fd_1^{-2})(f) = f \circ 
(\id_{P_m} \ot \id_{X_n \ot \dots \ot X_2} \ot \, \theta_{X_1}^{-2})$. So, if we introduce the abbreviations $X \deq X_n \ot \dots \ot X_2$ and $Y \deq X_1$, we have to show that the two cochain maps
\[f \mapsto f \circ (\id_{P_m}\ot \, \theta_{X \ot Y}) \circ 
(\id_{P_m}\ot \, \theta_{X}^{-1} \ot \, \theta_{Y}^{-1} )\]
and 
$f \mapsto f \circ (\id_{P_m}\ot\id_X\ot \, \theta_{Y} ^{-2})$
of the cochain complex $(\Hom_\CC(P_m\ot X \ot Y, \1))$
induce the same homomorphism in cohomology. 

\vspace{2mm}

By the first assertion in Lemma~\ref{TwistLem}, the cochain map
$f \mapsto f \circ (\id_{P_m} \ot \, \theta_{X \ot Y} )$
is equal to the cochain map $f \mapsto f \circ (\theta_{P_m}\ot\id_{X \ot Y} )$, 
and we saw in Lemma~\ref{LemDehnTw1} that this cochain map is homotopic to the identity. This shows that our assertion will follow if we can show that the two cochain maps
$f \mapsto f \circ (\id_{P_m}\ot\theta_{X}^{-1} \ot \id_{Y} )$ and
$f \mapsto f \circ (\id_{P_m}\ot\id_{X} \ot \, \theta_{Y}^{-1} )$
induce the same map in cohomology. But this follows directly from the second assertion in Lemma~\ref{TwistLem} by taking inverses.

\item
To prove the second relation, we note that the first assertion in Lemma~\ref{LemCaseSphere} gives
\begin{align*}
&\TFT(\fb_{1,2} \fb_{2,3} \cdots \fb_{n-1,n})(f)\\
&\quad = f \circ (\id_{P_m} \ot \, c_{X_{n-1}, X_{n}} \ot \id_{X_{n-2} \ot \dots \ot X_1}) 
\circ \cdots  \\
&\qquad \circ  
(\id_{P_m} \ot \id_{X_{n-1} \ot \dots \ot X_3} \ot \, c_{X_2, X_{n}} \ot \id_{X_{1}})
\circ (\id_{P_m} \ot \id_{X_{n-1} \ot \dots \ot X_2} \ot \, c_{X_1, X_{n}})  \\
&\quad =f \circ (\id_{P_m}\ot \, c_{X_{n-1} \ot \dots \ot X_{1},X_n} ) 
\end{align*}
and therefore the fourth assertion of this lemma yields
\begin{align*}
&\TFT((\fb_{1,2} \fb_{2,3} \cdots \fb_{n-1,n})^n)(f) \\
&\quad = f \circ (\id_{P_m}\ot \, c_{X_{n-1} \ot \dots \ot X_{1},X_n})
\circ (\id_{P_m}\ot \, c_{X_{n-2} \ot \dots \ot X_{1}\ot X_n,X_{n-1}}) \\
&\qquad \circ \cdots \circ (\id_{P_m}\ot \, c_{X_n \ot \dots \ot X_{2},X_1} ) \\
&\quad = f \circ (\id_{P_m}\ot \, \theta_{X_n \ot \dots \ot X_{1}} ) \circ 
(\id_{P_m}\ot \, \theta_{X_n}^{-1} \ot \dots \ot \, \theta_{X_1}^{-1}).
\end{align*}
As in the case of the first relation, the cochain map 
$f \mapsto f \circ (\id_{P_m}\ot\theta_{X_n \ot \dots \ot X_{1}})$
is equal to the cochain map
$f \mapsto f \circ (\theta_{P_m}\ot\id_{X_n \ot \dots \ot X_{1}})$
by the first assertion in Lemma~\ref{TwistLem}, which is homotopic to the identity by Lemma~\ref{LemDehnTw1}. Therefore, our cochain map is homotopic to the cochain map
\begin{align*}
&\TFT(\fd_1^{-1} \fd_2^{-1} \cdots \fd_n^{-1})(f) = 
f \circ (\id_{P_m}\ot\theta_{X_n}^{-1} \ot \dots \ot \, \theta_{X_1}^{-1} )
\end{align*}
as required.
\qedhere
\end{parlist}
\end{proof}

\subsection{Hochschild cohomology} \label{Hochsch}
In order to explain why Corollary~\ref{CorEqual} generalizes the main result of our previous article (cf.~\cite{LMSS1}), we now specialize the situation to the case where~$\CC$ is the category of left modules over the factorizable ribbon Hopf algebra~$A$ described in Paragraph~\ref{Fact}, and furthermore to the case where~$g=1$ and~$n=0$. From Paragraph~\ref{Torus}, we know 
that~$\Gamma_{1,0} \cong \SL(2,\Z)$ and~$\Gamma_{1,1} \cong B_3$. From Corollary~\ref{CorEqual}, we then obtain a projective action of~$\Gamma_{1,0} \cong \SL(2,\Z)$ on~$\Ext^m_A(\K,L)$, which arises as a quotient of the projective action of~$\Gamma_{1,1} \cong B_3$ on~$\Hom_A(P_m,L)$ for a projective resolution
\[\K \xlongleftarrow{} P_0 \xlongleftarrow{} P_1 \xlongleftarrow{} P_2 \xlongleftarrow{} \cdots\]
of the base field~$\K$, considered as a trivial \mbox{$A$-module}, as explained in Paragraph~\ref{Fact}. That the projective action of~$\Gamma_{1,1}$ indeed descends to a projective action of~$\Gamma_{1,0}$ on the cohomology groups can be seen quite explicitly in this case: According to Paragraph~\ref{MapClRep}, the generators~$\fs_1$ and~$\ft_1$ of~$\Gamma_{1,1}$ act by postcomposition with~$\fS$ and~$\fT$, respectively, and we have seen in Paragraph~\ref{ActTorus} that they satisfy the identity 
\mbox{$\fS \circ \fT \circ \fS = \rho(v) \; \fT^{-1} \circ \fS \circ \fT^{-1}$}, which means that they satisfy the defining identity of~$\Gamma_{1,1} \cong B_3$ projectively. We have also seen there that the 2-chain relation is satisfied projectively, i.e., that the actions of $\fs_1^{4}$ and~$\fd_1^{-1}$ agree up to a scalar. It therefore follows from Lemma~\ref{LemDehnTw1} that~$\fs_1^{4}$ acts on the cohomology groups as a scalar multiple of the identity, so that the defining relations of the modular group are satisfied projectively.

In order to relate this result to Hochschild cohomology, we choose our projective resolution of the base field in a special way. We first choose a projective resolution
\[A \xlongleftarrow{} Q_0 \xlongleftarrow{} Q_1 \xlongleftarrow{} Q_2 \xlongleftarrow{} \cdots\]
of the algebra~$A$ in the category of left $A \ot A^{\op}$-modules, or equivalently the category of $A$-bimodules, where we require, as in~\cite[Chap.~IX, \S~3, p.~167]{CE}, that the left and the right action of~$A$ on a bimodule become equal when restricted to~$\K$. By \cite[Chap.~X, Thm.~2.1, p.~185]{CE}, we can then set $P_m \deq Q_m \ot_A \K$ to obtain a projective resolution of~$A \ot_A \K \cong \K$.

Now the left $A$-module~$L$ can be considered as an $A$-bimodule via the trivial right $A$-action, i.e., the action $\varphi.a \deq \varepsilon(a) \varphi$ for $\varphi \in L = A^*$. We denote~$L$ by~$L_\varepsilon$ if it is considered as an $A$-bimodule in this way. With the help of the cochain map
\[\Hom_{A \ot A^{\op}}(Q_m, L_\varepsilon) \to \Hom_A(Q_m \ot_A \K, L),~f \mapsto
(q \ot \lambda\mapsto \lambda f(q)) \]
with inverse
\[\Hom_A(Q_m \ot_A \K, L) \to \Hom_{A \ot A^{\op}}(Q_m, L_\varepsilon),~g \mapsto
(q \mapsto g(q \ot 1_\K)) \]
we see that the cochain complexes $\Hom_A(Q_m \ot_A \K, L)$ and 
$\Hom_{A \ot A^{\op}}(Q_m, L_\varepsilon)$ are isomorphic, so that also their cohomology groups
$\Ext_A^m(\K, L)$ and $\Ext^m_{A \ot A^{\op}}(A, L_\varepsilon)$ are isomorphic. But the latter cohomology groups are, by definition, the Hochschild cohomology groups~$HH^m(A, L_\varepsilon)$.

Via this isomorphism of cochain complexes, we can transfer the action of~$\Gamma_{1,1}$ to the Hochschild cochain groups of the bimodule~$L_\varepsilon$. The generators~$\fs_1$ and~$\ft_1$ of~$\Gamma_{1,1}$ then clearly still act by postcomposition with the bimodule homomorphisms~$\fS$ and~$\fT$, while the Dehn twist~$\fd_1$ acts by precomposition with the endomorphism~$q \mapsto v.q$ of~$Q_m$, or equivalently by postcomposition with the endomorphism~$\varphi \mapsto v.\varphi$ of~$L_\varepsilon$. From this correspondence, we see that our Lemma~\ref{LemDehnTw1} can be considered as an analogue of Proposition~1.2 in our previous treatment~\cite{LMSS1}.

While the two results are analogous, it is not yet clear that the projective representations of the modular group constructed here and in~\cite[Cor.~5.6, p.~419]{LMSS1} are isomorphic; in fact, the action constructed above is on the space~$HH^m(A, L_\varepsilon)$, while the action constructed in~\cite{LMSS1} is on the space~$HH^m(A, A)$. However, it turns out that the two actions are indeed isomorphic:

\begin{Proposition} 
The projective actions of~$\SL(2,\Z)$ on the spaces $HH^m(A, L_\varepsilon)$ and~$HH^m(A, A)$ are isomorphic.
\end{Proposition}
\begin{proof}
\begin{parlist}
\item 
In general, for an $A$-bimodule~$N$, we can modify the bimodule structure by defining the new left action as
\[a.n \deq n.S^{-1}(a)\]
and similarly the new right action as $n.a \deq S^{-1}(a).n$. We denote~$N$ by~$N'$ if we consider it endowed with this bimodule structure. From the point of view of bimodules as $A \ot A^{\op}$-modules, this operation is the pullback along the ring homomorphism
\[A \ot A^{\op} \to A \ot A^{\op},~a \ot b \mapsto S^{-1}(b) \ot S^{-1}(a).\]

\item
Now 
\[A' \xlongleftarrow{} Q'_0 \xlongleftarrow{} Q'_1 \xlongleftarrow{} Q'_2 \xlongleftarrow{} \cdots\]
is a projective resolution of~$A'$, and the cochain complex~$\Hom_{A \ot A^{\op}}(Q'_m, N')$ is not only isomorphic to the cochain complex~$\Hom_{A \ot A^{\op}}(Q_m, N)$, but even set-theoretically equal. We therefore have that $\Ext^m_{A \ot A^{\op}}(A', N') = \Ext^m_{A \ot A^{\op}}(A, N)$.

\item
The antipode yields a bimodule isomorphism $S\colon A' \to A$, which by the comparison theorem lifts to a chain map
\begin{center}
\begin{tikzcd} {}
A' \arrow{d}{S} & Q'_0 \arrow{l} \arrow{d}{S_{0}} & Q'_1 \arrow{l} \arrow{d}{S_{1}} & \cdots \arrow{l}\\
A & Q_0\arrow{l} & Q_1 \arrow{l} & \cdots \arrow{l}
\end{tikzcd}
\end{center}
By pulling back along the chain map~$(S_m)$, we obtain an isomorphism
\[\Ext^m_{A \ot A^{\op}}(A, N') \to 
\Ext^m_{A \ot A^{\op}}(A', N') = \Ext^m_{A \ot A^{\op}}(A, N).\]
In other words, we obtain an isomorphism between~$HH^m(A, N')$ and~$HH^m(A, N)$.

\item
We know from~\cite[Lem.~5.2, p.~417]{LMSS1} and the subsequent discussion that the projective action of~$\SL(2,\Z)$ on~$HH^m(A, A)$ is isomorphic to the one on the cohomology groups 
$HH^m(A, \prescript{}{\varepsilon} A_{\ad})$, where the generators~$\fs$ and~$\ft$ act by postcomposition with~$\hS$ and~$\hT$, respectively. Here, $\prescript{}{\varepsilon} A_{\ad}$ is the $A$-bimodule with~$A$ as underlying vector space, trivial left action $a.a' \deq \varepsilon(a) a'$, and right action
$\ad(a \ot a') = S(a'_{(1)}) a a'_{(2)}$ (cf.~\cite[Sec.~2, p.~406]{LMSS1}). We note that, as in the case of~$L$ discussed above, $\prescript{}{\varepsilon} A_{\ad}$ itself is not an $\SL(2,\Z)$-module, but only a $\Gamma_{1,1}$-module; however, the induced projective $\Gamma_{1,1}$-action on 
$HH^m(A, \prescript{}{\varepsilon} A_{\ad})$ descends to a projective $\SL(2,\Z)$-action.

If we set~$N = \prescript{}{\varepsilon} A_{\ad}$, we find that~$N'$ is the bimodule with underlying vector space~$A$, left action given by $a.a' = a_{(2)} a' S^{-1}(a_{(1)})$, and trivial right action. From the way how we constructed the isomorphism between~$HH^m(A, N')$ 
and~$HH^m(A, N)$ in the preceding step, we see that it becomes equivariant if we let the generators~$\fs$ and~$\ft$ also act on~$HH^m(A, N')$ by postcomposition with~$\hS$ and~$\hT$, respectively. But now Lemma~\ref{LemEquiv} shows that~$\bar{\iota}\colon N' \to L_{\varepsilon}$ is a bimodule isomorphism that is also equivariant under the projective $\Gamma_{1,1}$-action. Therefore postcomposition with~$\bar{\iota}$ yields an isomorphism between~$HH^m(A, N')$ and~$HH^m(A, L_{\varepsilon})$ that is $\SL(2,\Z)$-equivariant. Combining this with our isomorphism between~$HH^m(A, \prescript{}{\varepsilon} A_{\ad})$  and~$HH^m(A, N')$ already obtained, our assertion follows.
\qedhere
\end{parlist}
\end{proof}

In view of this proposition, we can indeed say that the mapping class group actions obtained here in Corollary~\ref{CorEqual} generalize the projective $\SL(2,\Z)$-representation on the Hochschild cohomology groups~$HH^m(A, A)$ obtained in~\cite[Cor.~5.6, p.~419]{LMSS1}.

\end{document}